\documentclass[12pt,twoside]{article}
\renewcommand{\baselinestretch}{1.1}%
\usepackage{times}
\usepackage{bm}
\usepackage{graphics}
\usepackage{theorem}
\usepackage[dvipdfmx]{graphicx}
\usepackage{amsmath}
\usepackage{latexsym}
\usepackage{amssymb,mathrsfs}
\usepackage{natbib}
\usepackage{pdfpages}
\usepackage[flushmargin]{footmisc}
\renewcommand{\thefootnote}{\fnsymbol{footnote}}
\theorembodyfont{\itshape}
\newtheorem{thm}{Theorem}[section]
\newtheorem{lem}{Lemma}[section]
\newtheorem{prop}{Proposition}[section]
\newtheorem{coro}{Corollary}[section]

\newtheorem{lemnonum}{Lemma}
\renewcommand{\thelemnonum}{}

\newtheorem{cond}{Condition}[section]{\bf}{\rm}

\theorembodyfont{\rmfamily}
{\bf}{\rm}
\newtheorem{assumpt}{Assumption}[section]{\bf}{\rm}

\newtheorem{rem}{Remark}[section]{\itshape}{\rmfamily}
\newtheorem{com}{Comment}[section]{\itshape}{\rmfamily}

\newenvironment{proof}{\noindent{\it Proof.~~}}{\qed\medskip}

\makeatletter
\def\eqnarray{\stepcounter{equation}\let\@currentlabel=\theequation
\global\@eqnswtrue
\global\@eqcnt\z@\tabskip\@centering\let\\=\@eqncr
$$\halign to \displaywidth\bgroup\@eqnsel\hskip\@centering
  $\displaystyle\tabskip\z@{##}$&\global\@eqcnt\@ne 
  \hfil$\;{##}\;$\hfil
  &\global\@eqcnt\tw@ $\displaystyle\tabskip\z@{##}$\hfil 
   \tabskip\@centering&\llap{##}\tabskip\z@\cr}
\makeatother
\makeatletter
    \renewcommand{\theequation}{%
    \thesection.\arabic{equation}}
    \@addtoreset{equation}{section}
  \makeatother
\def\narrow{\list{}{}\item[]}
\let\endnarrow=\endlist
\newcommand{\dm}{\displaystyle}
\newcommand{\lleft}{\!\!\left}

\newcommand{\qed}{\hspace*{\fill}$\Box$\rule[-10pt]{0pt}{10pt}}
\newcommand{\QED}{\hspace*{\fill}$\Box$}
\def\red#1{{\color{red}{#1}}}

\newcommand{\vc}{\bm}
\def\svc#1{\mbox{\boldmath $\scriptstyle #1$}}
\def\ssvc#1{\mbox{\boldmath $\scriptscriptstyle #1$}}
\def\wwtilde#1{\,\widetilde{\!\widetilde{#1}}{}}
\newcommand{\wt}{\widetilde}
\newcommand{\wh}{\widehat}
\newcommand{\ol}{\overline}

\def\trunc#1{{}_{(n)}#1}
\def\trunctilde#1{{}_{(n)}\tilde{#1}}
\def\presub#1{\hspace{0.05em}{}_{(#1)}\hspace{-0.05em}}
\newcommand{\sfBI}{\mathsf{BI}}
\newcommand{\sfBM}{\mathsf{BM}}
\newcommand{\vmax}{\vee}
\newcommand{\vmin}{\wedge}
\newcommand{\EE}{\mathsf{E}}
\newcommand{\PP}{\mathsf{P}}
\newcommand{\II}{\mathit{I}}
\newcommand{\calA}{\mathcal{A}}
\newcommand{\calB}{\mathcal{B}}
\newcommand{\calC}{\mathcal{C}}
\newcommand{\calD}{\mathcal{D}}
\newcommand{\calE}{\mathcal{E}}
\newcommand{\calF}{\mathcal{F}}
\newcommand{\calG}{\mathcal{G}}
\newcommand{\calH}{\mathcal{H}}
\newcommand{\calI}{\mathcal{I}}
\newcommand{\calL}{\mathcal{L}}
\newcommand{\calO}{\mathcal{O}}
\newcommand{\calR}{\mathcal{R}}
\newcommand{\calS}{\mathcal{S}}
\newcommand{\calOL}{\mathcal{OL}}
\newcommand{\calOS}{\mathcal{OS}}
\newcommand{\SC}{\mathcal{SC}}
\newcommand{\bbA}{\mathbb{A}}
\newcommand{\bbC}{\mathbb{C}}
\newcommand{\bbD}{\mathbb{D}}
\newcommand{\bbF}{\mathbb{F}}
\newcommand{\bbG}{\mathbb{G}}
\newcommand{\bbI}{\mathbb{I}}
\newcommand{\bbK}{\mathbb{K}}
\newcommand{\bbL}{\mathbb{L}}
\newcommand{\bbM}{\mathbb{M}}
\newcommand{\bbN}{\mathbb{N}}
\newcommand{\bbR}{\mathbb{R}}
\newcommand{\bbS}{\mathbb{S}}
\newcommand{\bbZ}{\mathbb{Z}}
\newcommand{\adj}{\mathrm{adj}}
\newcommand{\diag}{\mathrm{diag}}
\newcommand{\sgn}{\mathrm{sgn}}
\newcommand{\trace}{\mathrm{trace}}
\newcommand{\Var}{\mathsf{Var}}
\newcommand{\Cov}{\mathsf{Cov}}
\newcommand{\CV}{\mathrm{C_V}}
\newcommand{\Mod}{\mathrm{mod}}
\newcommand{\rmt}{{\rm t}}
\newcommand{\rmd}{{\rm d}}
\newcommand{\rme}{{\rm e}}
\newcommand{\rmT}{{\rm T}}
\newcommand{\resp}{{\rm resp}}
\newcommand{\scrE}{\mathscr{E}}
\renewcommand{\labelenumi}{(\roman{enumi})}
\newcommand{\dd}[1]{\if#11 1\!\!1 
\else {\if#1C I\!\!\!C
\else {\if#1G I\!\!\!G 
\else {\if#1J J\!\!\!J 
\else {\if#1S S\!\!\!S
\else {\if#1Z Z\!\!\!Z
\else {\if#1Q O\!\!\!\!Q
\else I\!\!#1
\fi}
\fi}
\fi}
\fi} 
\fi} 
\fi} 
\fi} 

\makeatother

\begin{document}\thispagestyle{plain} 

\hfill

{\Large{\bf
\begin{center}
Error bounds for last-column-block-augmented
truncations of block-structured Markov chains%
\footnote[0]{
This manuscript is the revised version of the published paper ``Journal of the Operations Research Society of Japan, vol. 60, no. 3, pp. 271--320, 2017.'' This revised version includes Comment 2.1 (related to Lemma~2.1) and the corrigendum to the original version. In addition, the revised version corrects minor errors related to the domain of ``$\sup$'' in several error bounds. These corrections are marked in red.
}
%
%
\end{center}
}
}

\begin{center}
{
Hiroyuki Masuyama%
\footnote[2]{E-mail: masuyama@tmu.ac.jp}
}

\medskip

{\small
Graduate School of Management, Tokyo Metropolitan University\\
Tokyo 192--0397, Japan
}

\bigskip
\medskip

{\small
\textbf{Abstract}

\medskip

\begin{tabular}{p{0.85\textwidth}}
This paper discusses the error estimation of the
last-column-block-augmented northwest-corner truncation
(LC-block-augmented truncation, for short) of block-structured Markov
chains (BSMCs) in continuous time. We first derive upper bounds for
the absolute difference between the time-averaged functionals of a
BSMC and its LC-block-augmented truncation, under the assumption that
the BSMC satisfies the general $\vc{f}$-modulated drift condition. We
then establish computable bounds for a special case where the BSMC is
exponentially ergodic. To derive such computable bounds for the
general case, we propose a method that reduces BSMCs to be
exponentially ergodic. We also apply the obtained bounds to
level-dependent quasi-birth-and-death processes (LD-QBDs), and discuss
the properties of the bounds through the numerical results on an
M/M/$s$ retrial queue, which is a representative example of
LD-QBDs. Finally, we present computable perturbation bounds for the
stationary distribution vectors of BSMCs.
\end{tabular}
}
\end{center}

\begin{center}
\begin{tabular}{p{0.90\textwidth}}
{\small

\medskip

}
\end{tabular}

\end{center}

\section{Introduction}\label{introduction}
Let $\{(X(t),J(t));t\ge0\}$ denote a continuous-time regular-jump
Markov chain with state space $\bbF:= \cup_{k\in\bbZ_+}\{k\} \times
\bbS_k$ (see, e.g., \citet[Chapter~8, Definition~2.5]{Brem99}), where
\begin{eqnarray*}
\bbS_k = \{0,1,\dots,S_k\} \subset \bbZ_+, 
\qquad
\bbZ_+ = \{0\} \cup \bbN,\qquad
\bbN =\{1,2,3,\dots\}.
\end{eqnarray*}
Let $\vc{P}^{(t)}=(p^{(t)}(k,i;\ell,j))_{(k,i;\ell,j)\in\bbF^2}$
denote the transition matrix function of $\{(X(t),J(t))\}$, i.e.,
\[
p^{(t)}(k,i;\ell,j) =
\PP(X(t)=\ell,J(t)=j \mid X(0)=k,J(0)=i),
\qquad t \ge 0,\ (k,i;\ell,j) \in \bbF^2,
\]
where $(k,i;\ell,j)$ denotes ordered pair $((k,i), (\ell,j))$. Since
$\{(X(t),J(t))\}$ is a regular-jump Markov chain, the transition
matrix function $\vc{P}^{(t)}$ is continuous, which implies that the
infinitesimal generator of $\{(X(t),J(t))\}$ is well-defined (see,
e.g., \citet[Chapter~8, Theorems 2.1 and 3.4]{Brem99}). Thus, we
define $\vc{Q}:=(q(k,i;\ell,j))_{(k,i;\ell,j)\in\bbF^2}$ as the
infinitesimal generator of $\{(X(t),J(t))\}$, i.e.,
\[
\vc{Q} = \lim_{t \downarrow 0} \frac{\vc{P}^{(t)} - \vc{I}}{t}, 
\]
where $\vc{I}$ denotes the identity matrix with an appropriate order
according to the context.

It should be noted (see, e.g., \citet[Chapter~8, Definition 2.4 and
  Theorem~2.2]{Brem99}) that the
infinitesimal generator $\vc{Q}$ of the regular-jump Markov chain $\{(X(t),J(t))\}$ is stable and conservative,
i.e.,
\begin{align}
&&
\sum_{(\ell,j) \in \bbF \setminus\{(k,i)\}}q(k,i;\ell,j) 
=  -q(k,i;k,i) &< \infty, & (k,i) &\in \bbF,&&
\nonumber
\\
&&
0 \le q(k,i;\ell,j) &< \infty, 
& (k,i;\ell,j) &\in \bbF^2,\ (k,i) \neq (\ell,j).&&
\nonumber
\end{align}
Note also that $\vc{Q}$ and its principal submatrices (obtained by
deleting a set of rows and columns with the same indices; e.g., the
northwest-corner truncation $\vc{Q}_{\bbF_n}$ in (\ref{defn-Q_F_n})
below) belong to the set of {\it q-matrices}, i.e., diagonally
dominant matrices with nonpositive diagonal and nonnegative
off-diagonal elements (see, e.g., \citet[Section~2.1]{Ande91}). In
some cases, we refer to the $q$-matrix as the infinitesimal generator,
especially when it is connected with a specific Markov chain.  As with
the infinitesimal generator, any $q$-matrix is called {\it stable} if
its diagonal elements are all finite; and called {\it conservative} if
its row sums are all equal to zero.

We now assume that $\vc{Q}$ has the following block-structured form:
\begin{equation}
\vc{Q} =
\bordermatrix{
               & \bbL_0 &  \bbL_1  &  \bbL_2       &  \bbL_3       & \cdots       
\cr
\bbL_0 & 
\vc{Q}(0;0) & 
\vc{Q}(0;1) &
\vc{Q}(0;2) &
\vc{Q}(0;3) &
\cdots
\cr
\bbL_1 & 
\vc{Q}(1;0) & 
\vc{Q}(1;1) &
\vc{Q}(1;2) &
\vc{Q}(1;3) &
\cdots
\cr
\bbL_2 & 
\vc{Q}(2;0) & 
\vc{Q}(2;1) &
\vc{Q}(2;2) &
\vc{Q}(2;3) &
\cdots
\cr
\bbL_3 & 
\vc{Q}(3;0) & 
\vc{Q}(3;1) &
\vc{Q}(3;2) &
\vc{Q}(3;3) &
\ddots
\cr
~\vdots    
& \vdots					
&
\vdots					&
\vdots					&
\ddots					&
\ddots
},
\label{defn-Q}
\end{equation}
where $\bbL_k = \{k\} \times \bbS_k \subset \bbF$ for $k \in \bbZ_+$,
which is called {\it level $k$}. Markov chains with block-structured
infinitesimal generators like $\vc{Q}$ in (\ref{defn-Q}) are called {\it
  block-structured Markov chains (BSMCs)}. Typical examples of BSMCs
are in block-Toeplitz-like and/or block-Hessenberg forms (including
block-tridiagonal form), such as {\it level-independent}
GI/G/1-type Markov chains (see, e.g., \citet{Gras90,Neut89}); {\it
  level-dependent} quasi-birth-and-death processes (LD-QBDs) (see,
e.g., \citet[Chapter 12]{Lato99}); and {\it
  level-dependent} M/G/1- and GI/M/1-type Markov chains (see, e.g., \citet{Masu16,Masu05}).

Throughout the paper, we assume that the BSMC $\{(X(t),J(t))\}$ is
ergodic, i.e., irreducible and positive recurrent. It then follows
that the BSMC $\{(X(t),J(t))\}$ has the unique stationary distribution
vector (called {\it stationary distribution} or {\it stationary
  probability vector}), denoted by
$\vc{\pi}:=(\pi(\ell,j))_{(\ell,j)\in\bbF}$ (see, e.g., \citet[Section
  5.4, Theorem 4.5]{Ande91}).  By definition,
\[
\vc{\pi}\vc{Q} = \vc{0},
\qquad \vc{\pi}\vc{e} = 1,
\]
where $\vc{e}$ denotes a column vector of ones with an appropriate
order according to the context.

Let $\vc{\pi}(k) = (\pi(k,i))_{i\in\bbS_k}$ for $k\in\bbZ_+$, which is
the subvector of $\vc{\pi}$ corresponding to level $k$ and thus 
$\vc{\pi}=(\vc{\pi}(0),\vc{\pi}(1),\dots)$.  It is, in general,
difficult to compute $\vc{\pi}=(\vc{\pi}(0),\vc{\pi}(1),\dots)$
because we have to solve an infinite dimensional system of equations.
As for the BSMCs with the special structures mentioned above, we can
establish the stochastically interpretable expression of the
stationary distribution vector by matrix analytic methods
(\citet{Gras90,Lato99,Neut89,Zhao98}) and can also obtain the
analytical expression of the stationary distribution vector by
continued fraction approaches (\citet{Hans99,Pear89}).  However, the
construction of such expressions requires an infinite number of
computational steps involving an infinite number of block matrices
that characterize those BSMCs.

To solve this problem practically, we can truncate infinite
iterations (e.g., infinite sums, products and other algebraic
operations) and/or truncate the infinite set of block matrices. The
former truncation includes the state-space truncation and is
incorporated into many algorithms in the literature
(\citet{Baum10,Brig95,Gras93,Masu16,Phun10-QTNA,Taki16}). On the other
hand, the latter truncation can be achieved by the state-space
truncation, banded approximation (\citet{Zhao99}), spatial
homogenization (\citet{Klim06,LiuQuan05,Shin98}), etc.

This paper considers the last-column-block-augmented northwest-corner
truncation (LC-block-augmented truncation, for short) of $\vc{Q}$ and
thus the BSMC $\{(X(t),J(t))\}$ (see
\citet{LiHai00,Masu15-ADV,Masu16-SIAM,Masu17-LAA}). The
LC-block-augmented truncation is one of the state-space truncations
and is also a special case of {\it block-augmented truncations} (see,
e.g., \citet[Section~3]{LiHai00} for the discrete-time case; and
\citet[Definition~4.1]{Masu17-LAA} for the continuous-time case). In fact,
the LC-block-augmented truncation is an extension of the
last-column-augmented northwest-corner truncation
(last-column-augmented truncation, for short; see, e.g.,
\citet{Gibs87}) to BSMCs.

The reason we focus on the LC-block-augmented truncation is
twofold. The first reason is that the LC-block-augmented
truncation yields the best (in a certain sense) approximation to the
stationary distribution vector of {\it block-monotone} BSMCs among the
approximations by block-augmented truncations (see \citet[Theorem~3.6]{LiHai00} and \citet[Theorem~4.1]{Masu17-LAA}). Note here that block
monotonicity is an extension of {\it (classical) monotonicity} (see
\citet{Dale68}) to BSMCs (see, e.g., \citet[Definition
  1.1]{Masu15-ADV} and \citet[Definition 3.2]{Masu17-LAA} for the
definition of block monotonicity). Note also that block monotonicity
appears in the queue length processes of such representative
semi-Markovian queues as BMAP/GI/1, BMAP/M/$s$ and BMAP/M/$\infty$
queues (see \citet{Masu15-ADV,Masu16-SIAM,Masu17-LAA}). 

The second reason is that the
LC-block-augmented truncation is related to queueing models with
finite capacity. The (possibly embedded) queue length processes in
semi-Markovian queues with {\it finite} capacity (such as
MAP/PH/$s$/$N$ and MAP/GI/1/$N$; see, e.g., \citet{Baio94,Miya07}) can
be considered the LC-block-augmented truncations of the queue length
processes in the corresponding semi-Markovian queues with {\it
  infinite} capacity. Therefore, the estimation of the ``difference"
between those finite and infinite queues is reduced to the error
estimation of the LC-block-augmented truncation.

The above two reasons lead us to focus on the LC-block-augmented
truncation.  We now outline the procedure to construct the
LC-block-augmented truncation of $\vc{Q}$. To this end, we need some
symbols and notation. Let $| \cdot |$ denote the cardinality of the
set in the vertical bars.  Let $\bbF_n = \cup_{k=0}^n \bbL_k \subset
\bbF$ and $\overline{\bbF}_n = \bbF \setminus \bbF_n =
\cup_{k=n+1}^{\infty} \bbL_k$ for $n \in \bbZ_+$. In addition, let
$k_{\ast} = \inf\{k \in \bbN; S_{\ell} = S_k\ \mbox{for all $\ell \ge
  k$}\}$. Throughout the paper, unless otherwise stated, we assume
that $k_{\ast} = 1$, i.e.,
\[
S_k = S_1\quad \mbox{for all $k \in \bbN$}.
\]
It should be noted that the case where $k_{\ast} \ge 2$ can be reduced
to the case where $k_{\ast} = 1$ by relabeling
$\cup_{\ell=0}^{k_{\ast}-1}\bbL_{\ell}, \bbL_{k_{\ast}},
\bbL_{k_{\ast}+1},\dots$ as levels $0, 1, 2,\dots$, respectively.

Under the above assumption, we define $\vc{Q}_{\bbF_n} =
(q(k,i;\ell,j))_{(k,i;\ell,j)\in(\bbF_n)^2}$ for
$n\in\bbN$, which is the $|\bbF_n| \times |\bbF_n|$ northwest-corner
truncation of $\vc{Q}$, i.e.,


%
\begin{equation}
\vc{Q}_{\bbF_n}
= 
\left(
\begin{array}{cccc|c}
\vc{Q}(0;0)   & 
\vc{Q}(0;1)   &
\cdots        &
\vc{Q}(0;n-1) &
\vc{Q}(0;n) 
\\
\vc{Q}(1;0)   & 
\vc{Q}(1;1)   &
\cdots        &
\vc{Q}(1;n-1) &
\vc{Q}(1;n) 
\\
\vdots &
\ddots & 
\ddots & 
\vdots
\\
\vc{Q}(n-1;0)   & 
\vc{Q}(n-1;1)   &
\cdots        &
\vc{Q}(n-1;n-1) &
\vc{Q}(n-1;n) 
\\
\vc{Q}(n;0)   & 
\vc{Q}(n;1)   &
\cdots        &
\vc{Q}(n;n-1) &
\vc{Q}(n;n) 
\end{array}
\right).
\label{defn-Q_F_n}
\end{equation}
Since the BSMC $\{(X(t),J(t))\}$ is irreducible,
$\vc{Q}_{\bbF_n}$ is not conservative. In order to form a
conservative $q$-matrix from $\vc{Q}_{\bbF_n}$, we
augment the last block-column of the $|\bbF_n| \times |\bbF_n|$
northwest-corner truncation $\vc{Q}_{\bbF_n}$ by
\[
\left(
\begin{array}{cccc}
\sum_{m=n+1}^{\infty}\vc{Q}(0;m)
\\
\sum_{m=n+1}^{\infty}\vc{Q}(1;m)
\\ 
\vdots
\\
\sum_{m=n+1}^{\infty}\vc{Q}(n;m)
\end{array}
\right).
\]
We then extend the augmented northwest-corner truncation
$\vc{Q}_{\bbF_n}$ to the order of the original generator $\vc{Q}$ in
the manner described below, which enables us to perform algebraic
operations on the resulting $q$-matrix and original generator $\vc{Q}$.

We now provide a formal definition of the LC-block-augmented
truncation of the infinitesimal generator $\vc{Q}$. To shorten
expressions, we use the notation: $x \vmin y = \min(x,y)$. For $n \in
\bbN$, let
$\presub{n}\vc{Q}:=(\presub{n}q(k,i;\ell,j))_{(k,i;\ell,j)\in\bbF^2}$
denote a block-structured conservative $q$-matrix whose block matrices
$\presub{n}\vc{Q}(k;\ell):=(\presub{n}q(k,i;\ell,j))_{(i,j)\in\bbS_{k
    \vmin 1}\times\bbS_{\ell \vmin 1}}$, $k,\ell\in\bbZ_+$ are given
by
\begin{equation}
\presub{n}\vc{Q}(k;\ell)
=\left\{
\begin{array}{ll}
\vc{Q}(k;\ell), &  \mbox{if}~k \in \bbZ_+,\ 0 \le \ell \le n -1,
\\
\vc{Q}(k;n) + \dm\sum_{m > n,\, m \neq k}\vc{Q}(k;m), &  \mbox{if}~k \in \bbZ_+,\ \ell = n,
\\
\vc{Q}(k;k), &  \mbox{if}~k = \ell \ge n+1,
\\
\vc{O}, & \mbox{otherwise}.
\end{array}
\right.
\label{defn-(n)_Q}
\end{equation}
We call $\presub{n}\vc{Q}$ the {\it last-column-block-augmented
  $|\bbF_n| \times |\bbF_n|$ northwest-corner truncation
  (LC-block-augmented truncation, for short)} of $\vc{Q}$.

We now have the following result, whose proof is given in
Appendix~\ref{appen-proof-prop-{n}Q-communication}.
\begin{prop}\label{prop-{n}Q-communication}
For $n \in \bbN$, let $\{(\presub{n}X(t),\presub{n}J(t));t\ge0\}$ denote a Markov chain
with state space $\bbF$ and infinitesimal generator
$\presub{n}\vc{Q}$.  If the original generator $\vc{Q}$ is
irreducible, then (i) the Markov chain
$\{(\presub{n}X(t),\presub{n}J(t))\}$ (and thus $\presub{n}\vc{Q}$)
has at least one and at most $(S_1+1)$ closed communicating classes in
$\bbF_n$; and (ii) has no closed communicating classes in
$\overline{\bbF}_n$.
\end{prop}

Proposition~\ref{prop-{n}Q-communication} shows that the
LC-block-augmented truncation $\presub{n}\vc{Q}$ of the ergodic
generator $\vc{Q}$ may have more than one stationary distribution
vector. On the other hand, it follows from Theorem 2.1 and Remark 2.2
of \citet{Hart12} that
\begin{eqnarray*}
\lefteqn{
\lim_{n\to\infty}
\PP(\presub{n}X(t) = \ell,\presub{n}J(t)=j \mid \presub{n}X(0) = k,\presub{n}J(t)=i)
}
\quad
\nonumber
\\
&=& \PP(X(t) = \ell,J(t)=j \mid X(0) = k,J(t)=i),\qquad t \ge 0,\ (k,i;\ell,j)\in \bbF^2.
\end{eqnarray*}
From this fact and the ergodicity of $\vc{Q}$, we can expect that, in
many {\it natural} settings, $\presub{n}\vc{Q}$ has a single closed
communicating class in $\bbF_n$ for all $n$'s larger than some finite
$n_{\ast} \in \bbN$. Such cases are reduced to the special case where
$n_{\ast} = 1$ by relabeling $\cup_{\ell=0}^{n_{\ast}-1}\bbL_{\ell},
\bbL_{n_{\ast}}, \bbL_{n_{\ast}+1},\dots$ as levels $0, 1, 2,\dots$,
respectively.  Thus, for convenience, we assume that, for each $n \in
\bbN$, $\presub{n}\vc{Q}$ has a single closed communicating class in
the sub-state space $\bbF_n$, which implies that $\presub{n}\vc{Q}$
has the unique closed communicating class in the whole state space
$\bbF$ because all the states in $\overline{\bbF}_n$ are transient due
to Proposition~\ref{prop-{n}Q-communication}~(ii). As a result,
$\presub{n}\vc{Q}$ has the unique stationary distribution vector (see,
e.g., \citet[Section 5.4, Theorem 4.5]{Ande91}).

For $n \in \bbN$, let $\presub{n}\vc{\pi} := (\presub{n}\pi(k,i))_{(k,i)\in\bbF}$ denote
the unique stationary distribution vector of $\presub{n}\vc{Q}$, which
satisfies
\begin{equation}
\presub{n}\vc{\pi}\presub{n}\vc{Q} = \vc{0},
\quad \presub{n}\vc{\pi}\vc{e} = 1,\quad n \in \bbN.
\label{eqn-{n}pi-{n}Q=0}
\end{equation}
Since $\overline{\bbF}_n$ is transient, it holds (see
\citet[Lemma 4.2]{Masu17-LAA}) that
\begin{equation}
\presub{n}\vc{\pi}(k)=\vc{0}\quad \mbox{for all $k \ge n+1$ and $n \in \bbN$},
\label{eqn-{n}pi(k)=0}
\end{equation}
where $\presub{n}\vc{\pi}(k) :=
(\presub{n}\pi(k,i))_{i\in\bbS_{k\vmin1}}$ is the subvector of
$\presub{n}\vc{\pi}$ corresponding to level $k$.  It follows from
(\ref{eqn-{n}pi(k)=0}) that (\ref{eqn-{n}pi-{n}Q=0}) is reduced to a
finite dimensional system of equations and thus is solvable
numerically. Therefore, we consider $\presub{n}\vc{\pi}$ to be a
computable approximation to the stationary distribution vector
$\vc{\pi}$ of the original generator $\vc{Q}$.

From a practical point of view, it is significant to estimate the
error of the approximation $\presub{n}\vc{\pi}$ to $\vc{\pi}$, and
further, to derive computable error bounds for the approximation
$\presub{n}\vc{\pi}$.  Several authors have derived computable error
bounds for the approximation $\presub{n}\vc{\pi}$.
\citet{Twee98} and \citet{LiuYuan10} considered the
last-column-augmented truncation of discrete-time Markov chains
without block structure, which correspond to the case where $S_k = 0$
for all $k \in \bbZ_+$ in the context of this paper.
\citet{Twee98} assumed that the original Markov chain is
monotone and geometrically ergodic, and derived a computable upper
bound for the total variation distance between the stationary
distribution vectors of the original Markov chain and its
last-column-augmented truncation.  \citet{LiuYuan10} presented a
similar bound under the assumption that the original Markov chain is
monotone and polynomially ergodic. The monotonicity of Markov chains
is crucial to the derivation of the computable bounds presented in
\citet{Twee98} and \citet{LiuYuan10}.

Without the help of the monotonicity, \citet{Herv14} derived an error
bound for the stationary distribution vector of the
last-column-augmented truncation of a discrete-time Markov chain with
geometric ergodicity. However, the computation of \citet{Herv14}'s
bound requires the second largest eigenvalue of the last-column-augmented
truncation and thus the bound is less computation-friendly than the
bounds presented in \citet{Twee98} and \citet{LiuYuan10}.
\citet{Masu15-ADV,Masu16-SIAM} extended the results in \citet{Twee98}
and \citet{LiuYuan10} to discrete-time block-monotone BSMCs with
geometric ergodicity and those with subgeometric ergodicity,
respectively. By the uniformization technique (see, e.g.,
\citet[Section~4.5.2]{Tijm03}), the bounds presented in
\citet{Masu15-ADV,Masu16-SIAM} are applicable to continuous-time
block-monotone BSMCs with bounded infinitesimal generators.

There have been some studies on the truncation of continuous-time
Markov chains. \citet{Zeif14b,Zeif14c} studied the truncation of a
weakly ergodic non-time-homogeneous birth-and-death process with
bounded transition rates (see also
\citet{Zeif14a,Zeif12}). \citet{Hart12} discussed the convergence of
the stationary distribution vectors of the augmented northwest-corner
truncations of continuous-time Markov chains with monotonicity or
exponential ergodicity. \citet{Masu17-LAA} presented computable upper
bounds for the total variation distance between the stationary
distribution vectors of a BSMC (with possibly unbounded transition
rates) and its LC-block-augmented truncation, under the assumption
that the BSMC is block-wise dominated by a Markov chain with block
monotonicity and exponential ergodicity.

In this paper, we do not assume either $\vc{Q}$ is bounded or block
monotone. In addition, we do not necessarily assume that $\vc{Q}$ has
a specified ergodicity, such as exponential ergodicity and polynomial
ergodicity.  Instead, we assume that $\vc{Q}$ satisfies the
$\vc{f}$-modulated drift condition (see \citet[Equation
  (7)]{Meyn93-Proc} and \citet[Section~14.2.1]{Meyn09}):
\begin{cond}[$\vc{f}$-modulated drift condition]\label{assumpt-f-ergodic}
There exist some $b > 0$, $K \in \bbZ_+$, column vectors
$\vc{v}:=(v(k,i))_{(k,i)\in\bbF} \ge \vc{0}$ and
$\vc{f}:=(f(k,i))_{(k,i)\in\bbF} \ge \vc{e}$ such that
\begin{equation}
\vc{Q}\vc{v} \le  - \vc{f} + b \vc{1}_{\bbF_K},
\label{ineqn-Qv}
\end{equation}
where, for any set $\bbC \subseteq \bbF$,
$\vc{1}_{\bbC}:=(1_{\bbC}(k,i))_{(k,i)\in\bbF}$ denotes a column
vector whose $(k,i)$th element $1_{\bbC}(k,i)$ is given by
\[
1_{\bbC}(k,i)
=\left\{
\begin{array}{ll}
1, & (k,i) \in \bbC,
\\
0, & (k,i) \in \bbF\setminus\bbC.
\end{array}
\right.
\]
\end{cond}

Condition~\ref{assumpt-f-ergodic} is the basic condition of this
paper. If $\vc{f} = c\vc{v}$ for some $c> 0$, then
Condition~\ref{assumpt-f-ergodic} is reduced to the exponential drift
condition (i.e., the drift condition for exponential ergodicity; see
\citet[Theorem 20.3.2]{Meyn09}). On the other hand, if $f(k,i) =
\varphi(v(k,i))$ for some nondecreasing differentiable concave
function $\varphi:[1,\infty) \to (0,\infty)$ with
  $\lim_{t\to\infty}\varphi'(t) = 0$, then
  Condition~\ref{assumpt-f-ergodic} is reduced to the subgeometric
  drift condition (i.e., the drift condition for subgeometric ergodicity)
  presented in \citet{Douc09}.

Under Condition~\ref{assumpt-f-ergodic}, we study the estimate of the
absolute difference between the time-averaged functionals of the BSMC
$\{(X(t),J(t)); t \ge 0\}$ and its LC-block-augmented truncation. Let
$\vc{g}:=(g(k,i))_{(k,i)\in\bbF}$ denote a nonnegative column
vector. It is known that if $\vc{\pi}\vc{g} < \infty$ then the
time-average of the functional $g(X(t),J(t))$ is equal to
$\vc{\pi}\vc{g}$ with probability one (see, e.g., \citet[Chapter~8,
  Theorem 6.2]{Brem99}), i.e.,
\[
\lim_{T\to\infty}{1 \over T}\int_0^T g(X(t),J(t)) \rmd t
= \vc{\pi}\vc{g}\quad \mbox{with probability one}.
\]
Note here that if 
\[
\vc{g}^{\top} = 
\bordermatrix{
& \bbL_0 & \bbL_1 & \bbL_2 & \bbL_3 & \cdots
\cr
& \vc{0} &\vc{e}^{\top} & 2\vc{e}^{\top}& 3\vc{e}^{\top}& \dots
},
\]
then $\vc{\pi}\vc{g}$ is the mean of the stationary distribution vector.

The main contribution of this paper is to derive several bounds of the
following types under different technical conditions (together with
Condition~\ref{assumpt-f-ergodic}):
\begin{eqnarray}
 | \vc{\pi} - \presub{n}\vc{\pi} |\, \vc{g} 
&\le& {\vc{\pi}\vc{g} + 1 \over 2}
E(n)\quad~ \mbox{for all $n \in \bbN$ and $\vc{0} \le \vc{g} \le \vc{f}$},
\label{this-paper-error-bound-01}
\\
\red{\sup_{\vc{e} \le \vc{g} \le \vc{f}}}
{ | \vc{\pi} - \presub{n}\vc{\pi} |\, \vc{g} 
\over \vc{\pi}\vc{g}
}
&\le& E(n)\qquad\qquad~~\mbox{for all $n \in \bbN$},
\label{this-paper-error-bound-02}
\end{eqnarray}
where $|\cdot|$ denotes the vector (resp.\ matrix) obtained by taking the
absolute values of the elements of the vector (resp.\ matrix) in the
vertical bars; and where the function $E$ is called the {\it error
  decay function} and may be different in different bounds. Note here
that $|\vc{\pi} \vc{g} - \presub{n}\vc{\pi} \vc{g}| \le | \vc{\pi} -
\presub{n}\vc{\pi} |\, \vc{g}$. Note also that (\ref{ineqn-Qv}) yields
$\vc{\pi}\vc{g} \le \vc{\pi}\vc{f} \le b$ for $\vc{0} \le \vc{g} \le
\vc{f}$. Thus, from (\ref{this-paper-error-bound-01}) and
(\ref{this-paper-error-bound-02}), we obtain the bounds for the
approximation $\presub{n}\vc{\pi}\vc{g}$ to the time-averaged
functional $\vc{\pi}\vc{g}$:
\begin{eqnarray*}
| \vc{\pi}\vc{g}  - \presub{n}\vc{\pi}\vc{g}  |
&\le& {b + 1 \over 2}
E(n)\quad~ \mbox{for all $n \in \bbN$ and $\vc{0} \le \vc{g} \le \vc{f}$},
\\
\red{\sup_{\vc{e} \le \vc{g} \le \vc{f}}}
{ | \vc{\pi}\vc{g}  - \presub{n}\vc{\pi}\vc{g}  |
\over \vc{\pi}\vc{g}
}
&\le& E(n)\qquad\quad~~\,\mbox{for all $n \in \bbN$}.
\end{eqnarray*}
Furthermore, (\ref{this-paper-error-bound-01}) (or
(\ref{this-paper-error-bound-02})) leads to
\[
| \vc{\pi} - \presub{n}\vc{\pi}
|\, \vc{e} \le E(n),\qquad n \in \bbN,
\]
which is an upper bound for the total variation distance between
$\vc{\pi}$ and $\presub{n}\vc{\pi}$.

We now remark that, as with this paper, \citet{Baum15}
considered a similar condition to Condition~\ref{assumpt-f-ergodic},
under which they studied the truncation error of the infinite sum in
calculating the time-averaged functional $\vc{\pi}\vc{g}$. More
specifically, they derived an upper bound for the relative error of
the truncated sum $\sum_{(k,i)\in\bbC}\pi(k,i)g(k,i)$ to the
time-averaged functional
$\vc{\pi}\vc{g}=\sum_{(k,i)\in\bbF}\pi(k,i)g(k,i)$, where $\bbC
\subset \bbF$ is a finite set.

The rest of this paper is divided into four sections.  In
Section~\ref{sec-error-bounds}, we begin with two facts: (i) $\vc{\pi}
- \presub{n}\vc{\pi}$ can be expressed through the deviation matrix
$\vc{D}:=(d(k,i;\ell,j))_{(k,i;\ell,j) \in \bbF^2}$ of the BSMC
$\{(X(t),J(t))\}$ (see (\ref{eqn-diff-pi}) below); and (ii) the deviation
matrix $\vc{D}$ is a solution of a certain Poisson equation (see (\ref{Poisson-EQ-D}) below). By Dynkin's formula (see, e.g.,
\citet{Meyn93-III}), we then derive an upper bound for $\left| \vc{D}
\right|\vc{g}$ under Condition~\ref{assumpt-f-ergodic}, i.e., the
$\vc{f}$-modulated drift condition. 
Furthermore, using the upper bound
for $\left| \vc{D} \right|\vc{g}$, we present the bounds of the two
types (\ref{this-paper-error-bound-01}) and
(\ref{this-paper-error-bound-02}) in Theorem~\ref{thm-f-ergodic}
below, which are the foundation of the subsequent results of this
paper. 

These fundamental bounds of the two types are characterized by 
an error decay function that includes the
implicit factors $\vc{\pi}\vc{v}$ and $\presub{n}\vc{\pi}$.  However,
if we find two essentially different solutions $(b,K,\vc{v},\vc{f})$ and
$(b^{\sharp},K^{\sharp},\vc{v}^{\sharp},\vc{f}^{\sharp})$ to
Condition~\ref{assumpt-f-ergodic} such that
$\lim_{k\to\infty}v(k,i)/f^{\sharp}(k,i) = 0$ for all $i \in \bbS_1$,
then we can remove $\presub{n}\vc{\pi}$ from the error decay function,
which facilitates the qualitative sensitivity analysis of the error
decay function. On the other hand, the factor $\vc{\pi}\vc{v}$ cannot
be computed but can be estimated from above when $\vc{Q}$ satisfies
the exponential drift condition. Indeed, if
Condition~\ref{assumpt-f-ergodic} holds for $\vc{f} = c\vc{v} \ge
\vc{e}$, then (\ref{ineqn-Qv}) yields $\vc{\pi}\vc{v} < b/c$. As a
result, we obtain a computable error decay function under the
exponential drift condition.

In Section~\ref{sec-reduction}, we propose a method that reduces the
generator $\vc{Q}$ satisfying Condition~\ref{assumpt-f-ergodic} to be
exponentially ergodic. Combining the proposed method and the results
in Section~\ref{sec-error-bounds}, we can establish computable error
decay functions under the general $\vc{f}$-modulated drift condition
with some mild technical conditions. As far as we know, such a
reduction to exponential ergodicity has not been reported in the
literature.

In Section~\ref{sec-application}, we consider LD-QBDs, which describe
the queue length processes in various state-dependent queues with
Markovian environments, such as M/M/$s$ retrial queues and their
variants and generalizations (see, e.g.,
\citet{Breu02,Dudi13,Phun10-JIMO,Phun13}). The study of LD-QBDs and
their related queueing models has been a hot topic in queueing theory
for the last couple of decades (for an extensive bibliography, see
\citet{Arta99,Arta10,Arta-Gome08}). To demonstrate the usefulness of
our error bounds, we apply them to an M/M/$s$ retrial queue and show
some numerical results. Furthermore, using the numerical results, we
discuss the properties of our error bounds.

Finally, in Section~\ref{sec-remarks}, we consider the perturbation of
the stationary distribution vector $\vc{\pi}$ caused by that of the
generator $\vc{Q}$. The perturbation analysis of Markov chains is
closely related to the error estimation of the truncation
approximation of Markov chains (see, e.g.,
\citet{Herv14,LiuYuan15}). Many perturbation bounds have been shown
for the stationary distribution of (time-homogeneous) infinite-state
Markov chains
(\citet{Anis88,Heid10,Herv14,Kart86a,Kart86b,Kart86c,LiuYuan12,LiuYuan15,Mitr05,Mouh10,Twee80});
though these bounds require specific conditions on ergodicity (such as
uniform and exponential ergodicity) and/or include parameters
difficult to be identified or calculated (such as the stationary
distribution, the ergodic coefficient and other parameters associated
with the convergence rate to the steady state). On the other hand, we
establish a computable perturbation bound under the general
$\vc{f}$-modulated drift condition, by employing the technique used to
derive the error bounds for the LC-block-augmented truncation.

\section{Error Bounds for LC-Block-Augmented Truncations}\label{sec-error-bounds}

This section discusses the error estimation of the time-averaged
functions of the LC-block-augmented truncation $\presub{n}\vc{Q}$
under Condition~\ref{assumpt-f-ergodic}. To this end, we focus on the
deviation matrix of the Markov chain $\{(X(t),J(t))\}$. Using an upper
bound associated with the deviation matrix, we derive the fundamental
bounds of the two types (\ref{this-paper-error-bound-01}) and
(\ref{this-paper-error-bound-02}). Furthermore, utilizing an
additional condition on $\vc{v}$ and another solution to
Condition~\ref{assumpt-f-ergodic}, we discuss the convergence and
simplification of the error decay function of the fundamental bounds.
We then consider a special case where $\vc{Q}$ is an exponentially
ergodic generator. In this special case, we establish computable error
decay functions and propose a procedure for computing them.

\subsection{General case}

For convenience, we summarize all the assumptions made in
Section~\ref{introduction}, except for
Condition~\ref{assumpt-f-ergodic}.
\begin{assumpt}\label{basic-assumpt}
The stochastic process $\{(X(t),J(t))\}$ is an ergodic regular-jump
Markov chain with infinitesimal generator $\vc{Q}$ given in
(\ref{defn-Q}). Furthermore, the LC-block-augmented truncation
$\presub{n}\vc{Q}$ has the unique closed communicating class in
$\bbF_n$ for each $n \in \bbN$.
\end{assumpt}

In addition to Assumption~\ref{basic-assumpt} and
Condition~\ref{assumpt-f-ergodic}, we assume $\vc{\pi}\vc{v} <
\infty$. It then follows that each element of $ \int_0^{\infty} |
\vc{P}^{(t)} - \vc{e}\vc{\pi}|\rmd t$ is finite (see \citet[Theorem
  7]{Meyn93-Proc}). Based on this, we define
$\vc{D}=(d(k,i;\ell,j))_{(k,i;\ell,j) \in \bbF^2}$ as the deviation
matrix of the Markov chain $\{(X(t),J(t))\}$, i.e.,
\[
\vc{D}
= \int_0^{\infty}
\left(
 \vc{P}^{(t)} - \vc{e}\vc{\pi}
\right) \rmd t.
\]
It is known that the deviation matrix $\vc{D}$ is a solution to the
following Poisson equation (see, e.g., \citet[Theorem~5.2]{Cool02}):
\begin{equation}
-\vc{Q}\vc{D} = \vc{I} - \vc{e}\vc{\pi}
\quad \mbox{with~$\vc{\pi}\vc{D}=\vc{O}$}.
\label{Poisson-EQ-D}
\end{equation}
It is also known (see, e.g., \citet[Section~4.1, Equation~(9)]{Heid10}) that
\begin{equation}
\presub{n}\vc{\pi} - \vc{\pi}
= \presub{n}\vc{\pi} \left( \presub{n}\vc{Q} - \vc{Q} \right)
\vc{D},\qquad n \in \bbN.
\label{eqn-diff-pi}
\end{equation}
Therefore, we estimate $\presub{n}\vc{\pi} - \vc{\pi}$ through the
deviation matrix $\vc{D}$. 

For the estimation of the
deviation matrix $\vc{D}$, we introduce some symbols. For $\beta > 0$, let
$\vc{\Phi}^{(\beta)}=(\phi^{(\beta)}(k,i;\ell,j))_{(k,i;\ell,j)\in\bbF^2}$
denote a stochastic matrix such that
\begin{equation}
\vc{\Phi}^{(\beta)}
= \int_0^{\infty} \beta \rme^{- \beta t} \vc{P}^{(t)} \rmd t > \vc{O},
\label{defn-K}
\end{equation}
where $\vc{\Phi}^{(\beta)} > \vc{O}$ follows from the ergodicity of
$\{(X(t),J(t))\}$. The positivity of $\vc{\Phi}^{(\beta)}$ implies that
any finite set $\bbC \subset \bbF$ is a petite set of
$\{(X(t),J(t))\}$. Indeed, for any finite set $\bbC \subset \bbF$, let
$\mathfrak{m}_{\bbC}^{(\beta)}$ denote a measure on the Borel
$\sigma$-algebra $\calB(\bbF)$ of $\bbF$ such that
\begin{eqnarray*}
\mathfrak{m}_{\bbC}^{(\beta)}(\ell,j)
&:=&
\mathfrak{m}_{\bbC}^{(\beta)}(\{(\ell,j)\}) 
= \min_{(k,i)\in\bbC} \phi^{(\beta)}(k,i;\ell,j) > 0 ,\qquad (\ell,j) \in \bbF.
\end{eqnarray*}
It then follows that, for any finite set $\bbC \subset \bbF$,
\begin{equation}
\sum_{(\ell,j) \in \bbA} \phi^{(\beta)}(k,i;\ell,j) 
\ge \mathfrak{m}_{\bbC}^{(\beta)}(\bbA),
\qquad (k,i) \in \bbC,~\bbA \in \calB(\bbF),
\label{ineqn-K_{alpha}(k,i;l,j)}
\end{equation}
which shows that $\bbC$ is $\mathfrak{m}_{\bbC}^{(\beta)}$-petite 
 (see \citet[Sections 5.5.2 and 20.3.3]{Meyn09}).

We now define $\breve{\vc{g}}:=(\breve{g}(k,i))_{(k,i)\in\bbF}$ as a
column vector such that $\vc{0} \le |\breve{\vc{g}}| \le \vc{f}$. From
(\ref{ineqn-Qv}), we then have
\begin{equation}
\vc{\pi}\, |\breve{\vc{g}}|
\le \vc{\pi}\vc{f} \le b \quad 
\mbox{for all $\vc{0} \le |\breve{\vc{g}}| \le \vc{f}$}.
\label{ineqn-pi_breve{g}<b}
\end{equation}
Thus, since $\vc{\pi}\breve{\vc{g}}$ is finite, it follows from (\ref{Poisson-EQ-D}) that $\vc{h} :=
\vc{D}\breve{\vc{g}}$ is a solution of the following Poisson equation:
\begin{equation}
- \vc{Q}\vc{h} = \breve{\vc{g}} - (\vc{\pi}\breve{\vc{g}}) \vc{e}
\quad \mbox{with $\vc{\pi}\vc{h} = \vc{0}$}.
\label{Poisson-eqn-h}
\end{equation}
In addition, the boundedness and uniqueness of the
solution $\vc{h} = \vc{D}\breve{\vc{g}}$ are guaranteed by Lemma~\ref{lem-unique-h} below.
\begin{lem}\label{lem-unique-h}
Suppose that Assumption~\ref{basic-assumpt} and
Condition~\ref{assumpt-f-ergodic} are satisfied. If $\vc{\pi}\vc{v} <
\infty$, then, for some $c_0 \in (0,\infty)$,
\begin{equation}
|\vc{D}\breve{\vc{g}}| \le c_0(\vc{v}+\vc{e})\quad 
\mbox{for all $\vc{0} \le |\breve{\vc{g}}| \le \vc{f}$},
\label{ineqn-|h|}
\end{equation}
and $\vc{h}=\vc{D}\breve{\vc{g}}$ is the unique solution of the
Poisson equation (\ref{Poisson-eqn-h}) having an additional constraint
$\vc{\pi}\,|\vc{h}|< \infty$.
\end{lem}

\begin{proof}
The bound (\ref{ineqn-|h|}) follows from \citet[Theorem~1.2]{Kont16}. 
Therefore, we prove the uniqueness of the solution
$\vc{h}=\vc{D}\breve{\vc{g}}$.  From 
(\ref{ineqn-|h|}) and $\vc{\pi}\vc{v} < \infty$, we have
\begin{equation}
\vc{\pi}\, |\vc{h}| = \vc{\pi}\, |\vc{D}\breve{\vc{g}}| 
\le c_0(\vc{\pi}\vc{v} + 1) < \infty \quad 
\mbox{for all $\vc{0} \le |\breve{\vc{g}}| \le \vc{f}$}.
\label{ineqn-pi*|h|}
\end{equation}
Thus, $\vc{h}=\vc{D}\breve{\vc{g}}$ is a solution of the Poisson
equation (\ref{Poisson-eqn-h}) having the constraint
$\vc{\pi}\,|\vc{h}|< \infty$. We now assume that there exists another
solution $\vc{h}'$ of (\ref{Poisson-eqn-h}) such that
$\vc{\pi}\,|\vc{h}'|< \infty$. It follows from (\ref{ineqn-pi*|h|}), $\vc{\pi}\,|\vc{h}'| < \infty$  and Proposition 1.1 of \citet{Glyn96} that $\vc{h}' = \vc{h} + c\vc{e}$ for some finite constant
$c$. Furthermore, since $\vc{\pi}\vc{h}'=\vc{\pi}\vc{h}=0$, the
constant $c$ must be equal to zero and therefore $\vc{h}' = \vc{h}$.
\end{proof}

\begin{com}
{\it 
For the proof of Lemma~\ref{lem-unique-h}, we use \citet[Theorem~1.2]{Kont16}, which requires that the finite discrete set $\bbC$ (which appears in Condition~\ref{assumpt-f-ergodic}) is a {\it closed small set} of the Markov chain $\{(X(t),J(t))\}$, i.e., there exist some $c,T > 0$ and probability measure $\frak{p}$ on the Borel
$\sigma$-algebra $\calB(\bbF)$ of $\bbF$ such that
\begin{equation}
\min_{(k,i)\in\bbC} 
\sum_{(\ell,j) \in\bbA} p^T(k,i;\ell,j) \ge c \frak{p}(\bbA),
\qquad \bbA \in \calB(\bbF).
\tag{EQ.1}
\end{equation}
Indeed, this is true. Since $\{(X(t),J(t))\}$ is ergodic, for each $(k,i;\ell,j) \in \bbF^2$ there exists some $T(k,i;\ell,j) > 0$ such that $p^{T(k,i;\ell,j)}(k,i;\ell,j) > 0$. Therefore, we have 
\begin{equation}
p^{t+T(k,i;\ell,j)}(k,i;\ell,j)
\ge p^{T(k,i;\ell,j)}(k,i;\ell,j) \rme^{-|q(\ell,j;\ell,j)| t} > 0
\quad \mbox{for all $t > 0$}.\tag{EQ.2}
\end{equation}
We now define $T(\ell,j)$, $(\ell,j) \in \bbF$, as
\[
T(\ell,j) = \max_{(k,i)\in\bbC}T(k,i;\ell,j) > 0,
\qquad (\ell,j) \in \bbF,
\]
which is finite due to the finiteness of $\bbC$.
It thus follows from (EQ.2) that, for every $(\ell,j) \in \bbF$,
\begin{eqnarray*}
\min_{(k,i)\in\bbC} p^{t+T(\ell,j)}(k,i;\ell,j) > 0
\quad \mbox{for all $t > 0$},
\end{eqnarray*}
which implies that (EQ.1) holds for some $c,T > 0$ and probability measure $\frak{p}$. }
\end{com}

\medskip

The following lemma presents a more specific bound for the
solution $\vc{h} = \vc{D}\breve{\vc{g}}$.
\begin{lem}\label{lem-bound-h}
Suppose that Assumption~\ref{basic-assumpt} and
Condition~\ref{assumpt-f-ergodic} are satisfied. If $\vc{\pi}\vc{v} <
\infty$, then
\begin{equation}
| \vc{D}\breve{\vc{g}} |
\le (|\vc{\pi} \breve{\vc{g}}| +1) 
\left[
\vc{v} 
+ \left(
\vc{\pi}\vc{v} 
+ { 2b \over \beta\overline{\phi}_K^{(\beta)} } 
\right)\vc{e}
\right]\quad 
\mbox{for all $\vc{0} \le |\breve{\vc{g}}| \le \vc{f}$},
\label{bound-h}
\end{equation}
where
\begin{eqnarray}
\overline{\phi}_K^{(\beta)}
= \sup_{(\ell,j) \in \bbF}\mathfrak{m}_{\bbF_K}^{(\beta)}(\ell,j)
= \sup_{(\ell,j) \in \bbF}\min_{(k,i)\in\bbF_K} \phi^{(\beta)}(k,i;\ell,j) 
> 0.
\label{defn-overline{varphi}_F_K}
\end{eqnarray}
\end{lem}

\begin{rem}
The bound (\ref{bound-h}) includes the implicit factors $|\vc{\pi}
\breve{\vc{g}}|$, $\vc{\pi}\vc{v}$ and
$\overline{\phi}_K^{(\beta)}$. Owing to (\ref{ineqn-pi_breve{g}<b}),
the first one $|\vc{\pi} \breve{\vc{g}}|$ is bounded from above by
$b$, i.e., $|\vc{\pi} \breve{\vc{g}}| \le b$. Furthermore, if $\vc{f}
= c \vc{v}$ for some $c > 0$ (i.e., Condition~\ref{assumpt-f-ergodic}
is reduced the exponential drift condition), then the second one
$\vc{\pi}\vc{v}$ is also bounded from above by $b/c$.  As for the last
one $\overline{\phi}_K^{(\beta)}$, we will later discuss the
estimation and computation of this factor in Section~\ref{subsec-exp}.
\end{rem}

\medskip
\noindent
{\it Proof of Lemma~\ref{lem-bound-h}.~}
For $(\ell,j) \in \bbF$, let
$\vc{h}_{(\ell,j)}:=(h_{(\ell,j)}(k,i))_{(k,i)\in\bbF}$ denote a column
vector such that
\begin{equation}
h_{(\ell,j)}(k,i)
= \EE_{(k,i)}\!\!
\left[ \int_0^{\tau(\ell,j)} \breve{g}(X(t),J(t)) \rmd t \right]
- (\vc{\pi}\breve{\vc{g}}) \EE_{(k,i)}[\tau(\ell,j)],
\qquad (k,i)\in\bbF,
\label{defn-widetilde{h}}
\end{equation}
where $\tau(\ell,j) = \inf\{t \ge 0:(X(t),J(t))=(\ell,j)\}$ for
$(\ell,j)\in\bbF$ and
\[
\EE_{(k,i)}[\,\, \cdot \,\,] 
= \EE[~ \cdot \mid X(0)=k, J(0)=i], 
\qquad (k,i) \in \bbF.
\]
According to Lemma~\ref{lem-Poisson-eq}, the column vector
$\vc{h}_{(\ell,j)}$ is a solution of a Poisson equation of the same
type as (\ref{Poisson-eqn-h}):
\begin{equation}
- \vc{Q}\vc{h}_{(\ell,j)} = \breve{\vc{g}} - (\vc{\pi}\breve{\vc{g}}) \vc{e}.
\end{equation}

We now suppose that $\vc{\pi}\, |\vc{h}_{(\ell,j)} | < \infty$. It then
follows from (\ref{ineqn-pi*|h|}) and Proposition 1.1 of
\citet{Glyn96} that there exists some finite constant $c$ such that
$\vc{D}\breve{\vc{g}} = \vc{h}_{(\ell,j)} + c\vc{e}$. Combining this
with $\vc{\pi}(\vc{D}\breve{\vc{g}}) = \vc{0}$, we have $c = -
\vc{\pi}\vc{h}_{(\ell,j)}$ and thus
\[
\vc{D}\breve{\vc{g}} = \vc{h}_{(\ell,j)} - (\vc{\pi}\vc{h}_{(\ell,j)})\vc{e}
\quad \footnote{This, ``for all $(\ell,j) \in \bbF$", corrects a typo ``for all $(k,i) \in \bbF$" in the original version.}\mbox{for all $(\ell,j) \in \bbF$},
\]
which leads to
\begin{equation*}
|\vc{D}\breve{\vc{g}}| 
\le \inf_{(\ell,j) \in \bbF}
\left\{ |\vc{h}_{(\ell,j)}| 
+ (\vc{\pi}\, |\vc{h}_{(\ell,j)}|)\vc{e}
\right\}.
\end{equation*}
Therefore, to obtain the bound (\ref{bound-h}), it suffices to prove that
\begin{equation}
| \vc{h}_{(\ell,j)} |
\le (|\vc{\pi} \breve{\vc{g}}| + 1) 
\left(\vc{v} + {b \over \beta\mathfrak{m}_{\bbF_K}^{(\beta)}(\ell,j)}\vc{e} 
\right),\qquad (\ell,j) \in \bbF,
\label{bound-widetilde{h}}
\end{equation}
which implies that $\vc{\pi}\, |\vc{h}_{(\ell,j)} | < \infty$ due to
$\vc{\pi}\vc{v} < \infty$.

In what follows, we derive the bound (\ref{bound-widetilde{h}}) by
using the technique in the proof of Theorem 2.2 of \citet{Glyn96}.  It
follows from (\ref{defn-widetilde{h}}), $|\breve{\vc{g}}| \le \vc{f}$
and $\vc{f} \ge \vc{e}$ that, for $(k,i;\ell,j) \in \bbF^2$,
\begin{eqnarray}
|h_{(\ell,j)}(k,i)|
&\le& \EE_{(k,i)}\!\!
\left[ \int_0^{\tau(\ell,j)} f(X(t),J(t)) \rmd t \right]
+ |\vc{\pi}\breve{\vc{g}}| \, \EE_{(k,i)}[\tau(\ell,j)]
\nonumber
\\
&\le& ( 1 + |\vc{\pi} \breve{\vc{g}}| )
\EE_{(k,i)}\!\!
\left[ \int_0^{\tau(\ell,j)} f(X(t),J(t)) \rmd t \right].
\label{ineqn-widetilde{h}}
\end{eqnarray}
It also follows from (\ref{ineqn-K_{alpha}(k,i;l,j)}) with $\bbC =
\bbF_K$ and $\bbA = \{(\ell,j)\}$ that
\begin{equation}
1_{\bbF_K}(k,i) 
\le {\phi^{(\beta)}(k,i;\ell,j) \over \mathfrak{m}_{\bbF_K}^{(\beta)}(\ell,j) },\qquad
(k,i;\ell,j) \in \bbF^2.
\label{ineqn-1_C}
\end{equation}
Furthermore, using (\ref{ineqn-1_C}) and Lemma~\ref{lem-compa}
(replacing $Y(t)$ with $(X(t),J(t))$; $i$ with $(k,i)$; $\tau$ with
$\tau(\ell,j)$; and $\vc{w}$ with $b\vc{1}_{\bbF_K}$), we obtain, for
$(k,i;\ell,j) \in \bbF^2$,
\begin{eqnarray}
\lefteqn{
\EE_{(k,i)}\!\!
\left[ \int_0^{\tau(\ell,j)} f(X(t),J(t)) \rmd t \right]
}
~~~&&
\nonumber
\\
&\le& v(k,i) + b 
\EE_{(k,i)}\!\!\left[ 
\int_0^{\tau(\ell,j)} 1_{\bbF_K}(X(t),J(t)) \rmd t
\right]
\nonumber
\\
&\le& v(k,i) + {b \over \mathfrak{m}_{\bbF_K}^{(\beta)}(\ell,j) }
\EE_{(k,i)} \!\!\left[ \int_0^{\tau(\ell,j)} \phi^{(\beta)}(X(t),J(t);\ell,j) \rmd t \right]
\nonumber
\\
&=& v(k,i) + {b \over \mathfrak{m}_{\bbF_K}^{(\beta)}(\ell,j) } 
\int_0^{\infty} \beta \rme^{-\beta u}
\EE_{(k,i)}\!\!\left[ \int_0^{\tau(\ell,j)} p^{(u)}(X(t),J(t);\ell,j) \rmd t \right]
\rmd u
\nonumber
\\
&=& v(k,i) + {b \over \mathfrak{m}_{\bbF_K}^{(\beta)}(\ell,j) }
\int_0^{\infty} \beta\rme^{-\beta u} 
\EE_{(k,i)}\!\!\left[ 
\int_0^{\tau(\ell,j)} 1_{\{(\ell,j)\}}(X(t+u),J(t+u)) \rmd t
\right]\rmd u,\qquad~~~
\label{ineqn-01}
\end{eqnarray}
where we use (\ref{defn-K}) in the second-to-last equality. 

It is easy to see that
\begin{equation*}
\EE_{(k,i)}\!\!\left[ 
\left. 
\int_0^{\tau(\ell,j)} 1_{\{(\ell,j)\}}(X(t+u),J(t+u)) \rmd t \,\right|
\tau(\ell,j) \le u
\right] \le u.
\end{equation*}
In addition, since $\tau(\ell,j)$ is the first passage time to state
$(\ell,j)$,
\begin{eqnarray*}
&& \EE_{(k,i)}\!\!\left[ 
\left. 
\int_0^{\tau(\ell,j)} 1_{\{(\ell,j)\}}(X(t+u),J(t+u)) \rmd t \,\right|
\tau(\ell,j) > u
\right]
\nonumber
\\
&& {} \quad =
\EE_{(k,i)}\!\!\left[ 
\left. 
\int_{\tau(\ell,j) - u}^{\tau(\ell,j)} 
1_{\{(\ell,j)\}}(X(t+u),J(t+u)) \rmd t \,\right|
\tau(\ell,j) > u
\right] 
\le u.
\end{eqnarray*}
Therefore, 
\[
\EE_{(k,i)}\!\!\left[ 
\int_0^{\tau(\ell,j)} 1_{\{(\ell,j)\}}(X(t+u),J(t+u)) \rmd t
\right] \le u,
\qquad (k,i;\ell,j) \in \bbF^2.
\]
Applying this inequality to the right hand side of (\ref{ineqn-01})
yields
\begin{eqnarray}
\EE_{(k,i)}\!\!
\left[ \int_0^{\tau(\ell,j)} f(X(t),J(t)) \rmd t \right]
&\le& v(k,i) + {b \over \mathfrak{m}_{\bbF_K}^{(\beta)}(\ell,j) } 
\int_0^{\infty} u \beta \rme^{-\beta u} \rmd u
\nonumber
\\
&=& v(k,i) + { b \over \beta \mathfrak{m}_{\bbF_K}^{(\beta)}(\ell,j)  },
\qquad (k,i;\ell,j) \in \bbF^2. \qquad~~
\label{ineqn-02}
\end{eqnarray}
Furthermore, substituting (\ref{ineqn-02}) into
(\ref{ineqn-widetilde{h}}) results in
\[
| \vc{h}_{(\ell,j)} |
\le (|\vc{\pi} \breve{\vc{g}}| + 1) 
\left(\vc{v} + {b \over \beta\mathfrak{m}_{\bbF_K}^{(\beta)}(\ell,j)}\vc{e} 
\right),\qquad (\ell,j) \in \bbF,
\]
which shows that (\ref{bound-widetilde{h}}) holds.
\QED

\medskip

From Lemma~\ref{lem-bound-h}, we have a similar bound for
$|\vc{D}|\vc{g}$ with $\vc{0} \le \vc{g} \le \vc{f}$.
\begin{lem}\label{lem-bound-|D|g}
Suppose that Assumption~\ref{basic-assumpt} and
Condition~\ref{assumpt-f-ergodic} are satisfied. If $\vc{\pi}\vc{v} <
\infty$, then
\begin{equation}
|\vc{D}|\, \vc{g}
\le (\vc{\pi} \vc{g} +1) 
\left[
\vc{v} 
+ \left(
\vc{\pi}\vc{v} 
+ { 2b \over \beta\overline{\phi}_K^{(\beta)} } 
\right)\vc{e}
\right]\quad 
\mbox{for all $\vc{0} \le \vc{g} \le \vc{f}$},
\label{bound-|D|g}
\end{equation}
where $\overline{\phi}_K^{(\beta)}$ is given in
(\ref{defn-overline{varphi}_F_K}).
\end{lem}

\medskip
\noindent
{\it Proof.~}
Let $\vc{d}(k,i)$, $(k,i) \in \bbF$, denote the $(k,i)$th row of
$\vc{D}$, i.e., $\vc{d}(k,i)=(d(k,i;\ell,j))_{(\ell,j) \in \bbF}$.
Furthermore, let $\sgn(\,\cdot\,)$ denote the sign function, i.e.,
\[
\sgn(x)
= \left\{
\begin{array}{ll}
1,  & x > 0,
\\
0,  & x = 0,
\\
-1, & x < 0.
\end{array}
\right.
\] 
It then follows that $|\vc{d}(k,i)|\,\vc{g}$ is the $(k,i)$th element
of $|\vc{D}|\, \vc{g}$ and
\begin{eqnarray}
|\vc{d}(k,i)|\,\vc{g}
&=& \sum_{(\ell,j) \in \bbF} |d(k,i;\ell,j)|\, g(\ell,j)
\nonumber
\\
&=& \sum_{(\ell,j) \in \bbF} d(k,i;\ell,j)\, \sgn(d(k,i;\ell,j))\, g(\ell,j),
\nonumber
\\
&=& \vc{d}(k,i)\widetilde{\vc{g}}_{(k,i)},\qquad (k,i) \in \bbF,
\label{eqn-|d(k,i)|g}
\end{eqnarray}
where $\widetilde{\vc{g}}_{(k,i)} :=
(\widetilde{g}_{(k,i)}(\ell,j))_{(\ell,j) \in \bbF}$ is a column
vector such that
\[
\widetilde{g}_{(k,i)}(\ell,j) = \sgn(d(k,i;\ell,j))\, g(\ell,j),
\qquad (\ell,j) \in \bbF.
\]
Since $\vc{0} \le \vc{g} \le \vc{f}$, we have $\vc{0} \le
|\widetilde{\vc{g}}_{(k,i)}| \le \vc{f}$ for $(k,i) \in \bbF$. Thus,
combining Lemma~\ref{lem-bound-h} with $|\vc{\pi}
\breve{\vc{g}}_{(k,i)}| \le \vc{\pi} \vc{g}$ yields
\begin{equation}
|\vc{D} \widetilde{\vc{g}}_{(k,i)}|
\le (\vc{\pi} \vc{g} + 1) 
\left[
\vc{v} 
+ \left(
\vc{\pi}\vc{v} 
+ { 2b \over \beta\overline{\phi}_K^{(\beta)} } 
\right)\vc{e}
\right],\qquad (k,i) \in \bbF.
\label{ineqn-|D-wt{g}_{(k,i)}|}
\end{equation}
It also follows from (\ref{eqn-|d(k,i)|g}) and
(\ref{ineqn-|D-wt{g}_{(k,i)}|}) that
\begin{eqnarray*}
|\vc{d}(k,i)|\,\vc{g}
&=& |\vc{d}(k,i)\widetilde{\vc{g}}_{(k,i)}|
\le (\vc{\pi} \vc{g} + 1) 
\left[
v(k,i)
+ \left(
\vc{\pi}\vc{v} 
+ { 2b \over \beta\overline{\phi}_K^{(\beta)} } 
\right)
\right],\qquad (k,i) \in \bbF,
\end{eqnarray*}
which shows that (\ref{bound-|D|g}) holds. \QED

\medskip

Let $\vc{v}(k) = (v(k,i))_{i\in\bbS_{k\vmin1}}$ and $\vc{f}(k) =
(f(k,i))_{i\in\bbS_{k\vmin1}}$ for $k \in \bbZ_+$, which are the
subvectors of $\vc{v}$ and $\vc{f}$, respectively, corresponding to
$\bbL_k$.  Using Lemma~\ref{lem-bound-|D|g}, we obtain the following
theorem.
\begin{thm}\label{thm-f-ergodic}
Suppose that Assumption~\ref{basic-assumpt} and
Condition~\ref{assumpt-f-ergodic} are satisfied. If $\vc{\pi}\vc{v} <
\infty$, then the following bounds hold for all $n \in \bbN$.
\begin{eqnarray}
\left| \vc{\pi} - \presub{n}\vc{\pi} \right| \vc{g}
&\le& {\vc{\pi}\vc{g} + 1 \over 2} E(n)
\quad \mbox{for all $\vc{0} \le \vc{g} \le \vc{f}$},
\label{bound-pi-g}
\\
\red{\sup_{\vc{e} \le \vc{g} \le \vc{f}}}
{ \left| \vc{\pi} - \presub{n}\vc{\pi} \right| \vc{g} \over \vc{\pi}\vc{g} }
&\le& E(n),
\label{bound-sup-pi-g}
\end{eqnarray}
where the error decay function $E$ is given by
\begin{eqnarray}
E(n)
&=& 2\sum_{k=0}^n \presub{n}\vc{\pi}(k)
\sum_{m=n+1}^{\infty}\vc{Q}(k;m)
\nonumber
\\
&&
{} \times
\left\{
\vc{v}(m) + \vc{v}(n) + 2\left(
\vc{\pi}\vc{v} 
+ { 2b \over \beta\overline{\phi}_K^{(\beta)} } 
\right)\vc{e}
\right\},\qquad n \in \bbN.
\label{defn-E_{(l,j)}(n)}
\end{eqnarray}
\end{thm}

\begin{rem}\label{rem-main-thm}
As with (\ref{ineqn-pi_breve{g}<b}), it holds that
\begin{equation}
\vc{\pi}\vc{g}
\le \vc{\pi}\vc{f} \le b\quad 
\mbox{for all $\vc{0} \le \vc{g} \le \vc{f}$}.
\label{ineqn-pi_g<b}
\end{equation}
Substituting (\ref{ineqn-pi_g<b}) into the right hand side of
(\ref{bound-pi-g}), we have a bound for $\left| \vc{\pi} -
\presub{n}\vc{\pi} \right| \vc{g}$ below.
\begin{equation*}
\left| \vc{\pi} - \presub{n}\vc{\pi} \right| \vc{g}
\le {b + 1 \over 2} E(n)
\quad \mbox{for all $\vc{0} \le \vc{g} \le \vc{f}$},
\end{equation*}
which is insensitive to $\vc{g}$.
\end{rem}

\begin{rem}
The error decay function $E$ in (\ref{defn-E_{(l,j)}(n)}) depends on a
free parameter $\beta$. In fact, the parameter $\beta$ is also
included by the other error decay functions presented in the rest of
this paper. Although it is, in general, difficult to find an optimal
$\beta$, we discuss the impact of $\beta$ on the error decay functions
through some numerical examples in Section~4.2.3.
\end{rem}

\medskip
\noindent
{\it Proof of Theorem~\ref{thm-f-ergodic}.~}
From (\ref{eqn-diff-pi}), we have
\begin{equation}
\left| \vc{\pi} - \presub{n}\vc{\pi} \right| \vc{g}
\le \presub{n}\vc{\pi}  \left| \presub{n}\vc{Q} - \vc{Q} \right| 
|\vc{D}|\,\vc{g},\qquad n \in \bbN.
\label{ineqn-error-04}
\end{equation}
Substituting (\ref{defn-Q}), (\ref{defn-(n)_Q}) and (\ref{bound-|D|g}) 
into (\ref{ineqn-error-04}) yields
\begin{eqnarray*}
\left| \vc{\pi} - \presub{n}\vc{\pi} \right| \vc{g}
&\le& (\vc{\pi}\vc{g} + 1)
\presub{n}\vc{\pi}  \left| \presub{n}\vc{Q} - \vc{Q} \right| 
\left[
\vc{v} + \left(
\vc{\pi}\vc{v} + { 2b \over \beta\overline{\phi}_K^{(\beta)} } 
\right)\vc{e}
\right]
\nonumber
\\
&=& (\vc{\pi}\vc{g} + 1)
\sum_{k=0}^n \presub{n}\vc{\pi}(k)
\sum_{m=n+1}^{\infty}\vc{Q}(k;m)
\nonumber
\\
&&
{} \times
\left\{
\vc{v}(m) + \vc{v}(n) + 2\left(
\vc{\pi}\vc{v} 
+ { 2b \over \beta\overline{\phi}_K^{(\beta)} } 
\right)\vc{e}
\right\},\qquad n \in \bbN,
\end{eqnarray*}
which leads to (\ref{bound-pi-g}). Furthermore, using (\ref{bound-pi-g}) and $\sup_{\vc{g} \ge \vc{e}}
(\vc{\pi}\vc{g} + 1) / (2\vc{\pi}\vc{g}) = 1$, we obtain
\begin{eqnarray*}
\sup_{\vc{e} \red{\le} \vc{g} \le \vc{f}}
{ \left| \vc{\pi} - \presub{n}\vc{\pi} \right| \vc{g} 
\over \vc{\pi}\vc{g} }
&\le& 
\sup_{\vc{e} \le \vc{g} \le \vc{f}}
{\vc{\pi}\vc{g} + 1 \over 2\vc{\pi}\vc{g} } \cdot E(n)
\le 
\sup_{\vc{g} \ge \vc{e}} 
{\vc{\pi}\vc{g} + 1 \over 2\vc{\pi}\vc{g} } \cdot E(n) = E(n),
\qquad n \in \bbN,
\end{eqnarray*}
which shows that (\ref{bound-sup-pi-g}) holds.  \QED

\medskip

In fact, we can often find a solution $(b,K,\vc{v},\vc{f})$ of
Condition~\ref{assumpt-f-ergodic} such that the subvector
$\vc{v}_{\overline{\bbF}_0}:=(v(k,i))_{(k,i)\in\overline{\bbF}_0}$ of
$\vc{v}$ is level-wise nondecreasing, i.e., $\vc{v}(k) \le
\vc{v}(k+1)$ for all $k \in \bbN$. In such cases, we obtain the
following result, which is used in Section~\ref{sec-reduction}.
\begin{lem}\label{rem-pi-f<b}
If Condition~\ref{assumpt-f-ergodic} holds and
$\vc{v}_{\overline{\bbF}_0}$ is level-wise nondecreasing, then
\begin{equation}
\vc{\pi}\vc{f} \le b,
\quad 
\presub{n}\vc{\pi}\vc{f}
\le b \quad \mbox{for all $n \in \bbN$}.
\label{ineqn-pi_f<b}
\end{equation}
\end{lem}

\medskip
\noindent
{\it Proof.~}
Pre-multiplying both sides of (\ref{ineqn-Qv}) by $\vc{\pi}$ yields
the first inequality of (\ref{ineqn-pi_f<b}).  Furthermore, it follows
from (\ref{defn-(n)_Q}) and $\vc{v}(k) \le \vc{v}(k+1)$ for all $k \in
\bbN$ that
\[
\sum_{\ell = 0}^{\infty}\presub{n}\vc{Q}(k;\ell)\vc{v}(\ell)
\le \sum_{\ell = 0}^{\infty}\vc{Q}(k;\ell)\vc{v}(\ell),
\qquad  k \in \bbZ_+,
\]
and thus $\presub{n}\vc{Q}\vc{v} \le \vc{Q}\vc{v}$.
From this result and (\ref{ineqn-Qv}), we have
\[
\presub{n}\vc{Q}\vc{v}
\le \vc{Q}\vc{v} \le -\vc{f} + b\vc{1}_{\bbF_K},\qquad n \in \bbN,
\]
which yields the second inequality of (\ref{ineqn-pi_f<b}). \QED

\medskip

We now present another error decay function $E^+$, which is weaker but
(slightly) more tractable than $E$. At the same time, we also provide
a sufficient condition for the error decay functions $E$ and $E^+$ to
converge to zero.
\begin{thm}\label{thm-convergence}
Suppose that the conditions of Theorem~\ref{thm-f-ergodic}
(Assumption~\ref{basic-assumpt}, Condition~\ref{assumpt-f-ergodic} and
$\vc{\pi}\vc{v} < \infty$) are satisfied; and that the subvector
$\vc{v}_{\overline{\bbF}_0}$ of $\vc{v}$ (appearing in
Condition~\ref{assumpt-f-ergodic}) is positive and level-wise
nondecreasing. Let $E^+(n)$, $n \in \bbN$, denote
\begin{eqnarray}
E^+(n)
&=& 4\sum_{k=0}^n \presub{n}\vc{\pi}(k)
\sum_{m=n+1}^{\infty}\vc{Q}(k;m)
\left\{
\vc{v}(m)+ \left(
\vc{\pi}\vc{v} 
+ { 2b \over \beta\overline{\phi}_K^{(\beta)} } 
\right)\vc{e}
\right\}, \qquad n \in \bbN. \qquad
\label{defn-E_{(l,j)}^+(n)}
\end{eqnarray}
Under these conditions, the error bounds (\ref{bound-pi-g}) and
(\ref{bound-sup-pi-g}) hold and
\begin{equation}
E(n) \le E^+(n),
\qquad n \in \bbN.
\label{ineqn-E(n)-E^+(n)}
\end{equation}
Furthermore, if
\begin{eqnarray}
\sup_{n\in\bbN}\sum_{(k,i)\in\bbF} 
\presub{n}\pi(k,i)\, |q(k,i;k,i)|\, v(k,i) < \infty,
\label{bound-{n}pi-v-q}
\end{eqnarray}
then
\begin{equation}
\lim_{n\to\infty}E(n) =
\lim_{n\to\infty}E^{+}(n) = 0.
\label{lim-E_{(l,j)}(n)=0}
\end{equation}
\end{thm}

\medskip
\noindent
{\it Proof.~} Since Theorem~\ref{thm-f-ergodic} is available, the
bounds (\ref{bound-pi-g}) and (\ref{bound-sup-pi-g})
hold. Furthermore, since $\vc{v}_{\overline{\bbF}_0}$ is positive and
level-wise nondecreasing,
\begin{equation}
\vc{0} < \vc{v}(k) \le \vc{v}(k+1)\quad \mbox{for all $k \in \bbN$},
\label{ineqn-0<v(k)<=v(k+1)}
\end{equation}
and thus
\[
\sum_{m=n+1}^{\infty}\vc{Q}(k;m)\vc{v}(n)
\le
\sum_{m=n+1}^{\infty} \vc{Q}(k;m)\vc{v}(m),
\qquad  0 \le k \le n,\ n \in\bbN.
\]
Applying this to (\ref{defn-E_{(l,j)}(n)}), we obtain
\begin{eqnarray*}
E(n)
&\le& 4\sum_{k=0}^n \presub{n}\vc{\pi}(k)
\sum_{m=n+1}^{\infty}\vc{Q}(k;m)
\left\{
\vc{v}(m)+ \left(
\vc{\pi}\vc{v} 
+ { 2b \over \beta\overline{\phi}_K^{(\beta)} } 
\right)\vc{e}
\right\} = E^+(n), \quad n \in \bbN,
\end{eqnarray*}
which shows that (\ref{ineqn-E(n)-E^+(n)}) holds.

It remains to prove that $\lim_{n\to\infty}E^+(n) = 0$. From
(\ref{ineqn-0<v(k)<=v(k+1)}), we have
\[
{\vc{v}(m) \over \dm\min_{(\ell,j) \in \overline{\bbF}_0}v(\ell,j) } 
\ge \vc{e},
\qquad m \in \bbN.
\]
It follows from this inequality and (\ref{defn-E_{(l,j)}^+(n)})
that, for $n \in \bbN$,
\begin{eqnarray}
E^{+}(n)
&\le& 4\left\{ 
1 + { 
\vc{\pi}\vc{v} 
+ \dm{ 2b \over \beta\overline{\phi}_K^{(\beta)} } 
\over 
\dm\min_{(\ell,j) \in \overline{\bbF}_0}v(\ell,j) 
} 
\right\}
\sum_{k=0}^n \presub{n}\vc{\pi}(k)
\sum_{m=n+1}^{\infty}\vc{Q}(k;m)
\vc{v}(m). 
\label{ineqn-E_{(l,j)}^+(n)}
\end{eqnarray}
It also follows
from (\ref{ineqn-Qv}) that, for $n \ge k$ and $(k,i) \in \bbF$,
\begin{eqnarray}
0 &\le& \sum_{(m,j) \in \overline{\bbF}_n}q(k,i;m,j) v(m,j)
\nonumber
\\
&=& 
- q(k,i;k,i)  v(k,i)
- \sum_{(m,j) \in \bbF_n \setminus \{(k,i)\}} q(k,i;m,j) v(m,j)
+ \sum_{(m,j) \in \bbF} q(k,i;m,j) v(m,j)
\nonumber
\\
&\le&  \left| q(k,i;k,i) \right| v(k,i)
- \sum_{(m,j) \in \bbF_n \setminus \{(k,i)\}} q(k,i;m,j)  v(m,j) - f(k,i) + b
\nonumber
\\
&\le& \left| q(k,i;k,i) \right| v(k,i) + b,
\label{add-150925-01}
\end{eqnarray}
which implies that $\sum_{(m,j) \in \bbF}|q(k,i;m,j)|\, v(m,j) < \infty$
for all $(k,i)\in\bbF$. Thus,
\begin{equation}
\lim_{n\to\infty}\sum_{m=n+1}^{\infty}\vc{Q}(k;m)\vc{v}(m) = \vc{0},
\qquad k \in \bbZ_+.
\label{add-150925-02}
\end{equation}
In addition, (\ref{bound-{n}pi-v-q}) and (\ref{add-150925-01}) yield
\begin{eqnarray*}
\lefteqn{
\sup_{n\in\bbN}\sum_{k=0}^n \presub{n}\vc{\pi}(k)
\sum_{m=n+1}^{\infty}\vc{Q}(k;m) \vc{v}(m)
}
\quad &&
\nonumber
\\
&=& \sup_{n\in\bbN}\sum_{(k,i)\in\bbF_n} 
\presub{n}\pi(k,i)
\sum_{(m,j)\in\overline{\bbF}_n}q(k,i;m,j) v(m,j)
\nonumber
\\
&\le& \sup_{n\in\bbN}\sum_{(k,i)\in\bbF_n} 
\presub{n}\pi(k,i)
\left\{ \left| q(k,i;k,i) \right| v(k,i) + b \right\}
\nonumber
\\
&\le& \sup_{n\in\bbN}\sum_{(k,i)\in\bbF} 
\presub{n}\pi(k,i)\left| q(k,i;k,i) \right| v(k,i) + b
< \infty.
\end{eqnarray*}
Therefore, applying the dominated convergence theorem to the right
hand side of (\ref{ineqn-E_{(l,j)}^+(n)}) and using
(\ref{add-150925-02}), we obtain $\lim_{n\to\infty}E^+(n) = 0$.  \QED

\medskip

Theorem~\ref{thm-convergence} provides a sufficient condition for
convergence to zero of the error decay functions $E$ and
$E^+$. However, the convergence condition, as well as, the error decay
functions themselves are not tractable in the sense that they include
the stationary distribution vector $\presub{n}\vc{\pi}$ of the
LC-block-augmented truncation $\presub{n}\vc{Q}$. In what follows, by
removing $\presub{n}\vc{\pi}$ from them, we derive a simple error
decay function and convergence condition. To this end, we focus on an
empirical fact that once we find a solution
$(b,K,\vc{v},\vc{f})$ to the $\vc{f}$-modulated drift condition (i.e.,
Condition~\ref{assumpt-f-ergodic}) then we can readily obtain an
essentially different solution
$(b^{\sharp},K^{\sharp},\vc{v}^{\sharp},\vc{f}^{\sharp})$. Thus, we
proceed under Condition~\ref{cond-two-solutions-01} below.
\begin{cond}\label{cond-two-solutions-01}
(i) Condition~\ref{assumpt-f-ergodic} holds, and
  $\vc{v}_{\overline{\bbF}_0}$ is positive and level-wise
  nondecreasing; and (ii) there exist some $b^{\sharp} > 0$,
  $K^{\sharp} \in \bbZ_+$, column vectors
  $\vc{v}^{\sharp}:=(v^{\sharp}(k,i))_{(k,i)\in\bbF} \ge \vc{0}$ and
  $\vc{f}^{\sharp}:=(f^{\sharp}(k,i))_{(k,i)\in\bbF} \ge \vc{e}$ such
  that
  $\vc{v}_{\overline{\bbF}_0}^{\sharp}:=(v^{\sharp}(k,i))_{(k,i)\in\overline{\bbF}_0}$
  is level-wise nondecreasing and
\begin{equation}
\vc{Q}\vc{v}^{\sharp} 
\le  - \vc{f}^{\sharp} + b^{\sharp} \vc{1}_{\bbF_{K^{\sharp}}}.
\label{ineqn-Q_*v_*}
\end{equation}
\end{cond}

Under Condition~\ref{cond-two-solutions-01}, we present a tractable
sufficient condition for convergence to zero of the error decay
functions $E$ and $E^+$.
\begin{thm}\label{thm-convegence-02}
Suppose that Assumption~\ref{basic-assumpt},
Condition~\ref{cond-two-solutions-01} and $\vc{\pi}\vc{v} < \infty$
are satisfied. We then have (\ref{bound-pi-g}), (\ref{bound-sup-pi-g})
and (\ref{ineqn-E(n)-E^+(n)}).  Furthermore, if
\begin{equation}
\sup_{(k,i)\in\bbF}
{|q(k,i;k,i)|\,  v(k,i) \over f^{\sharp}(k,i)}
< \infty,
\label{add-150718-01}
\end{equation}
then (\ref{lim-E_{(l,j)}(n)=0}) holds.
\end{thm}

\medskip
\noindent
{\it Proof.~} Under the present conditions,
Theorem~\ref{thm-convergence} holds. Thus, it suffices to prove that
(\ref{bound-{n}pi-v-q}) is satisfied. It follows from
(\ref{add-150718-01}) that, for some $C > 0$,
\[
|q(k,i;k,i)|\,  v(k,i) \le C f^{\sharp}(k,i)
\quad \mbox{for all $(k,i) \in \bbF$},
\] 
which leads to 
\begin{eqnarray}
\sum_{(k,i)\in\bbF} \presub{n}\pi(k,i) \,|q(k,i;k,i)|\,  v(k,i) 
\le C \cdot  \presub{n}\vc{\pi}\vc{f}^{\sharp}, \qquad n \in \bbN.
\label{add-150718-02}
\end{eqnarray}
Furthermore, since $\vc{v}_{\overline{\bbF}_0}^{\sharp}$ is level-wise
nondecreasing, it follows from (\ref{ineqn-Q_*v_*}) and
Lemma~\ref{rem-pi-f<b} that
\begin{equation}
\presub{n}\vc{\pi}\vc{f}^{\sharp} \le b^{\sharp},\qquad n \in \bbN.
\label{ineqn-pi-f^{ddagger}}
\end{equation}
Therefore, substituting this inequality into (\ref{add-150718-02})
yields
\begin{eqnarray*}
\sup_{n\in\bbN}\sum_{(k,i)\in\bbF}\presub{n}\pi(k,i)\,|q(k,i;k,i)|\,  v(k,i)
\le C b^{\sharp} < \infty,
\end{eqnarray*}
which completes the proof. \QED

\medskip

In addition to Condition~\ref{cond-two-solutions-01}, we assume the
following condition.
\begin{cond}\label{cond-two-solutions-02}
There exist a column vector $\vc{a}=(a(i))_{i\in\bbS_1} > \vc{0}$ and
two nondecreasing log-subadditive functions $V: [0,\infty) \to
  [1,\infty)$ and $T : [0,\infty) \to [1,\infty)$ such that
\begin{eqnarray}
\vc{v}(k) 
&=& V(k)\vc{a},\qquad k \in \bbN,
\label{explicit-cond-01}
\\
\lim_{x\to\infty}T(x) &=& \infty,
\label{explicit-cond-03}
\\
\sup_{(k,i)\in\bbF}{T(k)V(k) \over f^{\sharp}(k,i)}
&<& \infty,
\label{explicit-cond-02}
\\
\sup_{k,\ell\in\bbZ_+} T(\ell)
\left\| \sum_{m = \ell + 1}^{\infty}
\vc{Q}(k;k+m)V(m)\vc{a} \right\|_{\infty}
&<& \infty,
\label{explicit-cond-04}
\end{eqnarray}
where $\|\cdot\|_{\infty}$ denotes the $\infty$-norm (or called ``the
uniform norm").
\end{cond}

\begin{rem}\label{rem-log-subadditive}
A function $F: [0,\infty) \to [1,\infty)$ is said to be
    log-subadditive if $\log F(x+y) \le \log F(x) + \log F(y)$, or
    equivalently, $F(x+y) \le F(x)F(y)$ for all $x \ge 0$ and $y \ge
    0$.
\end{rem}

Using Conditions~\ref{cond-two-solutions-01} and
\ref{cond-two-solutions-02}, we obtain a convergent error decay
function.
\begin{thm}\label{thm-two-solutions}
If Assumption~\ref{basic-assumpt}, 
Conditions~\ref{cond-two-solutions-01} and \ref{cond-two-solutions-02}
are satisfied, then the error bounds (\ref{bound-pi-g}) and
(\ref{bound-sup-pi-g}) hold and
\begin{eqnarray}
E(n) 
&\le& E^{+}(n)
\le {4r_0^{\sharp} r_1^{\sharp}b^{\sharp} \over T(n) }
\left[
1 + { \underline{a}^{-1} \over V(n+1)}
\left(
\vc{\pi}\vc{v} 
+ { 2b \over \beta\overline{\phi}_K^{(\beta)} } \right)
\right],\qquad n \in \bbN,
\label{bound-pi-g-two-solutions}
\end{eqnarray}
where $\underline{a}$, $r_0^{\sharp}$ and $r_1^{\sharp}$ are positive
numbers such that
\begin{eqnarray}
\underline{a}
&=& \min_{i \in \bbS_1} a(i),
\label{defn-underline{a}}
\\
r_0^{\sharp}
&\ge& \sup_{(k,i)\in\bbF}{T(k)V(k) \over f^{\sharp}(k,i)},
\label{defn-c_*}
\\
r_1^{\sharp}
&\ge& \sup_{k,\ell\in\bbZ_+} T(\ell)
\left\| \sum_{m = \ell + 1}^{\infty}
\vc{Q}(k;k+m)V(m)\vc{a} \right\|_{\infty}. 
\label{defn-r_1^{sharp}}
\end{eqnarray}
\end{thm}

\medskip
\noindent
{\it Proof.~} We first confirm that the conditions of
Theorem~\ref{thm-convergence} are satisfied. Note that
Condition~\ref{cond-two-solutions-01} implies that
Condition~\ref{assumpt-f-ergodic} holds and that
$\vc{v}_{\overline{\bbF}_0}$ is positive and level-wise
nondecreasing. Thus, it suffices to show that $\vc{\pi}\vc{v} <
\infty$.  It follows from (\ref{ineqn-Q_*v_*}) that
\begin{equation}
\vc{\pi}\vc{f}^{\sharp} \le b^{\sharp}.
\label{ineqn-pi-f^{sharp}}
\end{equation}
It also follows from $T \ge 1$ and (\ref{explicit-cond-02}) that there
exists some $C > 0$ such that
\begin{equation}
V(k) \le C f^{\sharp}(k,i)
\quad \mbox{for all $(k,i)\in\bbF$.}
\label{bound-V(k)}
\end{equation}
Using (\ref{explicit-cond-01}), (\ref{ineqn-pi-f^{sharp}}) and
(\ref{bound-V(k)}), we have
\begin{eqnarray*}
\vc{\pi}\vc{v}
&=& \sum_{i\in\bbS_0} \pi(0,i)v(0,i)
+ \sum_{k=1}^{\infty}\sum_{i\in\bbS_1} \pi(k,i)V(k)a(i)
\nonumber
\\
&\le& \sum_{i\in\bbS_0} \pi(0,i)v(0,i)
+ C
\sum_{k=1}^{\infty}\sum_{i\in\bbS_1} \pi(k,i)f^{\sharp}(k,i)a(i)
\nonumber
\\
&\le& \sum_{i\in\bbS_0} \pi(0,i)v(0,i)
+ C
\sum_{k=1}^{\infty}\sum_{i\in\bbS_1} \pi(k,i)f^{\sharp}(k,i) \sum_{j\in\bbS_1}a(j) 
\nonumber
\\
&\le& \sum_{i\in\bbS_0} \pi(0,i)v(0,i)
+ Cb^{\sharp} \sum_{j\in\bbS_1}a(j) < \infty,
\end{eqnarray*}
which shows that the conditions of Theorem~\ref{thm-convergence} are
satisfied. Therefore, (\ref{bound-pi-g}), (\ref{bound-sup-pi-g}) and
(\ref{ineqn-E(n)-E^+(n)}) hold.

In what follows, we prove the second inequality in
(\ref{bound-pi-g-two-solutions}).  Replacing $v(m)$ in
(\ref{defn-E_{(l,j)}^+(n)}) by $V(m)\vc{a}$ (see
(\ref{explicit-cond-01})) yields
\begin{eqnarray}
E^{+}(n)
&=& 4\sum_{k=0}^n \presub{n}\vc{\pi}(k)
\sum_{m=n+1}^{\infty}\vc{Q}(k;m)V(m)\vc{a}
\nonumber
\\
&& {} + 
4 \left(
\vc{\pi}\vc{v} 
+ { 2b \over \beta\overline{\phi}_K^{(\beta)} } 
\right)
\sum_{k=0}^n \presub{n}\vc{\pi}(k)
\sum_{m=n+1}^{\infty}\vc{Q}(k;m)\vc{e},\qquad n \in \bbN.
\label{add-eqn-150630-01}
\end{eqnarray}
Since $\vc{e} \le \vc{a}/\underline{a}$ and $V$ is nondecreasing,
\[
\sum_{m=n+1}^{\infty}\vc{Q}(k;m)\vc{e}
\le {\underline{a}^{-1} \over V(n+1)}
\sum_{m=n+1}^{\infty}\vc{Q}(k;m)V(m)\vc{a},\qquad n \in \bbN.
\]
Substituting this inequality into (\ref{add-eqn-150630-01}), we have, for $n \in \bbN$,
\begin{eqnarray}
E^{+}(n)
&\le& 
4 \left[ 1 + {\underline{a}^{-1} \over V(n+1)}
\left( 
\vc{\pi}\vc{v} 
+ { 2b \over \beta\overline{\phi}_K^{(\beta)} } 
\right)
\right]
\sum_{k=0}^n \presub{n}\vc{\pi}(k)
\sum_{m=n+1}^{\infty}\vc{Q}(k;m)V(m)\vc{a}.\qquad
\label{ineqn-E_{(l,j)}^{dag}(n)}
\end{eqnarray}
Note here that since $V \ge 1$ and $T \ge 1$ are log-subadditive (see
Remark~\ref{rem-log-subadditive}),
\begin{align}
&&&&
V(m) &\le V(k)V(m-k), & 0 &\le k \le m,~m \in \bbN, &&&&
\label{V-log-subadditive}
\\
&&&&
1 &\le { T(k)T(n-k) \over T(n) }, & 0 &\le k \le n,~~\,n \in \bbN.&&&&
\label{T-log-subadditive}
\end{align}
Using (\ref{V-log-subadditive}) and (\ref{T-log-subadditive}), we obtain, 
for $n \in \bbN$,
\begin{eqnarray}
\lefteqn{
\sum_{k=0}^n \presub{n}\vc{\pi}(k)
\sum_{m=n+1}^{\infty}\vc{Q}(k;m)V(m)\vc{a}
}
\quad &&
\nonumber
\\
&\le& 
\sum_{k=0}^n \presub{n}\vc{\pi}(k) {T(k)T(n-k) \over T(n)}  
\sum_{m=n+1}^{\infty}\vc{Q}(k;m)V(k)V(m-k)\vc{a}
\nonumber
\\
&=& {1 \over T(n)}\sum_{k=0}^n \presub{n}\vc{\pi}(k) T(k)V(k)
\cdot 
T(n-k) \sum_{m=n-k+1}^{\infty}\vc{Q}(k;k+m)V(m)\vc{a}
\nonumber
\\
&\le& {1 \over T(n)}\sum_{k=0}^n \presub{n}\vc{\pi}(k) T(k)V(k)\vc{e}
\cdot 
\sup_{k,\ell \in \bbZ_+} T(\ell)
\left\| \sum_{m = \ell + 1}^{\infty}\vc{Q}(k;k+m)V(m)\vc{a} \right\|_{\infty}
\nonumber
\\
&\le& {r_1^{\sharp} \over T(n)}
\sum_{k=0}^n \presub{n}\vc{\pi}(k) T(k)V(k)\vc{e},
\label{add-eqn-150630-02}
\end{eqnarray}
where the last inequality follows from (\ref{defn-r_1^{sharp}}). It
also follows from (\ref{defn-c_*}) that
\begin{equation}
T(k)V(k)\vc{e} \le r_0^{\sharp} \vc{f}^{\sharp}(k),\qquad k \in \bbZ_+.
\label{add-eqn-160905-01}
\end{equation}
Applying (\ref{add-eqn-160905-01}) to (\ref{add-eqn-150630-02}) and
using (\ref{ineqn-pi-f^{ddagger}}) leads to
\begin{eqnarray}
\sum_{k=0}^n \presub{n}\vc{\pi}(k)
\sum_{m=n+1}^{\infty}\vc{Q}(k;m)V(m)\vc{a}
&\le&  {r_0^{\sharp}r_1^{\sharp} \over T(n)}
\sum_{k=0}^n \presub{n}\vc{\pi}(k)\vc{f}^{\sharp}(k)
\le {r_0^{\sharp}r_1^{\sharp}b^{\sharp} \over T(n)},\qquad n \in \bbN.
\qquad
\label{add-eqn-160905-02}
\end{eqnarray}
Substituting (\ref{add-eqn-160905-02}) into
(\ref{ineqn-E_{(l,j)}^{dag}(n)}) results in
(\ref{bound-pi-g-two-solutions}).  \QED

\subsection{Exponentially ergodic case}\label{subsec-exp}

In this subsection, we derive some computable error bounds in the case
where $\vc{Q}$ is exponentially ergodic. To this end, we assume that
Condition~\ref{assumpt-f-ergodic} is satisfied together with
$\vc{f}=c\vc{v} \ge \vc{e}$ and $c> 0$ (see \citet[Theorem
  20.3.2]{Meyn09}), i.e., (\ref{ineqn-Qv}) is reduced to
\begin{equation}
\vc{Q}\vc{v} \le -c \vc{v} + b \vc{1}_{\bbF_K}.
\label{geo-drift-cond}
\end{equation}
From (\ref{geo-drift-cond}), we have $\vc{\pi}\vc{v} \le b / c$.
Applying this inequality to (\ref{defn-E_{(l,j)}(n)}) in
Theorem~\ref{thm-f-ergodic}, we obtain
\begin{eqnarray}
E(n)
&\le& 2\sum_{k=0}^n \presub{n}\vc{\pi}(k)
\sum_{m=n+1}^{\infty}\vc{Q}(k;m)
\nonumber
\\
&&
{} \times
\left\{
\vc{v}(m) + \vc{v}(n) + 2b\left(
{1 \over c}
+ { 2 \over \beta\overline{\phi}_K^{(\beta)} } 
\right)\vc{e}
\right\},\qquad n \in \bbN.
\label{defn-widetilde{E}_{(l,j)}(n)}
\end{eqnarray}
The right hand side of (\ref{defn-widetilde{E}_{(l,j)}(n)}) does not
include the computationally intractable factor $\vc{\pi}$. Thus, in
order to obtain a computable error decay function, we establish a
computable lower bound for $\overline{\phi}_K^{(\beta)}$.  In
estimating $\overline{\phi}_K^{(\beta)}$, we do not necessarily assume
that the vector $\vc{f}$ in Condition~\ref{assumpt-f-ergodic}
satisfies $\vc{f} = c\vc{v}$ for some $c > 0$.

Let $\vc{Q}_{\bbF_N} = (q(k,i;\ell,j))_{(k,i;\ell,j) \in (\bbF_N)^2}$
for $N \in \{K,K+1,\dots\}$, which is the $|\bbF_N| \times |\bbF_N|$
northwest corner of $\vc{Q}$. Let
$\vc{\Phi}_{\bbF_N}^{(\beta)}:=(\phi_{\bbF_N}^{(\beta)}(k,i;\ell,j))_{(k,i;\ell,j)\in(\bbF_N)^2}$,
$N \in \{K,K+1,\dots\}$, denote
\begin{equation}
\vc{\Phi}_{\bbF_N}^{(\beta)}
= \int_0^{\infty} \beta \rme^{-\beta t}
\exp\{ \vc{Q}_{\bbF_N} t\} \rmd t
= \left(\vc{I} - \vc{Q}_{\bbF_N}/\beta \right)^{-1}.
\label{defn-S_N^(beta)}
\end{equation} 
Since $\vc{Q}$ is an irreducible infinitesimal generator, its finite
northwest corner $\vc{Q}_{\bbF_N}$ is nonsingular and thus all the
eigenvalues of $\vc{Q}_{\bbF_N}$ are in the strictly left half of the
complex plane. Therefore, the matrix $\vc{\Phi}_{\bbF_N}^{(\beta)}$ in
(\ref{defn-S_N^(beta)}) is well-defined.

We now denote, by $[\,\cdot\,]_{\bbF_K}$, the $|\bbF_K| \times
|\bbF_K|$ northwest corner of the matrix in the square brackets. It
then follows from Proposition~2.2.14 of \citet{Ande91} that, for any
fixed $t \ge 0$ and $K \in \bbZ_+$,
\[
[\exp\{ \vc{Q}_{\bbF_N} t\}]_{\bbF_K} \nearrow
[\vc{P}^{(t)}]_{\bbF_K}\quad  \mbox{as $N \to \infty$}.
\]
Thus, by the monotone
convergence theorem, we have
\begin{eqnarray} 
\left[ 
\int_0^{\infty} \beta \rme^{-\beta t}
\exp\{ \vc{Q}_{\bbF_N} t\} \rmd t
\right]_{\bbF_K} \nearrow
\left[ 
\int_0^{\infty} \beta \rme^{-\beta t}
\vc{P}^{(t)} \rmd t
\right]_{\bbF_K}\quad \mbox{as $N \to \infty$}.
\label{add-eqn-00}
\end{eqnarray}
Combining (\ref{add-eqn-00}) with (\ref{defn-K}) and
(\ref{defn-S_N^(beta)}), we obtain
\begin{eqnarray}
\left[  \vc{\Phi}_{\bbF_N}^{(\beta)} \right]_{\bbF_K}
\nearrow
\left[ \vc{\Phi}^{(\beta)} \right]_{\bbF_K}
> \vc{O} \quad \mbox{as $N \to \infty$},
\label{add-eqn-160403-01}
\end{eqnarray}
which implies that, for all sufficiently large $N \in \{K,K+1,\dots\}$,
\begin{equation}
\vc{O} < [ \vc{\Phi}_{\bbF_N}^{(\beta)} ]_{\bbF_K}
\le \left[ \vc{\Phi}^{(\beta)} \right]_{\bbF_K}.
\label{ineqn-K}
\end{equation}

\begin{rem}
Suppose that $\vc{Q}_{\bbF_{N_0}}$ is irreducible for some $N_0 \in
\{K,K+1,\dots\}$. It then follows that, for all $N \ge N_0$,
$[\exp\{\vc{Q}_{\bbF_N} t\}]_{\bbF_K} > \vc{O}$ for all $t > 0$ and
thus $[ \vc{\Phi}_{\bbF_N}^{(\beta)} ]_{\bbF_K} > \vc{O}$ (see
(\ref{defn-S_N^(beta)})). Consequently, (\ref{ineqn-K}) holds for all
$N \ge N_0$.
\end{rem}

\begin{rem}\label{rem-Le-Boud}
Let $\vc{F}$ denote a nonnegative matrix such that
\begin{equation}
\vc{F}
= \vc{I} + {1 \over \overline{q}_{\bbF_N}^{(\beta)} + 1}
(\vc{Q}_{\bbF_N}/\beta - \vc{I}),
\label{defn-F}
\end{equation}
where $\overline{q}_{\bbF_N}^{(\beta)} =
\max_{(\ell,j)\in\bbF_N}|q(\ell,j;\ell,j)|/\beta$. It follows from
(\ref{defn-S_N^(beta)}) and (\ref{defn-F}) that
\begin{eqnarray}
\vc{\Phi}_{\bbF_N}^{(\beta)}
&=& {1 \over \overline{q}_{\bbF_N}^{(\beta)} + 1}
(\vc{I} - \vc{F})^{-1} 
= {1 \over \overline{q}_{\bbF_N}^{(\beta)} + 1} \sum_{m=0}^{\infty} \vc{F}^m,
\label{add-eqn-160906-01}
\end{eqnarray}
which leads to a numerically stable computation of
$\vc{\Phi}_{\bbF_N}^{(\beta)}=(\phi_{\bbF_N}^{(\beta)}(k,i;\ell,j))_{(k,i;\ell,j)
  \in (\bbF_N)^2}$. Indeed, \citet{Le-Boud91} proposed an efficient
and stable algorithm for computing $\vc{\Phi}_{\bbF_N}^{(\beta)} =
(\vc{I} - \vc{F})^{-1}$ (see Proposition 1 therein), which does not
depend on any structure of $\vc{F}$ and thus
$\vc{Q}_{\bbF_N}$. Furthermore, if $\vc{Q}_{\bbF_N}$ is
block-tridiagonal, then $\vc{Q}_{\bbF_N}/\beta - \vc{I}$ can be
considered the transient generator of a finite-state LD-QBD with an
absorbing state and thus its fundamental matrix
$\vc{\Phi}_{\bbF_N}^{(\beta)}=(\vc{I} - \vc{Q}_{\bbF_N}/\beta)^{-1}$
can be efficiently and stably computed by \citet{Shin09}'s algorithm.
\end{rem}

To proceed further, we fix $N \in \{K,K+1,\dots\}$ arbitrarily such
that (\ref{ineqn-K}) holds.  We then define
$\overline{\phi}_{K,N}^{(\beta)}$, $N \in \{K,K+1,\dots\}$, as
\begin{equation}
\overline{\phi}_{K,N}^{(\beta)}
= \sup_{(\ell,j)\in\bbF_N} \min_{(k,i)\in\bbF_K}\phi_{\bbF_N}^{(\beta)}(k,i;\ell,j),
\label{defn-phi_{K,N}}
\end{equation}
which is computable because so is $\vc{\Phi}_{\bbF_N}^{(\beta)}$ (see
Remark~\ref{rem-Le-Boud}). It follows from
(\ref{defn-overline{varphi}_F_K}), (\ref{add-eqn-160403-01}) and
(\ref{defn-phi_{K,N}}) that
\begin{equation}
\overline{\phi}_{K,N}^{(\beta)} 
\nearrow \overline{\phi}_K^{(\beta)}  \quad \mbox{as $N \to \infty$},
\label{ineqn-phi}
\end{equation}
which shows that $\overline{\phi}_{K,N}^{(\beta)}$ is a computable and
nontrivial lower bound for $\overline{\phi}_K^{(\beta)}$. As a result,
combining Theorem~\ref{thm-f-ergodic} with
(\ref{defn-widetilde{E}_{(l,j)}(n)}) and (\ref{ineqn-phi}), we have
the following result.
\begin{coro}\label{coro-exp-ergodic-comp}
Suppose that Assumption~\ref{basic-assumpt} is satisfied. Suppose that
there exist some $b>0$, $c > 0$, $K \in \bbZ_+$ and column vector $\vc{v}
\ge \vc{e}/c$ such that (\ref{geo-drift-cond}) holds; and fix $N \in
\{K,K+1,\dots\}$ arbitrarily such that (\ref{ineqn-K}) holds. Under these
conditions, we have, for all $n \in \bbN$,
\begin{eqnarray}
\left| \vc{\pi} - \presub{n}\vc{\pi} \right| \vc{g}
&\le& {\vc{\pi}\vc{g} + 1 \over 2} \widetilde{E}_N(n)
\quad \mbox{for all $\vc{0} \le \vc{g} \le c\vc{v}$},
\label{bound-pi-g-exp}
\\
\red{\sup_{\vc{e} \le \vc{g} \le c\vc{v}}}
{ \left| \vc{\pi} - \presub{n}\vc{\pi} \right| \vc{g} \over \vc{\pi}\vc{g} }
&\le& \widetilde{E}_N(n),
\label{bound-sup-pi-g-exp}
\end{eqnarray}
where the error decay function $\widetilde{E}_N$ is given by
\begin{eqnarray}
\widetilde{E}_N(n)
&=& 2\sum_{k=0}^n \presub{n}\vc{\pi}(k)
\sum_{m=n+1}^{\infty}\vc{Q}(k;m)
\nonumber
\\
&&
{} \times
\left\{
\vc{v}(m) + \vc{v}(n) + 2b\left(
{1 \over c}
+ { 2 \over \beta \overline{\phi}_{K,N}^{(\beta)} } 
\right)\vc{e}
\right\},\qquad n \in \bbN.
\label{defn-widetilde{E}_N(n)}
\end{eqnarray}
Furthermore, if the subvector $\vc{v}_{\overline{\bbF}_0}$ of $\vc{v}$
is level-wise nondecreasing, then $\widetilde{E}_N(n) \le
\widetilde{E}_N^{+}(n)$ for $n \in \bbN$, where
\begin{eqnarray}
\widetilde{E}_N^{+}(n)
&=& 4\sum_{k=0}^n \presub{n}\vc{\pi}(k)
\sum_{m=n+1}^{\infty}\vc{Q}(k;m) 
\left\{
\vc{v}(m) + b\left(
{1 \over c}
+ { 2 \over \beta \overline{\phi}_{K,N}^{(\beta)} } 
\right)\vc{e}
\right\},\qquad n \in \bbN. \qquad
\label{defn-widetilde{E}_N^+(n)}
\end{eqnarray}
\end{coro}

\medskip
\noindent
{\it Proof.~}
Recall that (\ref{defn-widetilde{E}_{(l,j)}(n)}) holds.  Applying
(\ref{ineqn-phi}) to (\ref{defn-widetilde{E}_{(l,j)}(n)}), we obtain
$E(n) \le \widetilde{E}_N(n)$ for $n\in\bbN$. Substituting this
inequality into (\ref{bound-pi-g}) and (\ref{bound-sup-pi-g}), we have
(\ref{bound-pi-g-exp}) and (\ref{bound-sup-pi-g-exp}),
respectively. Furthermore, it is clear that $\widetilde{E}_N(n) \le
\widetilde{E}_N^{+}(n)$ for $n \in \bbN$ if
$\vc{v}_{\overline{\bbF}_0}$ is level-wise nondecreasing. 
\QED

\medskip

It should be noted that the error decay functions $\widetilde{E}_N$
are $\widetilde{E}_N^{+}$ are computable. We summarize the procedure
for computing them.
\begin{enumerate}
\item Find $b> 0$, $c> 0$, $K \in \bbZ_+$ and $\vc{v} \ge \vc{e}/c$
  such that (\ref{geo-drift-cond}) holds.
\item Fix $\beta > 0$ arbitrarily and find $N \in \{K,K+1,\dots\}$
  such that (\ref{ineqn-K}) holds; and compute
  $\vc{\Phi}_{\bbF_N}^{(\beta)}$ by (\ref{add-eqn-160906-01}).
\item Compute $\overline{\phi}_{K,N}^{(\beta)}$ by (\ref{defn-phi_{K,N}}).
\item Compute $\presub{n}\vc{\pi}(k)$ for $k=0,1,\dots,n$.
\item Compute $\widetilde{E}_N(n)$ and $\widetilde{E}_N^{+}(n)$
  by (\ref{defn-widetilde{E}_N(n)}) and
  (\ref{defn-widetilde{E}_N^+(n)}), respectively.
\end{enumerate}

\medskip

We now present another corollary. 
\begin{coro}\label{coro-exp-ergodic-comp-02}
Suppose that Assumption~\ref{basic-assumpt} is satisfied; and
Conditions~\ref{cond-two-solutions-01} and \ref{cond-two-solutions-02}
are satisfied, together with $\vc{f} = \vc{c}\vc{v}$ for some $c >
0$. Fix $N \in \{K,K+1,\dots\}$ arbitrarily such that (\ref{ineqn-K})
holds. We then have the error bounds (\ref{bound-pi-g-exp}) and
(\ref{bound-sup-pi-g-exp}). In addition,
\begin{eqnarray}
\widetilde{E}_N(n) 
&\le& \widetilde{E}_N^{+}(n)
\nonumber
\\
&\le& {4r_0^{\sharp}r_1^{\sharp} b^{\sharp} \over T(n) }
\left[
1+ {\underline{a}^{-1} b \over V(n+1)}
\left(
{1 \over c} 
+ { 2 \over \beta \overline{\phi}_{K,N}^{(\beta)} } \right)
\right] =: \widetilde{E}_N^{\sharp}(n), \qquad n \in \bbN,
\label{bound-pi-g-two-solutions-EXP}
\end{eqnarray}
where $r_0^{\sharp}$ and $r_1^{\sharp}$ are positive numbers such that
(\ref{defn-c_*}) and (\ref{defn-r_1^{sharp}}) hold.
\end{coro}

\medskip
\noindent
{\it Proof.~} Corollary~\ref{coro-exp-ergodic-comp-02} is immediate
from (\ref{ineqn-phi}) and Theorem~\ref{thm-two-solutions}, and this
corollary is proved in a similar way to the proof of
Corollary~\ref{coro-exp-ergodic-comp}. Thus, we omit the details of
the proof.  \QED

\medskip

We close this section by summarizing the procedure for computing the
error decay function $\widetilde{E}_N^{\sharp}$ in
(\ref{bound-pi-g-two-solutions-EXP}).
\begin{enumerate}
\item Find $b> 0$, $c> 0$, $K \in \bbZ_+$, $\vc{v}(0) \ge \vc{e}/c$,
  $\vc{a} > \vc{0}$ and nondecreasing log-subadditive function $V \ge
  1$ such that $V(1)\vc{a} \ge \vc{e}/c$ and
\[
\vc{Q}
\left(
\begin{array}{c}
\vc{v}(0)
\\
V(1)\vc{a}
\\
V(2)\vc{a}
\\
\vdots
\end{array}
\right)
\le -c 
\left(
\begin{array}{c}
\vc{v}(0)
\\
V(1)\vc{a}
\\
V(2)\vc{a}
\\
\vdots
\end{array}
\right)
+ b\vc{1}_{\bbF_K}.
\]
\item Find $b^{\sharp}> 0$, $K^{\sharp} \in \bbZ_+$, $\vc{v}^{\sharp}
  \ge \vc{0}$, $\vc{f}^{\sharp} \ge \vc{e}$ and nondecreasing
  log-subadditive function $T \ge 1$ such that the subvector
  $\vc{v}_{\overline{\bbF}_0}^{\sharp}$ of $\vc{v}^{\sharp}$ is
  level-wise nondecreasing and the conditions (\ref{ineqn-Q_*v_*}),
  (\ref{explicit-cond-03}), (\ref{explicit-cond-02}) and
  (\ref{explicit-cond-04}) are satisfied.
\item Choose $r_0^{\sharp}$ and $r_1^{\sharp}$ such that
  (\ref{defn-c_*}) and (\ref{defn-r_1^{sharp}}) hold.
\item Fix $\beta > 0$ arbitrarily and find $N \in \{K,K+1,\dots\}$
  such that (\ref{ineqn-K}) holds; and compute
  $\vc{\Phi}_{\bbF_N}^{(\beta)}$ by (\ref{add-eqn-160906-01}).

\item Compute $\overline{\phi}_{K,N}^{(\beta)}$ by (\ref{defn-phi_{K,N}}).
\item Compute $\widetilde{E}_N^{\sharp}(n)$ by
  (\ref{bound-pi-g-two-solutions-EXP}), where $\underline{a}$ is given
  by (\ref{defn-underline{a}}).
\end{enumerate}

\section{Reduction to Exponentially Ergodic Case}\label{sec-reduction}

This section considers a procedure for establishing computable bounds
for $\left|\vc{\pi} - \presub{n}\vc{\pi} \right| \vc{g}$ with $\vc{0}
\le \vc{g} \le \vc{f}$ under the general $\vc{f}$-modulated drift
condition.

For any vector $\vc{x}$, we denote by $\vc{\Delta}_{\vc{x}}$ a
diagonal matrix whose $i$th diagonal element is equal to the $i$th
element of the vector $\vc{x}$. For any vectors $\vc{x}$ and $\vc{y} >
\vc{0}$ of the same order, we define $\vc{x}/\vc{y}$ as a vector such
that $\vc{\Delta}_{\vc{x}/\vc{y}}= \vc{\Delta}_{\vc{x}}
\vc{\Delta}_{\vc{y}}^{-1}$.  We also assume
Condition~\ref{assumpt-f-ergodic-02} below, in addition to
Assumption~\ref{basic-assumpt}.
\begin{cond}\label{assumpt-f-ergodic-02}
Condition~\ref{assumpt-f-ergodic} holds and
\begin{equation}
\overline{C}{}_{\vc{f}/\vc{v}} 
:= \sup_{(k,i) \in \bbF} {f(k,i) \over v(k,i)} < \infty.
\label{defn-C_{f/v}}
\end{equation}
\end{cond}

It follows from (\ref{defn-C_{f/v}}) that
\begin{eqnarray}
0 < \vc{\pi}(\vc{f}/\vc{v}) 
&\le& \overline{C}{}_{\vc{f}/\vc{v}},
\label{ineqn-pi_(f/v)<b-01}
\\
0 < \presub{n}\vc{\pi}(\vc{f}/\vc{v}) 
&\le& \overline{C}{}_{\vc{f}/\vc{v}} \quad
\mbox{for all $n \in \bbN$}.
\label{ineqn-pi_(f/v)<b-02}
\end{eqnarray}
Thus, we define $\widehat{\vc{\pi}}$ and
$\presub{n}\widehat{\vc{\pi}}$, $n\in\bbN$, as
\begin{eqnarray}
\widehat{\vc{\pi}} 
&=& {\vc{\pi}\vc{\Delta}_{\vc{f}/\vc{v}} 
\over \vc{\pi} \left(\vc{f}/\vc{v}\right) },
\label{eqn-hat{pi}}
\\
\presub{n}\widehat{\vc{\pi}} 
&=& {\presub{n}\vc{\pi}\vc{\Delta}_{\vc{f}/\vc{v}} 
\over \presub{n}\vc{\pi} \left(\vc{f}/\vc{v}\right)  },\qquad n \in \bbN,
\label{eqn-{n}hat{pi}}
\end{eqnarray}
respectively. We also define
$\widehat{\vc{Q}}$ and $\presub{n}\widehat{\vc{Q}}$, $n\in\bbN$, as
\begin{eqnarray}
\widehat{\vc{Q}} 
&=& \vc{\Delta}_{\vc{v}/\vc{f}}  \cdot \vc{Q},
\label{defn-hat{Q}}
\\
\presub{n}\widehat{\vc{Q}} 
&=& \vc{\Delta}_{\vc{v}/\vc{f}} \cdot \presub{n}\vc{Q},\qquad n \in \bbN,
\label{defn-{n}hat{Q}}
\end{eqnarray}
respectively. It then follows from
(\ref{eqn-hat{pi}})--(\ref{defn-{n}hat{Q}}) that $\widehat{\vc{Q}}$
and $\presub{n}\widehat{\vc{Q}}$ can be considered the $q$-matrices
with the stationary distribution vectors $\widehat{\vc{\pi}}$ and
$\presub{n}\widehat{\vc{\pi}} $, respectively. Furthermore, from
(\ref{defn-hat{Q}}) and Condition~\ref{assumpt-f-ergodic}, we have
\begin{eqnarray}
\widehat{\vc{Q}}\vc{v} 
&\le& - \vc{v} 
+ b \vc{\Delta}_{\vc{v}/\vc{f}} \vc{1}_{\bbF_K}
\le - \vc{v} 
+ \widehat{b} \vc{1}_{\bbF_K},
\label{ineqn-hat{Q}v}
\end{eqnarray}
where
\begin{equation*}
\widehat{b} = b \max_{(k,i)\in \bbF_K} v(k,i)/f(k,i).
\end{equation*}

Inequality (\ref{ineqn-hat{Q}v}) shows that $\widehat{\vc{Q}}$
satisfies the exponential drift condition and
\begin{equation}
\wh{\vc{\pi}}\vc{v} \le \widehat{b}.
\label{ineqn-wh{pi}*v}
\end{equation}
Thus, using Corollaries
\ref{coro-exp-ergodic-comp} and \ref{coro-exp-ergodic-comp-02}, we
obtain computable bounds for $\left| \widehat{\vc{\pi}} -
\presub{n}\widehat{\vc{\pi}} \right| \widehat{\vc{g}}$ with $\red{\vc{e}}
\le \widehat{\vc{g}} \le \vc{v}$, under appropriate conditions. As a
result, combining such bounds and Theorem~\ref{thm-reduction-EXP}
below, we have computable bounds for $\left|\vc{\pi} -
\presub{n}\vc{\pi} \right| \vc{g}$ with $\red{\vc{e}} \le \vc{g} \le
\vc{f}$.
\begin{thm}\label{thm-reduction-EXP}
Suppose that Assumption~\ref{basic-assumpt} and
Condition~\ref{assumpt-f-ergodic-02} are satisfied. Furthermore,
suppose that there exists some function $\widehat{E}: [0,\infty) \to
  [0,\infty)$ such that
\begin{equation}
\red{\sup_{\vc{e} \le \widehat{\vc{g}} \le \vc{v}}}
{
\left| \widehat{\vc{\pi}} - \presub{n}\widehat{\vc{\pi}} \right| 
\widehat{\vc{g}}
\over \widehat{\vc{\pi}} \widehat{\vc{g}}
}
\le \widehat{E}(n), \qquad n \in \bbN.
\label{ineqn-hat{E}(n)}
\end{equation}
Under these conditions, the following two bounds hold for $n \in
\bbN$:
\begin{eqnarray}
\left| \vc{\pi} - \presub{n}\vc{\pi} \right| \vc{e}
&\le& 2\widehat{E}(n),
\label{bound-pi-g-reduction-01}
\\
\red{\sup_{\vc{e} \le \vc{g} \le \vc{f}}}
{ \left| \vc{\pi} - \presub{n}\vc{\pi} \right| \vc{g}
\over
\vc{\pi}\vc{g}
}
&\le& \widehat{E}(n)
\left[
1 + 
{1 + \widehat{E}(n) 
\over 
\big( 1 - \widehat{E}(n) \vmin 1 \big) \vmax 
\big(\, \wh{b}\overline{C}_{\vc{f}/\vc{v}} \big)^{-1}
} 
\right], \qquad
\label{bound-pi-g-reduction-02}
\end{eqnarray}
where $x \vmax y = \max(x,y)$ and $x \vmin y = \min(x,y)$ (the latter
has been defined in Section~\ref{introduction}).  In addition, if the
subvector $\vc{v}_{\overline{\bbF}_0}$ of $\vc{v}$ is level-wise
nondecreasing, then
\begin{equation}
\red{\sup_{\vc{e} \le \vc{g} \le \vc{f}}}
{ \left| \vc{\pi} - \presub{n}\vc{\pi} \right| \vc{g}
\over
\vc{\pi}\vc{g}
}
\le \widehat{E}(n)
\left[
1 + 
{1 + \widehat{E}(n) 
\over 
\big( 1 - \widehat{E}(n) \vmin 1 \big) \vmax 
\big(\, \wh{b}\overline{C}_{\vc{f}/\vc{v}} \big)^{-1}
} 
\vmin 
b
\right],\qquad n \in \bbN.
\label{bound-pi-g-reduction-03}
\end{equation}
\end{thm}

\begin{rem}
Suppose that $\lim_{x\to\infty}\widehat{E}(x) = 0$. It then follows
from (\ref{bound-pi-g-reduction-02}) that, for all sufficiently large
$n \in \bbN$,
\begin{equation*}
\red{\sup_{\vc{e} \le \vc{g} \le \vc{f}}}
{ \left| \vc{\pi} - \presub{n}\vc{\pi} \right| \vc{g}
\over
\vc{\pi}\vc{g}
}
\le \widehat{E}(n)
\left(
1 + 
{1 + \widehat{E}(n) \over 1 - \widehat{E}(n)}
\right).
\end{equation*}
Furthermore, if $\widehat{E}(x)>0$ for all $x \ge 0$, then
\[
\limsup_{n\to\infty}
{1 \over \widehat{E}(n) }
\red{\sup_{\vc{e} \le \vc{g} \le \vc{f}}}
{ \left| \vc{\pi} - \presub{n}\vc{\pi} \right| \vc{g}
\over
\vc{\pi}\vc{g}
}
\le 2.
\]
\end{rem}

\medskip
\noindent
{\it Proof of Theorem~\ref{thm-reduction-EXP}.~}
It follows from (\ref{eqn-hat{pi}}) and (\ref{eqn-{n}hat{pi}}) that 
\begin{eqnarray}
\vc{\pi} 
&=& {\widehat{\vc{\pi}}\vc{\Delta}_{\vc{v}/\vc{f}} 
\over \widehat{\vc{\pi}} \left(\vc{v}/\vc{f}\right) },
\label{eqn-pi-hat{pi}}
\\
\presub{n}\vc{\pi} 
&=& {\presub{n}\widehat{\vc{\pi}}\vc{\Delta}_{\vc{v}/\vc{f}} 
\over \presub{n}\widehat{\vc{\pi}} \left(\vc{v}/\vc{f}\right) },
\qquad n \in \bbN,
\nonumber
\end{eqnarray}
which yield
\begin{eqnarray}
\vc{\pi} - \presub{n}\vc{\pi}
&=& \left[
{1 \over \widehat{\vc{\pi}} \left(\vc{v}/\vc{f}\right) } 
(\widehat{\vc{\pi}} 
- \presub{n}\widehat{\vc{\pi}} )
+ 
\left( 
{1 \over \widehat{\vc{\pi}} \left(\vc{v}/\vc{f}\right) }
-
{1 \over \presub{n}\widehat{\vc{\pi}} \left(\vc{v}/\vc{f}\right) }
\right)
\presub{n}\widehat{\vc{\pi}} 
\right]
\vc{\Delta}_{\vc{v}/\vc{f}}
\nonumber
\\
&=&
{1 \over \widehat{\vc{\pi}} \left(\vc{v}/\vc{f}\right)}
\left[
( \widehat{\vc{\pi}} - \presub{n}\widehat{\vc{\pi}} )
+
\left( 
1
-
{\widehat{\vc{\pi}} \left(\vc{v}/\vc{f}\right)  
\over \presub{n}\widehat{\vc{\pi}} \left(\vc{v}/\vc{f}\right) }
\right) \presub{n}\widehat{\vc{\pi}} 
\right]\vc{\Delta}_{\vc{v}/\vc{f}}. \qquad
\nonumber
\\
&=&
{1 \over \widehat{\vc{\pi}} \left(\vc{v}/\vc{f}\right)}
\left[
( \widehat{\vc{\pi}} - \presub{n}\widehat{\vc{\pi}} )
+
( \presub{n}\widehat{\vc{\pi}} - \widehat{\vc{\pi}} )
\left(\vc{v}/\vc{f}\right) 
{ 
\presub{n}\widehat{\vc{\pi}}
\over 
\presub{n}\widehat{\vc{\pi}} \left(\vc{v}/\vc{f}\right) 
}
\right]\vc{\Delta}_{\vc{v}/\vc{f}}, \qquad n \in \bbN.\qquad
\label{eqn-pi-hat{pi}-02}
\end{eqnarray}
We now fix $\red{\vc{e} \le\,\,} \widehat{\vc{g}} \le \vc{v}$ arbitrarily and
$\vc{g} = \vc{\Delta}_{\vc{f}/\vc{v}}\,\widehat{\vc{g}}$ (i.e.,
$\widehat{\vc{g}} = \vc{\Delta}_{\vc{v}/\vc{f}}\, \vc{g}$). It then
follows from (\ref{eqn-pi-hat{pi}}) that
\begin{equation}
\widehat{\vc{\pi}} \widehat{\vc{g}} 
= 
\vc{\pi}\vc{g} \cdot \widehat{\vc{\pi}} \left(\vc{v}/\vc{f}\right).
\label{eqn-hat{pi}hat{g}}
\end{equation}
Using (\ref{eqn-pi-hat{pi}-02}) and (\ref{eqn-hat{pi}hat{g}}), we
obtain, for $n \in \bbN$,
\begin{eqnarray}
{ \left| \vc{\pi} - \presub{n}\vc{\pi} \right| \vc{g}
\over \vc{\pi}\vc{g} }
&\le&
{1 \over \vc{\pi}\vc{g} \cdot \widehat{\vc{\pi}} \left(\vc{v}/\vc{f}\right)}
\left[
\left| \widehat{\vc{\pi}} - \presub{n}\widehat{\vc{\pi}} \right|
+
\left| \widehat{\vc{\pi}} - \presub{n}\widehat{\vc{\pi}} \right|
\left(\vc{v}/\vc{f}\right) 
{ 
\presub{n}\widehat{\vc{\pi}}
\over 
\presub{n}\widehat{\vc{\pi}} \left(\vc{v}/\vc{f}\right) 
}
\right]\vc{\Delta}_{\vc{v}/\vc{f}}\,\vc{g}
\nonumber
\\
&=&
{1 \over \widehat{\vc{\pi}}  \widehat{\vc{g}} }
\left[
\left| \widehat{\vc{\pi}} - \presub{n}\widehat{\vc{\pi}} \right|
+
\left| \widehat{\vc{\pi}} - \presub{n}\widehat{\vc{\pi}} \right|
\left(\vc{v}/\vc{f}\right) 
{ 
\presub{n}\widehat{\vc{\pi}}
\over 
\presub{n}\widehat{\vc{\pi}} \left(\vc{v}/\vc{f}\right) 
}
\right] \widehat{\vc{g}}
\nonumber
\\
&=&
{ \left| \widehat{\vc{\pi}} - \presub{n}\widehat{\vc{\pi}} \right|
\widehat{\vc{g}}
\over \widehat{\vc{\pi}}  \widehat{\vc{g}} } 
+ 
{ \left| \widehat{\vc{\pi}} - \presub{n}\widehat{\vc{\pi}} \right|
\left(\vc{v}/\vc{f}\right) 
\over \presub{n}\widehat{\vc{\pi}} \left(\vc{v}/\vc{f}\right) } 
{ \presub{n}\widehat{\vc{\pi}}\widehat{\vc{g}}
\over \widehat{\vc{\pi}}  \widehat{\vc{g}} }
\nonumber
\\
&=&
{ \left| \widehat{\vc{\pi}} - \presub{n}\widehat{\vc{\pi}} \right|
\widehat{\vc{g}}
\over \widehat{\vc{\pi}}  \widehat{\vc{g}} } 
+ 
{ \left| \widehat{\vc{\pi}} - \presub{n}\widehat{\vc{\pi}} \right|
\left(\vc{v}/\vc{f}\right) 
\over \widehat{\vc{\pi}} \left(\vc{v}/\vc{f}\right) } 
\left(
{ \widehat{\vc{\pi}} \left(\vc{v}/\vc{f}\right)
\over \presub{n}\widehat{\vc{\pi}} \left(\vc{v}/\vc{f}\right) }
{ \presub{n}\widehat{\vc{\pi}}\widehat{\vc{g}}
\over \widehat{\vc{\pi}}  \widehat{\vc{g}} }
\right). \qquad
\label{bound-pi-g-04}
\end{eqnarray}
Note here that $\red{\vc{e} \le\,\,} \widehat{\vc{g}} \le \vc{v}$ and $\vc{0} <
\vc{v}/\vc{f} \le \vc{v}$ (due to $\vc{f} \ge \vc{e}$). Thus,
(\ref{ineqn-hat{E}(n)}) yields
\begin{eqnarray}
{ \left| \widehat{\vc{\pi}} - \presub{n}\widehat{\vc{\pi}} \right|
\widehat{\vc{g}}
\over
\widehat{\vc{\pi}}\widehat{\vc{g}}
} \le \widehat{E}(n),
\qquad
{ \left| \widehat{\vc{\pi}} - \presub{n}\widehat{\vc{\pi}} \right|
\left(\vc{v}/\vc{f}\right)
\over
\widehat{\vc{\pi}}\left(\vc{v}/\vc{f}\right)
}
\le \widehat{E}(n),
\qquad n \in \bbN.
\label{add-eqn-150630-04}
\end{eqnarray}
Applying (\ref{add-eqn-150630-04}) to (\ref{bound-pi-g-04}), we
obtain, for all $n \in \bbN$ and $\red{\vc{e} \le\,\,} \vc{g} \le \vc{f}$,
\begin{eqnarray}
{ \left| \vc{\pi} - \presub{n}\vc{\pi} \right| \vc{g}
\over \vc{\pi}\vc{g} }
&\le&
\widehat{E}(n)
\left(
1 + 
{ \widehat{\vc{\pi}} \left(\vc{v}/\vc{f}\right)
\over \presub{n}\widehat{\vc{\pi}} \left(\vc{v}/\vc{f}\right) }
{ \presub{n}\widehat{\vc{\pi}}\widehat{\vc{g}}
\over \widehat{\vc{\pi}}  \widehat{\vc{g}} }
\right).
\label{add-eqn-150708-01}
\end{eqnarray}
Therefore, if $\vc{g} = \vc{e}$, i.e., $\widehat{\vc{g}} =
\vc{v}/\vc{f}$, then (\ref{add-eqn-150708-01}) is reduced to
(\ref{bound-pi-g-reduction-01}).

Next, we prove (\ref{bound-pi-g-reduction-02}). To this
end, we estimate the term
\[
{ \widehat{\vc{\pi}} \left(\vc{v}/\vc{f}\right)
\over \presub{n}\widehat{\vc{\pi}} \left(\vc{v}/\vc{f}\right) }
{ \presub{n}\widehat{\vc{\pi}}\widehat{\vc{g}}
\over \widehat{\vc{\pi}}  \widehat{\vc{g}} },\qquad n \in \bbN.
\]
From (\ref{add-eqn-150630-04}), we have
\begin{eqnarray}
{\presub{n}\widehat{\vc{\pi}}\widehat{\vc{g}} 
\over \widehat{\vc{\pi}}\widehat{\vc{g}}} 
&\le& 1 + \widehat{E}(n),
\qquad
{ \presub{n}\widehat{\vc{\pi}} \left(\vc{v}/\vc{f}\right)
\over \widehat{\vc{\pi}}\left(\vc{v}/\vc{f}\right) }
\ge 1 - \widehat{E}(n) \vmin 1,\qquad n \in \bbN.
\label{add-eqn-150708-04}
\end{eqnarray}
Furthermore, from (\ref{defn-C_{f/v}}) and $\vc{f} \ge \vc{e}$, we
have
\begin{equation}
\vc{v} 
\ge \vc{v}/\vc{f} 
\ge  {1 \over  \overline{C}_{\vc{f}/\vc{v}} } \vc{e}.
\label{ineqn-v-v/f}
\end{equation}
Using (\ref{ineqn-wh{pi}*v}) and (\ref{ineqn-v-v/f}), we obtain
\begin{equation}
{ \presub{n}\widehat{\vc{\pi}} \left(\vc{v}/\vc{f}\right)
\over \widehat{\vc{\pi}}\left(\vc{v}/\vc{f}\right) }
\ge {1 \over \overline{C}_{\vc{f}/\vc{v}}}
 { \presub{n}\widehat{\vc{\pi}} \vc{e}
\over \widehat{\vc{\pi}} \vc{v} }
\ge { 1 \over \widehat{b} \overline{C}_{\vc{f}/\vc{v}} },
\qquad n \in \bbN.
\label{add-ineqn-160207-01}
\end{equation}
Combining (\ref{add-eqn-150708-04}) and (\ref{add-ineqn-160207-01}) yields
\begin{eqnarray}
{ \widehat{\vc{\pi}} \left(\vc{v}/\vc{f}\right)
\over \presub{n}\widehat{\vc{\pi}} \left(\vc{v}/\vc{f}\right) }
{ \presub{n}\widehat{\vc{\pi}}\widehat{\vc{g}}
\over \widehat{\vc{\pi}}  \widehat{\vc{g}} }
\le 
{ 1 + \widehat{E}(n)
\over \big( 1 - \widehat{E}(n) \vmin 1 \big)
\vmax \big( \,\widehat{b} \overline{C}_{\vc{f}/\vc{v}} \big)^{-1} 
},\qquad n \in \bbN.
\label{add-eqn-150708-05}
\end{eqnarray}
Substituting (\ref{add-eqn-150708-05}) into (\ref{add-eqn-150708-01}),
we obtain (\ref{bound-pi-g-reduction-02}).

Finally, we prove (\ref{bound-pi-g-reduction-03}) under the additional
condition that $\vc{v}_{\overline{\bbF}_0}$ is level-wise
nondecreasing.  We fix $\widehat{\vc{g}} =
\vc{\Delta}_{\vc{v}/\vc{f}}\, \vc{g}$ and $\vc{e} \le \vc{g} \le
\vc{f}$. We then have $\vc{v}/\vc{f} \le \widehat{\vc{g}} \le \vc{v}$
and thus
\begin{equation}
{ 
\widehat{\vc{\pi}} \left(\vc{v}/\vc{f}\right)
\over 
\presub{n}\widehat{\vc{\pi}} \left(\vc{v}/\vc{f}\right) 
}
{ 
\presub{n}\widehat{\vc{\pi}}\widehat{\vc{g}}
\over 
\widehat{\vc{\pi}}  \widehat{\vc{g}} 
}
\le 
{ 
\widehat{\vc{\pi}} \wh{\vc{g}}
\over 
\presub{n}\widehat{\vc{\pi}} \left(\vc{v}/\vc{f}\right) 
}
{ 
\presub{n}\widehat{\vc{\pi}}\vc{v}
\over 
\widehat{\vc{\pi}}  \widehat{\vc{g}} 
}
=
{
\presub{n}\widehat{\vc{\pi}}\vc{v}
\over \presub{n}\widehat{\vc{\pi}} \left(\vc{v}/\vc{f}\right)
},\qquad n \in \bbN.
\label{ineqn-hat{pi}_v_f/hat{pi}_hat{g}}
\end{equation}
From (\ref{eqn-{n}hat{pi}}), we also have
\begin{eqnarray*}
\presub{n}\widehat{\vc{\pi}}\vc{v}
&=& 
{\presub{n}\vc{\pi}\vc{f} \over \presub{n}\vc{\pi} \left(\vc{f}/\vc{v}\right)},
\qquad
\presub{n}\widehat{\vc{\pi}} \left(\vc{v}/\vc{f}\right)
= {1 \over \presub{n}\vc{\pi} \left(\vc{f}/\vc{v}\right)},
\qquad n \in \bbN.
\end{eqnarray*}
Substituting these equations into
(\ref{ineqn-hat{pi}_v_f/hat{pi}_hat{g}}) and using
(\ref{ineqn-pi_f<b}) yields
\begin{equation}
{ 
\widehat{\vc{\pi}} \left(\vc{v}/\vc{f}\right)
\over 
\presub{n}\widehat{\vc{\pi}} \left(\vc{v}/\vc{f}\right) 
}
{ 
\presub{n}\widehat{\vc{\pi}}\widehat{\vc{g}}
\over 
\widehat{\vc{\pi}}  \widehat{\vc{g}} 
}
\le \presub{n}\vc{\pi}\vc{f} \le b,
\qquad n \in \bbN.
\label{ineqn-{n}hat{pi}_(v/f)}
\end{equation}
Combining (\ref{add-eqn-150708-01}) with (\ref{add-eqn-150708-05}) and
(\ref{ineqn-{n}hat{pi}_(v/f)}) leads to \red{(\ref{bound-pi-g-reduction-03}).}  \QED

\section{Application to Level-Dependent Quasi-Birth-and-Death Processes}\label{sec-application}

In this section, we first establish a computable error bound for
LD-QBDs with exponential ergodicity by using the results in
Section~\ref{subsec-exp}. We then consider the queue length process in an
M/M/$s$ retrial queue, which is a special case of LD-QBDs. For this
special case, we derive two bounds: one includes $\presub{n}\vc{\pi}$
and the other does not. Using the two bounds, we present some
numerical examples.

\subsection{Numerical procedure for the error bound}

We assume that the infinitesimal generator $\vc{Q}$ of the
Markov chain $\{(X(t),J(t))\}$  has the following block-tridiagonal form:
\begin{equation}
\vc{Q} =
\bordermatrix{
               & \bbL_0 &  \bbL_1  &  \bbL_2       &  \bbL_3       & \cdots       
\cr
\bbL_0 & 
\vc{A}_0(0) & 
\vc{A}_0(1) &
\vc{O} &
\vc{O} &
\cdots
\cr
\bbL_1 & 
\vc{A}_1(-1) & 
\vc{A}_1(0) &
\vc{A}_1(1) &
\vc{O} &
\cdots
\cr
\bbL_2 & 
\vc{O}&
\vc{A}_2(-1) & 
\vc{A}_2(0) &
\vc{A}_2(1) &
\cdots
\cr
\bbL_3
& \vc{O}					
&
\vc{O}					
&
\vc{A}_3(-1)& 
\vc{A}_3(0) &
\ddots
\cr
~~\vdots    
& \vdots					
&
\vdots					&
\vdots					&
\ddots					&
\ddots
}.
\label{eqn-Q-LD-QBD}
\end{equation}
In this setting, $\{(X(t),J(t))\}$ is called the {\it level-dependent
  quasi-birth-and-death process (LD-QBD)} and $\vc{Q}$ is called the
{\it LD-QBD generator}.  Applying
Corollary~\ref{coro-exp-ergodic-comp} to $\vc{Q}$ in
(\ref{eqn-Q-LD-QBD}), we readily obtain the following result.
\begin{coro}\label{coro-LD-QBD}
Suppose that (i) $\vc{Q}$ in (\ref{eqn-Q-LD-QBD}) is irreducible and
its LC-block-augmented truncation $\presub{n}\vc{Q}$ has a single
communicating class in $\bbF_n$ for each $n \in \bbN$; and (ii) there
exist some $b>0$, $c > 0$, $K \in \bbZ_+$ and column vector $\vc{v}
\ge \vc{e}/c$ such that (\ref{geo-drift-cond}) holds. Furthermore, fix
$N \in \{K,K+1,\dots\}$ arbitrarily such that (\ref{ineqn-K})
holds. Under these conditions,
\begin{eqnarray}
\red{\sup_{\vc{e} \le \vc{g} \le c\vc{v}}}
{ \left| \vc{\pi} - \presub{n}\vc{\pi} \right| \vc{g} \over \vc{\pi}\vc{g} }
&\le& 2 \presub{n}\vc{\pi}(n)\vc{A}_n(1)
\nonumber
\\
&& {} \times 
\left[
\vc{v}(n) + \vc{v}(n+1)
+2b \left( {1 \over c}
+ {2 \over \beta \overline{\phi}_{K,N}^{(\beta)}}
\right)\vc{e} 
\right], \qquad n \in \bbN, \qquad
\label{bound-exp-LD-QBD}
\end{eqnarray}
where $\overline{\phi}_{K,N}^{(\beta)}$ is defined in (\ref{defn-phi_{K,N}}).
\end{coro}

Recall here that $\overline{\phi}_{K,N}^{(\beta)}$ is expressed in
terms of the fundamental matrix $\vc{\Phi}_{\bbF_N}^{(\beta)} =
(\vc{I} - \vc{Q}_{\bbF_N}/\beta)^{-1}$ of $\vc{I} -
\vc{Q}_{\bbF_N}/\beta$ (see (\ref{defn-S_N^(beta)}) and
(\ref{defn-phi_{K,N}})). Since $\vc{Q}_{\bbF_N}$ is block-tridiagonal,
we can efficiently compute $\vc{\Phi}_{\bbF_N}^{(\beta)} = (\vc{I} -
\vc{Q}_{\bbF_N}/\beta)^{-1}$ by \citet{Shin09}'s algorithm (see
Remark~\ref{rem-Le-Boud}). In addition, since $\presub{n}\vc{Q}$ is
block-tridiagonal in its unique communicating class $\bbF_n$, we can
compute its stationary distribution vector $\presub{n}\vc{\pi}$ in an
efficient way, which is described as follows.
\begin{prop}[\citet{Gave84}, Lemma 3]\label{prop-product-form}
For each $n \in \bbN$, let $\{\presub{n}\vc{R}_{\ell};\ell = 0,1,\dots, \break n-1\}$ denote a
sequence of $(S_{\ell \vmin 1}+1) \times (S_1+1)$ nonnegative matrices
defined recursively by
\begin{eqnarray*}
\presub{n}\vc{R}_{n-1}
&=& \vc{A}_{n-1}(1) \left(-\vc{A}_n(0) - \vc{A}_n(1)\right)^{-1},
\\
\presub{n}\vc{R}_{\ell}
&=& \vc{A}_{\ell}(1) 
\left( 
-\vc{A}_{\ell + 1}(0) - \presub{n}\vc{R}_{\ell + 1}\vc{A}_{\ell + 2}(-1)
\right)^{-1}, 
\qquad  \ell = n-2, n-3,\dots,0.
\label{recursion-LHA{R}_n}
\end{eqnarray*}
It then holds that, for $n \in \bbN$,
\begin{eqnarray*}
\presub{n}\vc{\pi}(0) 
\left(\vc{A}_0(0) + \presub{n}\vc{R}_0 \vc{A}_1(-1) \right)
&=& \vc{0},
\\
\presub{n}\vc{\pi}(0) \vc{e}
+ 
\presub{n}\vc{\pi}(0) 
\sum_{k=1}^n \prod_{\ell = 0}^{k-1}\presub{n}\vc{R}_{\ell}
\vc{e}
&=& 1,
\\
\presub{n}\vc{\pi}(k)
= \presub{n}\vc{\pi}(0) \prod_{\ell = 0}^{k-1}\presub{n}\vc{R}_{\ell},
\quad~~~  k &=& 1,2,\dots,n,
\end{eqnarray*}
where $\prod_{\ell = 0}^{k-1}\presub{n}\vc{R}_{\ell} =
\presub{n}\vc{R}_0 \cdot \presub{n}\vc{R}_1 \cdot \cdots \cdot
\presub{n}\vc{R}_{k-1}$ for $k=1,2,\dots,n$.
\end{prop}

We summarize the procedure for computing the bound
(\ref{bound-exp-LD-QBD}).
\begin{enumerate}
\item Find $b> 0$, $c> 0$, $K \in \bbZ_+$ and $\vc{v} \ge \vc{e}/c$
  such that (\ref{geo-drift-cond}) holds.
\item Fix $\beta > 0$ arbitrarily and find $N \in \{K,K+1,\dots\}$
  such that (\ref{ineqn-K}) holds; and compute
  $\vc{\Phi}_{\bbF_N}^{(\beta)}$ by \citet{Shin09}'s algorithm.
\item Compute $\overline{\phi}_{K,N}^{(\beta)}$ by (\ref{defn-phi_{K,N}}).
\item Compute $\presub{n}\vc{\pi}(n)$ according to
  Proposition~\ref{prop-product-form}.
\item Compute the bound (\ref{bound-exp-LD-QBD}).
\end{enumerate}

\subsection{Numerical example: M/M/$s$ retrial queue}

\subsubsection{Model description}

In this subsection, we consider an M/M/$s$ retrial queue, where $s$ is
a positive integer.  The system has $s$ identical servers but no
waiting room.  Customers arrive at the system according to a Poisson
process with rate $\lambda > 0$. Such customers are called {\it
  primary customers}. If a primary customer finds at least one server
idle, then the customer occupies one of them; otherwise joins the
orbit (virtual waiting room). The customers in the orbit are called
{\it retrial customers}.  Each retrial customer tries to occupy one of
idle servers after a random sojourn time in the orbit, which is
independent of the sojourn times of other retrial customers and is
distributed with an exponential distribution having mean
$1/\eta>0$. If a retrial customer is not accepted by any server (i.e.,
finds all the server busy), it goes back to the orbit and becomes a
retrial customer again.  Primary and retrial customers in service
leave the system after exponential service times with mean $1/\mu >
0$, which are independent one another.

Let $L(t)$, $t \ge 0$, denote the number of customers in the orbit,
called the queue length, at time $t$. Let $B(t)$, $t \ge 0$, denote
the number of busy servers at time $t$.  It is known (see, e.g.,
\citet{LiuBin10}) that $\{(L(t),B(t));t \ge 0\}$ is an LD-QBD whose
infinitesimal generator is given by $\vc{Q}$ in (\ref{eqn-Q-LD-QBD}),
where $\bbS_0=\bbS_1=\{0,1,\dots,s\}$, 
\begin{eqnarray}
\vc{A}_k(1)  & = & 
\left (
\begin{array}{cccccc}
0      & 0      & \cdots & 0      & 0 
\\
0      & 0      & \cdots & 0      & 0 
\\
\vdots & \vdots & \ddots & \vdots & \vdots 
\\
0      & 0      & \cdots & 0      & 0 
\\
0      & 0      & \cdots & 0      & \lambda  
\\
\end{array}
\right ),  
\quad 
\vc{A}_k(-1)   =  
\left (
\begin{array}{llllll}
0 \  & k \eta  \  & 0 & \cdots &  0 
\\
0 & 0 & k \eta \  &  \ddots &  \vdots 
\\
\vdots &  & \ddots & \ddots &   0 
\\
\vdots &   &  & 0 \  & k \eta  
\\
0 & \cdots & \cdots & 0 & 0 
\\
\end{array}
\right ), \qquad 
\label{eqn-A_k(1)}
\end{eqnarray}
and
\begin{eqnarray}
\vc{A}_k(0)  
&=& 
\left (
\begin{array}{cccccc}
-\psi_{k,0} \ & \lambda\ & 0 & \cdots   &  \cdots & 0 
\\
\mu   & -\psi_{k,1}  & \lambda  & \ddots   &   & \vdots 
\\
0  & 2 \mu \ & -\psi_{k,2} \ & \ddots &  \ddots & \vdots 
\\
\vdots & \ddots & \ddots & \ddots & \ddots & 0 
\\
\vdots & & \ddots & \ddots \quad & -\psi_{k,s-1} \quad & \lambda 
\\
0 & \cdots & \cdots & 0 &  s \mu \ & -\psi_{k,s}
\\
\end{array}
\right ),
\label{eqn-A_k(0)}
\end{eqnarray}
with
\begin{align*}
\psi_{k,i}
&= \lambda + i \mu + k\eta, & k &\in \bbZ_+,\ i = 0,1,\dots,s-1,
\\
\psi_{k,s}
&= \lambda + s \mu, & k &\in \bbZ_+.
\end{align*}

In the rest of this section, we assume that $\vc{Q}$ is the
infinitesimal generator of the LD-QBD $\{(L(t),B(t));t \ge 0\}$, i.e.,
the LD-QBD generator given by (\ref{eqn-Q-LD-QBD}) together with
(\ref{eqn-A_k(1)}) and (\ref{eqn-A_k(0)}). Thus, $\vc{Q}$ is not
uniformizable because its diagonal elements are unbounded. Therefore,
the existing results on discrete-time Markov chains (see
\citet{Herv14,LiuYuan10,Masu15-ADV,Masu16-SIAM,Twee98}) are not
applicable to the LD-QBD generator $\vc{Q}$ considered here.

We first that condition~(i) of Corollary~\ref{coro-LD-QBD} is
satisfied.  We then define $\rho = \lambda / (s\mu)$ and assume $\rho
< 1$. It thus follows that the LD-QBD generator $\vc{Q}$
(equivalently, the LD-QBD $\{(L(t),B(t))\}$) is ergodic (see, e.g.,
\citet[Section~2.2]{Fali97}) and therefore has the unique stationary
distribution vector, denoted by
$\vc{\pi}=(\vc{\pi}(0),\vc{\pi}(1),\dots)$. By definition,
\[
\pi(k,i)
= \lim_{t \to \infty} \PP(L(t)=k, B(t)=i),
\qquad k \in \bbZ_+,\ i = 0,1,\dots,s.
\]

We now define $L$ and $B$ as random variables such that 
\begin{eqnarray*}
\PP(L = k, B=i) 
&=& \lim_{t \to \infty} \PP(L(t)=k, B(t)=i)
= \pi(k,i),
\qquad k \in \bbZ_+,\ i=0,1,\dots,s,
\end{eqnarray*}
where $L$ and $B$ can be interpreted as the queue length and the
number of busy servers, respectively, in steady state. We also define
$\presub{n}L$ and $\presub{n}B$, $n\in\bbN$, as random variables such
that
\begin{eqnarray*}
\PP(\presub{n}L = k, \presub{n}B = i) = \presub{n}\pi(k,i),
\qquad k \in \bbZ_+,\ i=0,1,\dots,s.
\end{eqnarray*}
We then consider $\EE[g(\presub{n}L,\presub{n}B)]$ as an approximation to
$\EE[g(L,B)]$, where $\EE[g(L,B)]$ is the time-averaged functional of the LD-QBD $\{(L(t),B(t));t \ge 0\}$.

\subsubsection{Error bounds for time-averaged functionals}

In what follows, we estimate the relative error of the approximation  
$\EE[g(\presub{n}L,\presub{n}B)]$ to the time-averaged functional
$\EE[g(L,B)]$, i.e.,
\[
{ \left| \EE[g(L,B)] - \EE[g(\presub{n}L,\presub{n}B)] \right| 
\over \EE[g(L,B)] }.
\]
Note here that if $\vc{g}=\vc{e}$ then $\EE[g(L,B)] = \EE[L]$, which
is equal to the mean queue length in steady state. Note also that
\begin{eqnarray}
\red{\sup_{\vc{e} \le \vc{g} \le c\vc{v}}}
{ \left| \EE[g(L,B)] - \EE[g(\presub{n}L,\presub{n}B)] \right| 
\over \EE[g(L,B)] }
\le \red{\sup_{\vc{e} \le \vc{g} \le c\vc{v}}}
{
\left| \vc{\pi} - \presub{n}\vc{\pi} \right| \vc{g}
\over
\vc{\pi}\vc{g}
}.
\label{ineqn-relative-error}
\end{eqnarray}
Therefore, once we can establish the exponentially drift condition
(\ref{geo-drift-cond}), we can use Corollary~\ref{coro-LD-QBD} to
estimate the relative error of $\EE[g(\presub{n}L,\presub{n}B)]$.

The following lemma leads to the exponentially drift condition
(\ref{geo-drift-cond}).
\begin{lem}\label{lem-MMC}
Let $\vc{Q}$ be given by (\ref{eqn-Q-LD-QBD}) together with
(\ref{eqn-A_k(1)}) and (\ref{eqn-A_k(0)}).  Suppose $\rho = \lambda /
(s\mu) < 1$ and let $\acute{\vc{v}}:=(\acute{v}(k,i))_{(k,i)\in\bbF}$
be given by
\begin{equation}
\acute{v}(k,i)
= \left\{
\begin{array}{ll}
\alpha^k, & k \in \bbZ_+,\ i = 0,1,\dots,s-1,
\\
\gamma^{-1}\alpha^k,& k \in \bbZ_+,\ i = s,
\end{array}
\right.
\label{defn-widetilde{v}(k)-queue}
\end{equation}
where $\alpha$ and $\gamma$ are positive constants such that
\begin{eqnarray}
1 &<& \alpha < \rho^{-1},
\label{defn-theta}
\\
\alpha^{-1}
&<& \gamma
< 1 - \rho(\alpha-1).
\label{defn-gamma}
\end{eqnarray}
Furthermore, let 
\begin{eqnarray}
c &=& s\mu \left[ 1 - \rho(\alpha-1) - \gamma \right],
\label{eqn-c}
\\
\acute{b} &=& \max_{0 \le k \le K}\alpha^k
\left[c - 
\left\{ k\eta ( 1 - \gamma^{-1}\alpha^{-1}) + \lambda ( 1 - \gamma^{-1} )
\right\}
\right] \vmax 0,
\label{eqn-widetilde{b}}
\\
K &=& 
\left\lceil 
{ c + \lambda ( \gamma^{-1} - 1) \over \eta ( 1 - \gamma^{-1}\alpha^{-1}) } 
\right\rceil \vmax 1 - 1.
\label{eqn-eta}
\end{eqnarray}
Under these conditions,
\begin{equation}
\vc{Q}\acute{\vc{v}} 
\le - c\acute{\vc{v}} + \acute{b} \vc{1}_{\bbF_K}.
\label{ineqn-Q_widetilde{v}}
\end{equation}

\end{lem}

\medskip
\noindent
{\it Proof.~}
We first confirm that there exist constants $\alpha$ and $\gamma$ such
that (\ref{defn-theta}) and (\ref{defn-gamma}) hold. A positive
constant $\gamma$ satisfying (\ref{defn-gamma}) exists if
\[
\alpha^{-1}
< 1 - \rho(\alpha-1),\qquad \alpha > 1,
\]
or equivalently,
\begin{equation}
\rho \alpha^2 
- \left(\rho + 1 \right)\alpha
+ 1 = (\alpha - 1)(\rho\alpha - 1) < 0, \qquad \alpha > 1.
\label{ineqn-theta}
\end{equation}
Clearly, (\ref{ineqn-theta}) is equivalent to (\ref{defn-theta}).
Therefore, there exist positive constants $\alpha$ and $\gamma$
satisfying (\ref{defn-theta}) and (\ref{defn-gamma}).

Next we prove that (\ref{ineqn-Q_widetilde{v}}) holds.
For $k \in \bbZ_+$,
let $\vc{u}(k) := (u(k,i))_{i\in\{0,1,\dots,s\}}$ denote
\begin{eqnarray}
\vc{u}(k)
&=& \sum_{\ell=0}^{\infty}\vc{Q}(k;\ell)\acute{\vc{v}}(\ell)
\nonumber
\\
&=& 
\vc{A}_k(-1)\acute{\vc{v}}(k-1) 
+ \vc{A}_k(0)\acute{\vc{v}}(k)
+ \vc{A}_k(1)\acute{\vc{v}}(k+1),\qquad k\in\bbZ_+,
\label{defn-u(k)}
\end{eqnarray}
where $\acute{\vc{v}}(k) = (\acute{v}(k,i))_{i\in\{0,1,\dots,s\}}$ for
$k \in \bbZ_+$. Thus, it suffices to show that
\begin{equation}
\vc{u}(k)
\le
\left\{
\begin{array}{ll}
-c \acute{\vc{v}}(k) + \acute{b}\vc{e}, & k = 0,1,\dots,K,
\\
-c \acute{\vc{v}}(k), & k = K + 1, K+2,\dots.
\end{array}
\right.
\label{ineqn-vc{u}(k)}
\end{equation}

It follows from (\ref{eqn-A_k(1)}), (\ref{eqn-A_k(0)}),
(\ref{defn-widetilde{v}(k)-queue}) and (\ref{eqn-c}) that, for $k \in
\bbZ_+$,
\begin{eqnarray}
u(k,s)
&=& 
s\mu \alpha^k - \psi_{k,s} \gamma^{-1}\alpha^k
+ \lambda \gamma^{-1}\alpha^{k+1}
\nonumber
\\
&=& \left\{ s\mu (\gamma - 1) + \lambda (\alpha - 1) \right\}
 \gamma^{-1}\alpha^k
\nonumber
\\
&=& -s\mu
\left\{ 
1  - \gamma - { \lambda \over s\mu }(\alpha-1)
\right\} \cdot \gamma^{-1} \alpha^k
\nonumber
\\
&=& -s\mu
\left\{ 1 - \rho(\alpha-1) - \gamma \right\} \cdot \gamma^{-1} \alpha^k
\nonumber
\\
&=& - c \cdot \gamma^{-1} \alpha^k,
\label{eqn-x(n)-03}
\end{eqnarray}
and 
\begin{eqnarray}
u(k,s-1)
&=&
k\eta \gamma^{-1} \alpha^{k-1}
+ \{ (s-1) \mu - \psi_{k,{s-1}} \} \alpha^k
+ \lambda \gamma^{-1} \alpha^k
\nonumber
\\
&=&
\left\{ 
k\eta ( \gamma^{-1}\alpha^{-1} - 1 ) + \lambda ( - 1 + \gamma^{-1} ) 
\right\} \cdot \alpha^k
\nonumber
\\
&=&
-
\left\{ 
k\eta ( 1 - \gamma^{-1}\alpha^{-1}) + \lambda ( 1 - \gamma^{-1} ) 
\right\} \cdot \alpha^k,
\label{eqn-x(n)-02}
\\
u(k,i)
&=&
k\eta \alpha^{k-1}
+ (i\mu - \psi_{k,i} + \lambda )\alpha^k
\nonumber
\\
&=& - k\eta(1 - \alpha^{-1}) \cdot \alpha^k,
\qquad\qquad\qquad~~ i=0,1,\dots,s-2.
\label{eqn-x(n)-01}
\end{eqnarray}
Since $0 < \gamma < 1$ (see (\ref{defn-theta}) and
(\ref{defn-gamma})),
\[
k\eta ( 1 - \gamma^{-1}\alpha^{-1}) + \lambda ( 1 - \gamma^{-1} ) 
\le k\eta(1 - \alpha^{-1}).
\]
Therefore, from (\ref{eqn-x(n)-02}) and (\ref{eqn-x(n)-01}), we have
\begin{equation}
u(k,i) \le -
\left\{ 
k\eta ( 1 - \gamma^{-1}\alpha^{-1}) + \lambda ( 1 - \gamma^{-1} ) 
\right\} \cdot \alpha^k,
\qquad k \in \bbZ_+,\ i = 0,1,\dots,s-1.
\label{ineqn-u(n)}
\end{equation}
Note here that (\ref{eqn-eta}) implies
\begin{equation}
k\eta ( 1 - \gamma^{-1}\alpha^{-1}) + \lambda ( 1 - \gamma^{-1} ) 
\ge c\quad \mbox{for all $k = K+1,K+2,\dots$}.
\label{add-eqn-150710-01}
\end{equation}
Combining (\ref{ineqn-u(n)}) with (\ref{add-eqn-150710-01}) and using
(\ref{defn-widetilde{v}(k)-queue}) and (\ref{eqn-widetilde{b}}) yields
\begin{align}
&&&&
u(k,i) 
&\le - c \cdot \acute{v}(k,i), & k &=K+1,K+2,\dots, & i &= 0,1,\dots,s-1.
&&&&
\label{ineqn-x(n)-01}
\\
&&&&
u(k,i) 
&\le -c \cdot \acute{v}(k,i) + \acute{b}, 
& k &=0,1,\dots, K, & i &= 0,1,\dots,s-1.&&&&
\label{ineqn-x(n)-02}
\end{align}
Furthermore, applying (\ref{defn-widetilde{v}(k)-queue}) to
(\ref{eqn-x(n)-03}) leads to
\begin{equation}
u(k,s) \le - c \cdot \acute{v}(k,s), \qquad k \in \bbZ_+.
\label{ineqn-u(k,s)}
\end{equation}
As a result, from (\ref{ineqn-x(n)-01}), (\ref{ineqn-x(n)-02}) and
(\ref{ineqn-u(k,s)}), we obtain (\ref{ineqn-vc{u}(k)}). 
\QED

\medskip

Let $\vc{v}$ be given by
\begin{equation}
v(k,i)
= c^{-1}\acute{v}(k,i)
= \left\{
\begin{array}{ll}
\alpha^k/c, & k \in \bbZ_+,\ i = 0,1,\dots,s-1,
\\
\alpha^k/(c\gamma),& k \in \bbZ_+,\ i = s,
\end{array}
\right.
\label{defn-v(k)-queue}
\end{equation}
where $c$ is defined in (\ref{eqn-c}). Clearly, $\vc{v} \ge
\vc{e}/c$. Furthermore, from (\ref{eqn-widetilde{b}}) and
(\ref{ineqn-Q_widetilde{v}}), we have
\[
\vc{Q}\vc{v} 
\le - c\vc{v} + b \vc{1}_{\bbF_K},
\]
where
\begin{equation}
b 
= \acute{b}/c
=  \max_{0 \le k \le K} \alpha^k
\left[1 - c^{-1}
\left\{ k\eta ( 1 - \gamma^{-1}\alpha^{-1}) + \lambda ( 1 - \gamma^{-1} )
\right\}
\right] \vmax 0.
\label{eqn-b}
\end{equation}
Therefore, condition (ii) of Corollary~\ref{coro-LD-QBD} holds.

We now fix $N \in \{K,K+1,\dots\}$ arbitrarily such that
(\ref{ineqn-K}) holds. Thus, all the conditions of
Corollary~\ref{coro-LD-QBD} are satisfied.  It follows from
Corollary~\ref{coro-LD-QBD} and (\ref{ineqn-relative-error}) that
\begin{eqnarray}
\lefteqn{
\red{\sup_{\vc{e} \le \vc{g} \le c\vc{v}}}
{ \left| \EE[g(L,B)] - \EE[g(\presub{n}L,\presub{n}B)] \right| 
\over \EE[g(L,B)] }
}
\quad &&
\nonumber
\\
&\le& 2 \presub{n}\vc{\pi}(n)\vc{A}_n(1)
\left[
\vc{v}(n) + \vc{v}(n+1)
+ 2b \left( {1 \over c}
+ {2 \over \beta \overline{\phi}_{K,N}^{(\beta)}}
\right)\vc{e} 
\right], \qquad n \in \bbN. \qquad
\label{bound-retrial-0}
\end{eqnarray}
Note here that
\begin{align}
&&&&&&&&
\vc{A}_k(1)
&= \vc{e}_s \vc{\lambda},
& k &\in \bbZ_+, &&&&&&&&
\label{eqn-A_k(1)-01}
\\
&&&&&&&&
\vc{v}(k) 
&= \alpha^k \vc{a},
&  k &\in \bbZ_+,&&&&&&&&
\label{eqn-A_k(1)-02}
\end{align}
where 
\begin{eqnarray}
\vc{e}_s^{\top} = (0,0,\dots,0,1),
\quad
\vc{\lambda} = (0,0,\dots,0,\lambda),
\quad
\vc{a}^{\top} = c^{-1}(1,1,\dots,1,\gamma^{-1}). \qquad
\label{defn-e_s-lmabda-a}
\end{eqnarray}
 Substituting (\ref{eqn-A_k(1)-01}) and (\ref{eqn-A_k(1)-02})
into (\ref{bound-retrial-0}), we obtain the following bound:
\begin{eqnarray}
\lefteqn{
\red{\sup_{\vc{e} \le \vc{g} \le c\vc{v}}}
{ \left| \EE[g(L,B)] - \EE[g(\presub{n}L,\presub{n}B)] \right| 
\over \EE[g(L,B)] }
}
\quad &&
\nonumber
\\
&\le& 2 \presub{n}\vc{\pi}(n)\vc{e}_s \cdot \vc{\lambda}
\left[
(\alpha + 1) \alpha^n\vc{a}
+2b \left( {1 \over c}
+ {2 \over \beta \overline{\phi}_{K,N}^{(\beta)}}
\right)\vc{e} 
\right]
\nonumber
\\
&=&
{4 \lambda \over \gamma}
\left[
{ \alpha + 1 \over  2c}  
+ {\gamma b \over \alpha^n}\left( {1 \over c}
+ {2 \over \beta \overline{\phi}_{K,N}^{(\beta)}}
\right)
\right]\presub{n}\pi(n,s)\alpha^n,
\qquad n \in \bbN,\qquad
\label{bound-retrial-1}
\end{eqnarray}
where $c$, $b$ and $K$ are given in (\ref{eqn-c}), (\ref{eqn-b}) and
(\ref{eqn-eta}), respectively, and where $\alpha$ and $\gamma$ are
positive constants that satisfy (\ref{defn-theta}) and
(\ref{defn-gamma}).  Recall here that $\presub{n}\vc{\pi}(n)$ can be
computed through $\{\presub{n}\vc{R}_{\ell};\ell = 0,1,\dots,n-1\}$ (see
Proposition~\ref{prop-product-form}). Owing to (\ref{eqn-A_k(1)-01}),
the recursion of $\{\presub{n}\vc{R}_{\ell}\}$ is rewritten as follows: For $n \in \bbN$,
\begin{align*}
\presub{n}\vc{R}_{\ell}
&= \vc{e}_s \cdot \presub{n}\vc{\xi}_{\ell},
& \ell &=  0,1,\dots,n-1,
\\
\presub{n}\vc{\xi}_{n-1} 
&= \vc{\lambda} \left( -\vc{A}_n(0) - \vc{e}_s\vc{\lambda} \right)^{-1},
\\
\presub{n}\vc{\xi}_{\ell} 
&= \vc{\lambda}
\left( -\vc{A}_{\ell + 1}(0) 
- \vc{e}_s \cdot \presub{n}\vc{\xi}_{\ell + 1}\vc{A}_{\ell + 2}(-1)
\right)^{-1}, 
& \ell &= n-2,n-3,\dots,0.
\end{align*}
Therefore, the cost of computing $\presub{n}\vc{\pi}(n)$ is somewhat
reduced.

In what follows, we derive a computable bound without
$\presub{n}\pi(n,s)$ by using
Corollary~\ref{coro-exp-ergodic-comp-02}. To this end, we fix
\begin{equation}
v^{\sharp}(k,i)
= \left\{
\begin{array}{ll}
(\alpha^{\sharp})^k, & k \in \bbZ_+,\ i = 0,1,\dots,s-1,
\\
(\alpha^{\sharp})^k/\gamma^{\sharp},& k \in \bbZ_+,\ i = s,
\end{array}
\right.
\label{defn-v^{sharp}(k)-queue}
\end{equation}
where $\alpha^{\sharp}$ and $\gamma^{\sharp}$ are positive constants
such that
\begin{eqnarray}
1 &<& \alpha < \alpha^{\sharp} < \rho^{-1},
\label{defn-alpha^{sharp}}
\\
1/\alpha^{\sharp}
&<& \gamma^{\sharp}
< 1 - \rho(\alpha^{\sharp}-1).
\label{defn-gamma^{sharp}}
\end{eqnarray}
We also fix
\begin{eqnarray}
f^{\sharp}(k,i)
&=& c^{\sharp} v^{\sharp}(k,i),
\qquad  (k,i) \in \bbF,
\label{eqn-f^{sharp}}
\\
c^{\sharp} 
&=& s\mu \left[ 1 - \rho(\alpha^{\sharp}-1) - \gamma^{\sharp} \right],
\label{eqn-c^{sharp}}
\\
b^{\sharp} &=& \max_{0 \le k \le K^{\sharp}} (\alpha^{\sharp})^k
\left[c^{\sharp} - 
\left\{ k\eta \left( 1 - { 1 \over \gamma^{\sharp}\alpha^{\sharp} } \right) 
+ \lambda ( 1 - 1/\gamma^{\sharp} )
\right\}
\right] \vmax 0, \qquad
\label{eqn-b^{sharp}}
\\
K^{\sharp} &=& 
\left\lceil 
{ c^{\sharp} + \lambda ( 1/\gamma^{\sharp} - 1) 
\over \eta \{ 1 -  1/(\gamma^{\sharp}\alpha^{\sharp}) \} } 
\right\rceil \vmax 1
- 1.
\label{eqn-eta^{sharp}}
\end{eqnarray}
It then follows from Lemma~\ref{lem-MMC} that
\begin{equation*}
\vc{Q}\vc{v}^{\sharp} 
\le -c^{\sharp}\vc{v}^{\sharp} + b^{\sharp} \vc{1}_{\bbF_{K^{\sharp}}}
= -\vc{f}^{\sharp} + b^{\sharp} \vc{1}_{\bbF_{K^{\sharp}}}.
\label{ineqn-Qv^{sharp}-retrial}
\end{equation*}
Note here that the subvectors $\vc{v}_{\overline{\bbF}_0}$ and
$\vc{v}_{\overline{\bbF}_0}^{\sharp}$ of $\vc{v}$ and
$\vc{v}^{\sharp}$ in (\ref{defn-v(k)-queue}) and
(\ref{defn-v^{sharp}(k)-queue}), respectively, are level-wise
nondecreasing. As a result, Condition~\ref{cond-two-solutions-01} is
satisfied.

Next we confirm that Condition~\ref{cond-two-solutions-02} is
satisfied, in order to use Corollary~\ref{coro-exp-ergodic-comp-02}.  Let
$V$ and $T$ be positive functions on $[0,\infty)$ such that
\begin{eqnarray}
V(x) = \alpha^x,
\qquad 
T(x) = \left( {\alpha^{\sharp} \over \alpha} \right)^x,
\qquad x \ge 0.
\label{defn-V-T}
\end{eqnarray}
Thus, (\ref{defn-v(k)-queue}) and (\ref{defn-e_s-lmabda-a}) yield
(\ref{explicit-cond-01}). Furthermore, $V$ and $T$ are log-subadditive
and $\lim_{x\to\infty}V(x) =\lim_{x\to\infty}T(x) = \infty$
(therefore, (\ref{explicit-cond-03}) holds). From
(\ref{defn-v^{sharp}(k)-queue}), (\ref{eqn-f^{sharp}}) and
(\ref{defn-V-T}), we have

\begin{eqnarray}
\sup_{(k,i)\in\bbF}{T(k)V(k) \over f^{\sharp}(k,i)}
= \sup_{(k,i)\in\bbF}{T(k)V(k) \over c^{\sharp}v^{\sharp}(k,i)}
= {1 \over c^{\sharp}} =: r_0^{\sharp}.
\label{retrial-r_0^{sharp}}
\end{eqnarray}
From (\ref{eqn-Q-LD-QBD}), (\ref{eqn-A_k(1)-01}),
(\ref{defn-e_s-lmabda-a}) and (\ref{defn-V-T}), we also have
\begin{eqnarray}
\lefteqn{
\sup_{k,\ell\in\bbZ_+} T(\ell)
\left\| \sum_{m = \ell + 1}^{\infty}
\vc{Q}(k;k+m)V(m)\vc{a} \right\|_{\infty}
}
\quad &&
\nonumber
\\
&=&  T(0)V(1)\sup_{k\in\bbZ_+}
\left\| \vc{A}_k(1)\vc{a} \right\|_{\infty}
= \alpha \left\| \vc{e}_s \vc{\lambda}\vc{a} \right\|_{\infty}
= { \alpha\lambda \over c\gamma} =: r_1^{\sharp}.
\label{retrial-r_1^{sharp}}
\end{eqnarray}
As a result, Condition~\ref{cond-two-solutions-02} holds.

We are ready to use Corollary~\ref{coro-exp-ergodic-comp-02}. We set
$\underline{a} = c^{-1}$ according to (\ref{defn-underline{a}}) and
(\ref{defn-e_s-lmabda-a}).  Combining
Corollary~\ref{coro-exp-ergodic-comp-02} with
(\ref{ineqn-relative-error}), $\underline{a} = c^{-1}$ and
(\ref{defn-V-T})--(\ref{retrial-r_1^{sharp}}), we obtain
\begin{eqnarray}
\lefteqn{
\red{\sup_{\vc{e} \le \vc{g} \le c\vc{v}}}
{ \left| \EE[g(L,B)] - \EE[g(\presub{n}L,\presub{n}B)] \right| 
\over \EE[g(L,B)] }
}
\quad &&
\nonumber
\\
&\le&  
{4\alpha \lambda \over c\gamma}{ b^{\sharp} \over c^{\sharp}} 
\left( {\alpha \over \alpha^{\sharp}} \right)^n
\left[
1 + {cb \over \alpha^{n+1}}
\left(
{1 \over c}
+ { 2 \over \beta \overline{\phi}_{K,N}^{(\beta)} } \right)
\right] 
\nonumber
\\
&=& {4 \lambda \over \gamma}
\left[
{\alpha \over c} + {b \over \alpha^n}
\left(
{1 \over c}
+ { 2 \over \beta \overline{\phi}_{K,N}^{(\beta)} } \right)
\right] {b^{\sharp} \over c^{\sharp}} 
\left( {\alpha \over \alpha^{\sharp}} \right)^n,
\qquad n \in \bbN. \qquad
\label{bound-retrial-2}
\end{eqnarray}

Finally, we compare the two bounds (\ref{bound-retrial-1}) and (\ref{bound-retrial-2}), where the former includes $\presub{n}\pi(n,s)$ whereas the latter does not.  For simplicity, let
\begin{align}
&&
\wwtilde{E}_N(n)
&= {4 \lambda \over \gamma}
\left[
{ \alpha + 1 \over  2c}  
+ {\gamma b \over \alpha^n}\left( {1 \over c}
+ {2 \over \beta \overline{\phi}_{K,N}^{(\beta)}}
\right)
\right]\presub{n}\pi(n,s)\alpha^n,
& n &\in \bbN, &&
\label{defn-wwtilde{E}_N(n)}
\\
&&
\wwtilde{E}_N^{\sharp}(n)
&=
{4 \lambda \over \gamma}
\left[
{\alpha \over c} + {b \over \alpha^n}
\left(
{1 \over c}
+ { 2 \over \beta \overline{\phi}_{K,N}^{(\beta)} } \right)
\right] {b^{\sharp} \over c^{\sharp}} 
\left( {\alpha \over \alpha^{\sharp}} \right)^n,
&  n &\in \bbN, &&
\label{defn-wwtilde{E}_N^{ddag}(n)}
\end{align}
which are the error decay functions of the bounds
(\ref{bound-retrial-1}) and (\ref{bound-retrial-2}),
respectively. Note here that (\ref{ineqn-pi-f^{ddagger}}) holds in the present setting.
Using (\ref{ineqn-pi-f^{ddagger}}) and (\ref{eqn-f^{sharp}}), we have
\begin{eqnarray}
\presub{n}\pi(n,s)v^{\sharp}(n,s)
&=&
\presub{n}\pi(n,s)f^{\sharp}(n,s) / c^{\sharp}
\nonumber
\\
&\le& \sum_{k=0}^n \presub{n}\vc{\pi}(k)\vc{f}^{\sharp}(k) / c^{\sharp}
\le b^{\sharp}/c^{\sharp},\qquad n \in \bbN.
\label{ineqn-{n}pi(n,s)v^{sharp}(n,s)}
\end{eqnarray}
Combining (\ref{ineqn-{n}pi(n,s)v^{sharp}(n,s)}) with
(\ref{defn-v^{sharp}(k)-queue}) and $\gamma^{\sharp} < 1$ yields
\begin{eqnarray}
\presub{n}\pi(n,s) \alpha^n
&=& \presub{n}\pi(n,s) {(\alpha^{\sharp})^n \over \gamma^{\sharp}} 
\cdot \gamma^{\sharp} \left( {\alpha \over \alpha^{\sharp}} \right)^n
\nonumber
\\
&=& \presub{n}\pi(n,s) v^{\sharp}(n,s)
\cdot \gamma^{\sharp} \left( {\alpha \over \alpha^{\sharp}} \right)^n
< {b^{\sharp} \over c^{\sharp}}
\left( {\alpha \over \alpha^{\sharp}} \right)^n,\qquad n \in \bbN.
\label{ineqn-{n}pi(n,s)-alpha^n}
\end{eqnarray}
Substituting (\ref{ineqn-{n}pi(n,s)-alpha^n}), $\gamma < 1$ and
$\alpha > 1$ into (\ref{defn-wwtilde{E}_N(n)}) and using
(\ref{defn-wwtilde{E}_N^{ddag}(n)}) leads to
\begin{equation}
\wwtilde{E}_N(n) \le \wwtilde{E}_N^{\sharp}(n),
\qquad n \in \bbN.
\label{ineqn-wwtilde{E}_N(n)<=wwtilde{E}_N^{ddag}(n)}
\end{equation}
Consequently,
\[
\red{\sup_{\vc{e} \le \vc{g} \le c\vc{v}}}
{ \left| \EE[g(L,B)] - \EE[g(\presub{n}L,\presub{n}B)] \right| 
\over \EE[g(L,B)] }
\le \wwtilde{E}_N(n) 
\le \wwtilde{E}_N^{\sharp}(n),
\qquad n \in \bbN.
\]

\subsubsection{Numerical results and discussion}

First of all, we discuss the impact of $\alpha$ and $\alpha^{\sharp}$
on the error decay functions $\wwtilde{E}_N$ and
$\wwtilde{E}_N^{\sharp}$. According to
(\ref{defn-wwtilde{E}_N^{ddag}(n)}), the decay rate of
$\wwtilde{E}_N^{\sharp}$ is equal to $\alpha^{\sharp}/\alpha >
1$. Recall here that $\alpha$ and $\alpha^{\sharp}$ must satisfy the
constraint (\ref{defn-alpha^{sharp}}), i.e., $1 < \alpha <
\alpha^{\sharp} < \rho^{-1}$, which leads to
\begin{equation}
1 < {\alpha^{\sharp} \over \alpha } < \rho^{-1}.
\label{bounds-decay-rate}
\end{equation}
Clearly, the decay rate $\alpha^{\sharp}/\alpha$ of
$\wwtilde{E}_N^{\sharp}$ is larger (i.e.,
$\wwtilde{E}_N^{\sharp}$ decays more rapidly) as $\alpha$ is
smaller and/or $\alpha^{\sharp}$ is larger. However, it follows from
(\ref{defn-gamma}) and (\ref{eqn-c}) that if $\alpha \downarrow 1$
then $\gamma \uparrow 1$ and thus
\[
1/c \to \infty\quad \mbox{as $\alpha \downarrow 1$}.
\]
This result, in combination with (\ref{defn-wwtilde{E}_N(n)}) and
(\ref{ineqn-wwtilde{E}_N(n)<=wwtilde{E}_N^{ddag}(n)}), implies that
\begin{equation}
\wwtilde{E}_N(1) \to \infty~~
\mbox{and} ~~ \wwtilde{E}_N^{\sharp}(1) \to
\infty \quad  \mbox{as $\alpha \downarrow 1$}.
\label{lim-wwtilde{E}_N(1)}
\end{equation}
Similarly, it follows from
(\ref{defn-gamma^{sharp}}), (\ref{eqn-c^{sharp}}) and
(\ref{eqn-eta^{sharp}}) that if $\alpha^{\sharp} \uparrow \rho^{-1}$
then $\gamma^{\sharp} \downarrow \rho$, which causes
\[
1/c^{\sharp} \to \infty ~~
\mbox{and} ~~
K^{\sharp} \to \infty
\quad \mbox{as $\alpha^{\sharp} \uparrow \rho^{-1}$}.
\]
It is likely, from these facts and (\ref{eqn-b^{sharp}}), that the
factor $b^{\sharp}/c^{\sharp}$ of (\ref{defn-wwtilde{E}_N^{ddag}(n)})
diverges and thus $\wwtilde{E}_N^{\sharp}(1)$ does. In summary, the
decay rare and the initial value of the error decay function are in a
trade-off relation.

To support the above argument, we present Figures~\ref{fig-01} and \ref{fig-02} below. In the examples therein and all the
subsequent ones, we fix $s=\eta=50$, $\mu=1$ and
\begin{eqnarray*}
\gamma 
&=& {1 \over 2}\left[{1 \over \alpha} + \{1 - \rho(\alpha-1)\} \right],
\label{example-gamma}
\\
\gamma^{\sharp} 
&=& {1 \over 2}\left[{1 \over \alpha^{\sharp}} + \{1 - \rho(\alpha^{\sharp}-1)\} \right].
\label{example-gamma^{sharp}}
\end{eqnarray*}
Figure~\ref{fig-01} plots $\wwtilde{E}_N(1)$ with $\rho = 0.1, 0.5, 0.9,
0.95, 0.99$, as a function of $x \in (0,1)$,
where
\begin{eqnarray*}
\alpha &=& 1 + x (\rho^{-1} - 1),\qquad 0 < x < 1,
\\
\beta &=& 1, \qquad N=K+100.
\end{eqnarray*}
Figure~\ref{fig-02} plots $\wwtilde{E}_N^{\sharp}(1)$ with $\rho = 0.1,
0.5, 0.9, 0.95, 0.99$, as a function of $y \in (0,1)$, where
\begin{eqnarray*}
\alpha^{\sharp} &=& \alpha + y (\rho^{-1} - \alpha), \qquad 0 < y < 1,
\\
\alpha &=& 1 + 10^{-3},
\\
\beta &=& 1, \qquad N=K+100.
\end{eqnarray*}

\begin{figure}[htb]
\begin{center}
\includegraphics[bb=50 50 410 302, scale=0.7]{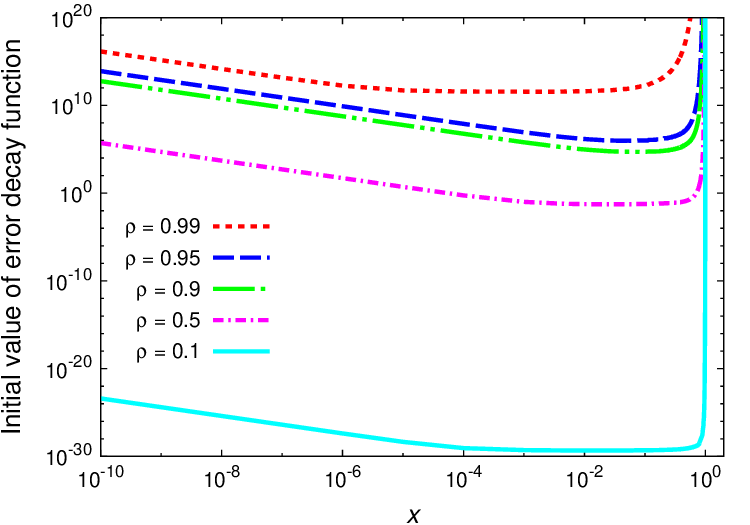}
\end{center}
%
\caption{Impact of $\alpha$ ($=1 + x (\rho^{-1} - 1)$) on initial value $\wwtilde{E}_N(1)$}
\label{fig-01}
\end{figure}

\begin{figure}[htb]
\begin{center}
\includegraphics[bb=50 50 410 302, scale=0.7]{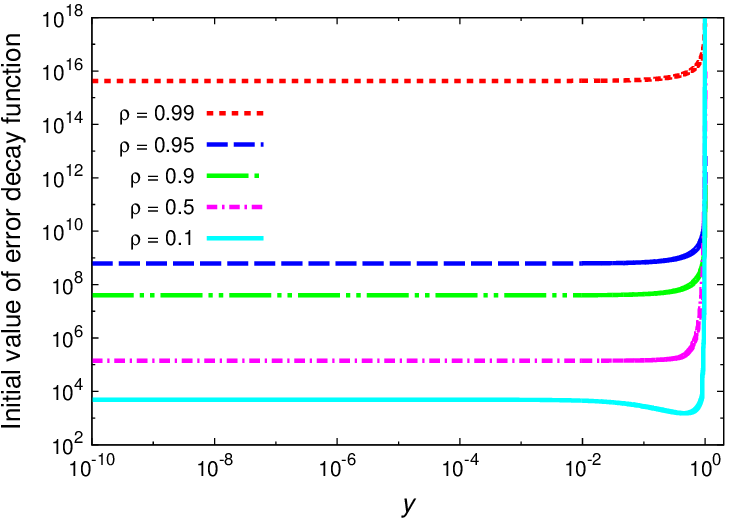}
\end{center}
%
\caption{Impact of $\alpha^{\sharp}$ ($=\alpha + y (\rho^{-1} - \alpha)$) on  initial value $\wwtilde{E}_N^{\sharp}(1)$ }
\label{fig-02}
\end{figure}
%
As expected, Figure~\ref{fig-01} shows that $\wwtilde{E}_N(1)$ increases
as $\alpha$ decreases toward one (i.e., $x$ decreases toward zero), and
Figure~\ref{fig-02} shows that $\wwtilde{E}_N^{\sharp}(1)$ increases as
$\alpha^{\sharp}$ increases toward $\rho^{-1}$ (i.e., $y$ increases
toward one).  Furthermore, we can see from Figure~\ref{fig-01} that
$\wwtilde{E}_N(1)$ rapidly increases as $\alpha$ increases toward
$\rho^{-1}$. This observation is justified as follows: It follows from
(\ref{defn-gamma}) and (\ref{eqn-c}) that if $\alpha \uparrow
\rho^{-1}$ then $\gamma \downarrow \rho$ and thus $1/c \to
\infty$. This result and (\ref{defn-wwtilde{E}_N(n)}) imply that
$\wwtilde{E}_N(1) \to \infty$ as $\alpha \uparrow \rho^{-1}$.

It should be noted that $\alpha = 1 + 10^{-3}$ in Figure~\ref{fig-02}, which
corresponds to $x= 10^{-3}/(\rho^{-1} - 1)$ in
Figure~\ref{fig-01}. Table~\ref{table-00} provides the values of $x$ for
which $\alpha = 1 + 10^{-3}$ in Figure~\ref{fig-01}.
%
\begin{table}[htb]
\begin{center}
\caption{Values of $x$ for which $\alpha = 1 + 10^{-3}$ in Figure~\ref{fig-01}}\label{table-00}
\begin{tabular}{|l||l|}
\hline
\rule[-5mm]{0mm}{12mm} $\rho$ & $x= \dm{ 10^{-3} \over \rho^{-1} - 1}$
\\
\hline
0.1  	& \rule{0mm}{4.5mm}$1.11\ol{1} \times 10^{-4}$ 
\\
0.5  	& 0.001   
\\
0.9  	& 0.009    
\\
0.95  	& 0.019      
\\
0.99  	& 0.099  
\\
\hline 
\end{tabular}
\end{center}
\end{table}
We can see from Figure~\ref{fig-01} and Table~\ref{table-00} that
$\wwtilde{E}_N(1)$ with $\alpha = 1 + 10^{-3}$ takes a value not much
different from the minimum for each $\rho = 0.1,0.5,0.9,095,0.99$. In
addition, $1 + 10^{-3}$ is close to one, i.e., the lower limit of
$\alpha$. Recall here that the decay rate $\alpha^{\sharp}/\alpha$ of
$\wwtilde{E}_N^{\sharp}$ is larger as $\alpha$ is smaller. Based on
these facts, we set $\alpha = 1 + 10^{-3}$ in the subsequent numerical
examples.

According to (\ref{defn-wwtilde{E}_N^{ddag}(n)}), we can expect that
the behavior of $\wwtilde{E}_N^{\sharp}$ is sensitive to the choice of
$\alpha^{\sharp}$, provided that $\alpha$ is fixed.  Thus, we observe
the impact of $\alpha^{\sharp}$ on the error decay function
$\wwtilde{E}_N^{\sharp}$. To this end, we define
\begin{equation*}
\alpha_i
= \alpha + {i \over 100}( \rho^{-1} - \alpha ),
\qquad i = 0,1,10,50,90, 99,
\end{equation*}
with $\alpha = 1 + 10^{-3}$.  We then denote by ``line $i$" the
$\wwtilde{E}_N(n)$'s with $\alpha=\alpha_i$ and denote by ``line
$(i,j)$" the $\wwtilde{E}_N^{\sharp}(n)$'s with
$(\alpha,\alpha^{\sharp})=(\alpha_i,\alpha_j)$. Furthermore, we fix
$\lambda=0.5s$ (thus $\rho=0.5$), $\beta=1$ and $N=K+100$.  In this
setting, Figure~\ref{fig-04} plots
\[
\mbox{lines $0, (0,1), (0,10), (0, 50), (0,90), (0,99)$}, 
\]
where line 0, i.e., the $\wwtilde{E}_N(n)$'s with $\alpha =
1+10^{-3}$, serves as the ``reference line" because the other lines
must be over line 0 due to
(\ref{ineqn-wwtilde{E}_N(n)<=wwtilde{E}_N^{ddag}(n)}).
%
%
\begin{figure}[htb]
\begin{center}
\includegraphics[bb=50 50 410 302, scale=0.7]{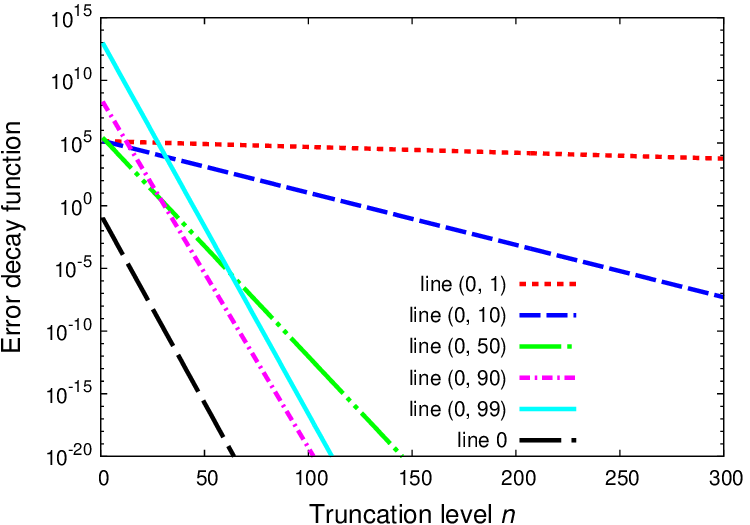}
\end{center}
%
\caption{Impact of $\alpha^{\sharp}$ on $\wwtilde{E}_N^{\sharp}(n)$}\label{fig-04}
\end{figure}
%
As shown in Figure~\ref{fig-04}, the choice of {\it large}
$\alpha^{\sharp}$ is basically better. Although the initial value of
line $(0,99)$ is larger than that of line $(0,90)$, the decay rate of
the former is larger than that of the latter and thus the two lines
cross over eventually. Anyway, for later discussion, we fix
$\alpha^{\sharp} = \alpha_{99}$.

Next, we discuss the impact of the traffic intensity $\rho$ on the
decay rates of the error decay functions $\wwtilde{E}_N$ and
$\wwtilde{E}_N^{\sharp}$. Inequality (\ref{bounds-decay-rate}) shows
that, as $\rho \uparrow 1$, the decay rate $\alpha^{\sharp}/\alpha$ of
$\wwtilde{E}_N^{\sharp}$ becomes smaller and thus that of
$\wwtilde{E}_N$ can be also smaller. In addition,
(\ref{defn-alpha^{sharp}}) shows that if $\rho \uparrow 1$ then
$\alpha \downarrow 1$, which leads to $\wwtilde{E}_N(1) \to \infty$
and $\wwtilde{E}_N^{\sharp}(1) \to \infty$ (see
(\ref{lim-wwtilde{E}_N(1)})). Consequently, as $\rho \uparrow 1$, the
decay rates of $\wwtilde{E}_N$ and $\wwtilde{E}_N^{\sharp}$ decrease
and their initial values $\wwtilde{E}_N(1)$ and
$\wwtilde{E}_N^{\sharp}(1)$ increase, which is a ``double whammy" for
the bounds (\ref{bound-retrial-1}) and (\ref{bound-retrial-2}).

To visualize the impact of the traffic intensity $\rho$ on the error
decay functions $\wwtilde{E}_N$ and $\wwtilde{E}_N^{\sharp}$, we
provide Figures~\ref{fig-05} and \ref{fig-06}, where $s=\eta=50$,
$\mu=1$, $\lambda=\rho s$, $\beta=1$ and $N=K+100$.
Figures~\ref{fig-05} and \ref{fig-06} plot lines 0 and (0,99),
respectively, for $\rho=0.1,0.5, 0.9, 0.95, 0.99$.%
\begin{figure}[htb]
\begin{center}
\includegraphics[bb=50 50 410 302, scale=0.7]{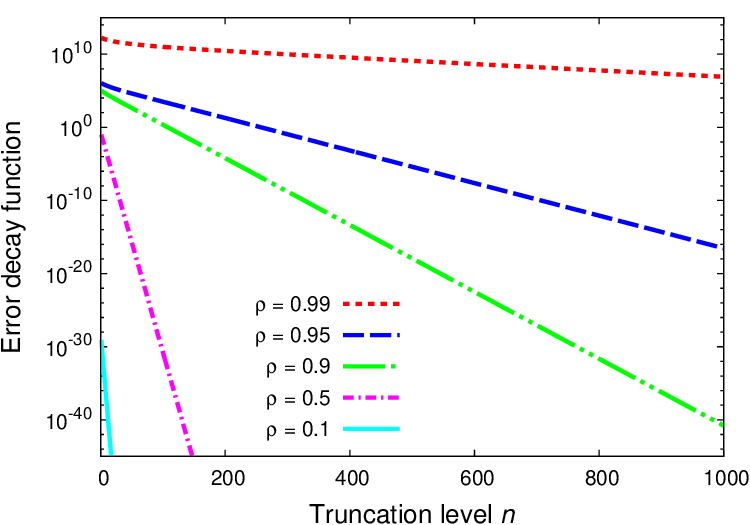}
\end{center}
%
\caption{Impact of traffic intensity $\rho$ on $\wwtilde{E}_N(n)$ with
  $\alpha=\alpha_0$}\label{fig-05}
\end{figure}
\begin{figure}[htb]
\begin{center}
\includegraphics[bb=50 50 410 302, scale=0.7]{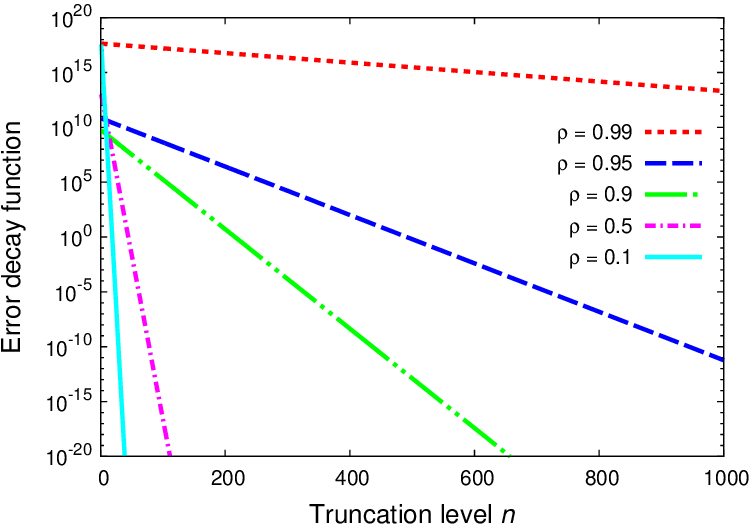}
\end{center}
%
\caption{Impact of traffic intensity $\rho$ on
  $\wwtilde{E}_N^{\sharp}(n)$ with
  $(\alpha,\alpha^{\sharp})=(\alpha_0,\alpha_{99})$}\label{fig-06}
\end{figure}
These two figures show that, in the case where $\rho
= 0.99$, the error decay functions $\wwtilde{E}_N$ and
$\wwtilde{E}_N^{\sharp}$ take extremely large values and yield useless
bounds in the region of the truncation level $n$ shown
therein. This is mainly because the common factor $\overline{\phi}_{K,N}^{(\beta)}$ of $\wwtilde{E}_N$ and
$\wwtilde{E}_N^{\sharp}$ (with
$\beta = 1$ in Figures~\ref{fig-05} and \ref{fig-06}) takes exceedingly
small values, as shown in Table~\ref{table-03}. Note here that
Table~\ref{table-03} presents the values of
$\overline{\phi}_{K,N}^{(1)}$ with $N=K+10, K+50, K+100, K+100, K+500$, which show the validity of our choice $N = K+100$ for computing
$\overline{\phi}_{K,N}^{(\beta)}$.

\begin{table}[htb]
\caption{Values of $K$ and $\overline{\phi}_{K,N}^{(1)}$ in the same
  setting as Figures~\ref{fig-05} and \ref{fig-06}}\label{table-03}
\begin{center}
\begin{tabular}{|l||l||l|l|l|l|}
\hline
\rule[-2mm]{0mm}{7.5mm}  &  & \multicolumn{4}{|c|}{$\overline{\phi}_{K,N}^{(1)}$} 
\\
\cline{1-6}
\rule[0mm]{0mm}{4mm} $\rho$ & $K$ & $N=K+10$ & $N=K+50$ & $N=K+100$ & $N=K+500$
\\
\hline
\rule[0mm]{0mm}{4mm}0.1 & 1    & $1.84\times10^{-2}$ & $1.84\times10^{-2}$ & $1.84\times10^{-2}$ & $1.84\times10^{-2}$
\\
0.5 & 2    & $1.79\times10^{-2}$ & $1.79\times10^{-2}$ & $1.79\times10^{-2}$ & $1.79\times10^{-2}$
\\
0.9 & 18   & $8.66\times10^{-3}$ & $8.66\times10^{-3}$ & $8.66\times10^{-3}$ & $8.66\times10^{-3}$
\\
0.95 & 38  & $1.48\times10^{-3}$ & $1.52\times10^{-3}$ & $1.52\times10^{-3}$ & $1.52\times10^{-3}$
\\
0.99 & 219 & $4.32\times10^{-9}$ & $4.52\times10^{-9}$ & $4.52\times10^{-9}$ & $4.52\times10^{-9}$
\\
\hline
\end{tabular}
\end{center}
\end{table}

We now discuss the impact of $\beta$ on the error decay functions
$\wwtilde{E}_N$ and $\wwtilde{E}_N^{\sharp}$. It follows from
(\ref{defn-S_N^(beta)}) and (\ref{defn-phi_{K,N}}) that if the minimum
element of each column of $\vc{\Phi}_{\bbF_N}^{(\beta)}$ in (\ref{defn-S_N^(beta)}) is
small then so is $\overline{\phi}_{K,N}^{(\beta)}$.  Since
$\vc{Q}_{\bbF_N}$ considered here is block-tridiagonal, there can be a
large variation in the elements of $\exp\{\vc{Q}_{\bbF_N} t\}$ for
small values of $t$. However, such a variation would become smaller as
$t$ increases, because $\vc{Q}_{\bbF_N}$ is irreducible.  Furthermore,
as $\beta$ is smaller, the integrand factor $\exp\{\vc{Q}_{\bbF_N}
t\}$ for large values of $t$ (that is, the right tail of this factor)
has a greater contribution to
$\vc{\Phi}_{\bbF_N}^{(\beta)}$. Therefore, we can expect that
$\overline{\phi}_{K,N}^{(\beta)}$ takes a large value if $\beta$ is
small. In addition, it is known that the queue length process reaches
the limiting state more slowly as $1 - \rho$ approaches to zero (see,
e.g., \citet{Door11,Kiji89,Kiji90}). As a result, it would be better
to decrease $\beta$ with $1 - \rho$ in order to keep the value of
$\overline{\phi}_{K,N}^{(\beta)}$ ``moderate". Indeed,
Table~\ref{table-04} shows that such choices of $\beta$ improve the
values of $\overline{\phi}_{K,N}^{(\beta)}$ for $\rho=0.99$, compared
to those of $\overline{\phi}_{K,N}^{(1)}$ in
Table~\ref{table-03}. Note here that Table~\ref{table-04} is provided
in the same setting as Figures~\ref{fig-05} and \ref{fig-06} except the
value of $\beta$.

\begin{table}[htb]
\caption{Impact of $\beta$ on $\overline{\phi}_{K,N}^{(\beta)}$}\label{table-04}
\begin{center}
\begin{tabular}{|l||l|l|l|l|}
\hline
\rule[-2mm]{0mm}{7.5mm}  &  \multicolumn{4}{|c|}{$\overline{\phi}_{K,N}^{(\beta)}$} 
\\
\hline
\rule[-1mm]{0mm}{5mm} $\rho$ & $\beta=(1-\rho)^{1/2}$ & $\beta=1-\rho$ & $\beta=(1-\rho)^2$ & $\beta=(1-\rho)^3$
\\
\hline
\rule[0mm]{0mm}{4mm}0.1 & $2.03\times10^{-2}$     & $2.23\times10^{-2}$ & $2.65\times10^{-2}$ & $3.09\times10^{-2}$
\\
0.5 & $2.70\times10^{-2}$     & $3.65\times10^{-2}$ & $5.34\times10^{-2}$ & $6.50\times10^{-2}$
\\
0.9 & $2.37\times10^{-2}$     & $3.70\times10^{-2}$ & $4.77\times10^{-2}$ & $4.92\times10^{-2}$
\\
0.95 & $8.87\times10^{-3}$    & $2.10\times10^{-2}$ & $3.11\times10^{-2}$ & $2.13\times10^{-2}$
\\
0.99 & $1.81\times10^{-4}$    & $2.11\times10^{-3}$ & $1.86\times10^{-3}$ & $2.67\times10^{-5}$
\\
\hline
 \end{tabular}
\end{center}

\medskip

\caption{Impact of $\beta$ on $1/(\beta \overline{\phi}_{K,N}^{(\beta)})$}\label{table-05}
\begin{center}
\begin{tabular}{|l||l|l|l|l|}
\hline
\rule[-2mm]{0mm}{7.5mm}  &  \multicolumn{4}{|c|}{$1/(\beta \overline{\phi}_{K,N}^{(\beta)})$} 
\\
\hline
\rule[-1mm]{0mm}{5mm} $\rho$ & $\beta=(1-\rho)^{1/2}$ 
& $\beta=1-\rho$ & $\beta=(1-\rho)^2$ & $\beta=(1-\rho)^3$
\\
\hline
\rule[0mm]{0mm}{4mm}0.1 & $5.20\times10^{1}$   & $4.99\times10^{1}$ & $4.66\times10^{1}$  & $4.44\times10^{1}$
\\
0.5 & $5.24\times10^{1}$   & $5.48\times10^{1}$ & $7.49\times10^{1}$  & $1.23\times10^2$
\\
0.9 & $1.34\times10^2$     & $2.70\times10^2$  & $2.10\times10^3$  & $2.03\times10^4$
\\
0.95 & $5.04\times10^2$    & $9.53\times10^2$  & $1.29\times10^4$  & $3.76\times10^5$
\\
0.99 & $5.52\times10^4$    & $4.74\times10^4$  & $5.39\times10^6$  & $3.74\times10^{10}$
\\
\hline
 \end{tabular}
\end{center}
\end{table}

We have to remark that the error decay functions $\wwtilde{E}_N$ and
$\wwtilde{E}_N^{\sharp}$ include a factor $1/(\beta
\overline{\phi}_{K,N}^{(\beta)})$ and thus the small value of $\beta$
does not necessarily yield tight bounds, as shown in
Table~\ref{table-05} provided in the same setting as
Table~\ref{table-04}.  It would not be easy to systematically find an
optimal value of $\beta$ such that $\wwtilde{E}_N$ and
$\wwtilde{E}_N^{\sharp}$ are minimized.  Anyway, we fix $\beta = 1 -
\rho$ and present Figure~\ref{fig-07}, which plots the
$\wwtilde{E}_N(n)$'s and the $\wwtilde{E}_N^{\sharp}(n)$'s in the same
setting as Figures~\ref{fig-05} and \ref{fig-06} except the value of
$\beta$. Obviously, for sufficiently large $n$'s, $\wwtilde{E}_N(n)$
and $\wwtilde{E}_N^{\sharp}(n)$ are so small that the obtained bounds
are practically useful even in the ``worst" case, where $\rho = 0.99$.
\begin{figure}[htb]
\begin{center}
\includegraphics[bb=50 50 410 302, scale=0.7]{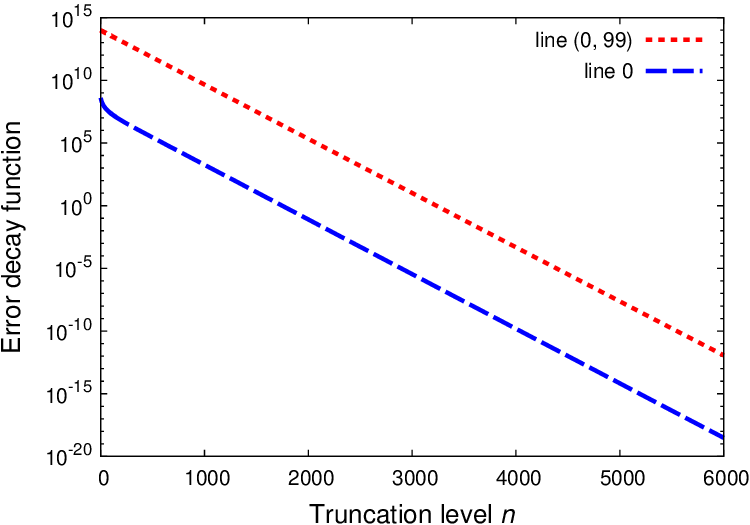}
\end{center}
%
\caption{Values of $\wwtilde{E}_N(n)$ and
  $\wwtilde{E}_N^{\sharp}(n)$ with $\rho=0.99$ and $\beta = 1 - \rho$}\label{fig-07}
\end{figure}

\section{Perturbation Bounds}\label{sec-remarks}

In this section, we consider the perturbation bound for the stationary
distribution vector $\vc{\pi}$ of $\vc{Q}$.  Let
$\vc{Q}^{\ast}=(q^{\ast}(k,i;\ell,j))_{(k,i;\ell,j)\in\bbF^2}$ denote
the infinitesimal generator of an ergodic Markov chain with state
space $\bbF$, and $\vc{\pi}^{\ast}=(\pi^{\ast}(k,i))_{(k,i)\in\bbF}$
denote the stationary distribution vector of
$\vc{Q}^{\ast}$. Furthermore, we introduce the $\vc{v}$-norm $\|\,
\cdot \,\|_{\vc{v}}$ for row vectors and matrices, where
$\vc{v}=(v(k,i))_{(k,i)\in\bbF}$ is a nonnegative $|\bbF| \times 1$
vector, as in the previous sections. For any row vector
$\vc{x}:=(x(k,i))_{(k,i) \in \bbF}$ and matrix
$\vc{Z}:=(z(k,i;\ell,j))_{(k,i;\ell,j)\in\bbF^2}$, let $\|
\vc{x}\|_{\vc{v}}$ and $\| \vc{Z}\|_{\vc{v}}$ denote
\[
\| \vc{x}\|_{\vc{v}}
= \sum_{(k,i)\in\bbF} |x(k,i)|\, v(k,i),
\qquad
\| \vc{Z}\|_{\vc{v}}
= \sup_{(k,i) \in \bbF} 
{ \sum_{(\ell,j)\in\bbF} |z(k,i;\ell,j)|\, v(\ell,j) \over v(k,i)},
\]
respectively. By definition, $|\vc{x}|\,\vc{v} = \|
\vc{x}\|_{\vc{v}}$.

We first present a perturbation bound under the exponential drift
condition.
\begin{thm}\label{thm-perturbation-exp}
Suppose that Assumption~\ref{basic-assumpt} is satisfied; and there
exist some $b>0$, $c > 0$, $K \in \bbZ_+$ and column vector $\vc{v}
\ge \vc{e}/c$ such that (\ref{geo-drift-cond}) holds. Furthermore, fix
$N \in \{K,K+1,\dots\}$ arbitrarily such that (\ref{ineqn-K}) holds;
and suppose that
\begin{equation}
\| \vc{Q}^{\ast} - \vc{Q}\|_{\vc{v}} < {1 \over C_{K,N}^{(\beta)}},
\label{comp-bound-Q^*-Q}
\end{equation}
where
\begin{equation}
C_{K,N}^{(\beta)}
= { b + 1 \over c}
\left(
1  + b + { 2bc \over \beta\overline{\phi}_{K,N}^{(\beta)} } 
\right).
\label{defn-C_{K,N}^{(beta)}}
\end{equation}
We then have
\begin{eqnarray}
\| \vc{\pi}^{\ast} - \vc{\pi} \|_{\vc{v}}
&\le& {b \over c} \cdot
{ C_{K,N}^{(\beta)} \| \vc{Q}^{\ast} - \vc{Q} \|_{\vc{v}}  
\over 
1 - C_{K,N}^{(\beta)} \| \vc{Q}^{\ast} - \vc{Q} \|_{\vc{v}} 
}.
\label{comp-perturbation-bound-01}
\end{eqnarray}
\end{thm}

\begin{rem}
As mentioned in Section~\ref{subsec-exp}, we can compute
$\overline{\phi}_{K,N}^{(\beta)}$ and thus
$C_{K,N}^{(\beta)}$. Therefore, the perturbation bound
(\ref{comp-perturbation-bound-01}) is computable, provided that
$\|\vc{Q}^{\ast} - \vc{Q} \|_{\vc{v}}$ is obtained.
\end{rem}

\begin{rem}\label{rem-C_K^{(beta)}}
It follows from (\ref{ineqn-phi}) and (\ref{defn-C_{K,N}^{(beta)}})
that $\{C_{K,N}^{(\beta)};N=K,K+1,\dots\}$ is decreasing and
\[
\lim_{N\to\infty} C_{K,N}^{(\beta)}
= { b + 1 \over c}
\left(
1  + b + { 2bc \over \beta\overline{\phi}_K^{(\beta)} } 
\right) =: C_K^{(\beta)}.
\]
Thus, as $N$ increases, the bound (\ref{comp-perturbation-bound-01})
becomes tighter. In addition, if the conditions of
Theorem~\ref{thm-perturbation-exp} are satisfied and $\| \vc{Q}^{\ast}
- \vc{Q}\|_{\vc{v}} < 1 / C_K^{(\beta)}$, then
\begin{eqnarray}
\| \vc{\pi}^{\ast} - \vc{\pi} \|_{\vc{v}}
&\le& {b \over c} \cdot
{ C_K^{(\beta)} \| \vc{Q}^{\ast} - \vc{Q} \|_{\vc{v}}  
\over 
1 - C_K^{(\beta)} \| \vc{Q}^{\ast} - \vc{Q} \|_{\vc{v}} 
}.
\label{perturbation-bound-01}
\end{eqnarray}
\end{rem}

\medskip
\noindent
{\it Proof of Theorem~\ref{thm-perturbation-exp}.~} Combining
Lemma~\ref{lem-bound-|D|g} with $\vc{f}=c\vc{v} \ge \vc{e}$ and
$\vc{\pi}\vc{v} < b/c$ yields
\begin{eqnarray*}
|\vc{D}|\, \vc{v}
&\le& { c \vc{\pi} \vc{v} +1 \over c} 
\left[
\vc{v} 
+ \left(
\vc{\pi}\vc{v} 
+ { 2b \over \beta\overline{\phi}_K^{(\beta)} } 
\right)(c\vc{v})
\right]
\le { b + 1 \over c}
\left(
1  + b + { 2bc \over \beta\overline{\phi}_K^{(\beta)} } 
\right) \vc{v}.
\end{eqnarray*}
Furthermore, applying (\ref{ineqn-phi}) to the above inequality leads to
\begin{eqnarray*}
|\vc{D}|\, \vc{v}
&\le&  { b + 1 \over c}
\left(
1  + b + { 2bc \over \beta\overline{\phi}_{K,N}^{(\beta)} } 
\right) \vc{v}
= C_{K,N}^{(\beta)} \vc{v},
\end{eqnarray*}
which implies that
\begin{equation}
\| \vc{D} \|_{\vc{v}} \le C_{K,N}^{(\beta)}.
\label{bound-|D|v-02}
\end{equation}
From (\ref{comp-bound-Q^*-Q}) and (\ref{bound-|D|v-02}), we have
\begin{equation*}
\| (\vc{Q}^{\ast} - \vc{Q}) \vc{D} \|_{\vc{v}}
\le \| (\vc{Q}^{\ast} - \vc{Q})  \|_{\vc{v}}
\cdot  \| \vc{D} \|_{\vc{v}}
\le C_{K,N}^{(\beta)} \| (\vc{Q}^{\ast} - \vc{Q})  \|_{\vc{v}}
< 1.
\end{equation*}
Thus, it holds (see, e.g., \citet[Section 4.1]{Heid10}) that
\begin{eqnarray}
\vc{\pi}^{\ast} - \vc{\pi}
&=& \vc{\pi} \sum_{m=1}^{\infty} \{ (\vc{Q}^{\ast} - \vc{Q}) \vc{D} \}^m.
\label{Update-formula}
\end{eqnarray}
It follows from (\ref{bound-|D|v-02}) and
(\ref{Update-formula}) that
\begin{eqnarray*}
\| \vc{\pi}^{\ast} - \vc{\pi} \|_{\vc{v}}
&\le& \| \vc{\pi} \|_{\vc{v}} 
\sum_{m=1}^{\infty} 
\left\{ \|\vc{Q}^{\ast} - \vc{Q} \|_{\vc{v}} \cdot \| \vc{D} \|_{\vc{v}} \right\}^m
\nonumber
\\
&\le& \| \vc{\pi} \|_{\vc{v}} 
\sum_{m=1}^{\infty} 
\left\{ C_{K,N}^{(\beta)} \|\vc{Q}^{\ast} - \vc{Q} \|_{\vc{v}}   \right\}^m
\nonumber
\\
&\le& {b \over c}
\cdot { C_{K,N}^{(\beta)} \|\vc{Q}^{\ast} - \vc{Q} \|_{\vc{v}} 
\over 1 -  C_{K,N}^{(\beta)} \|\vc{Q}^{\ast} - \vc{Q} \|_{\vc{v}}
},
\end{eqnarray*}
where the last inequality holds because $\| \vc{\pi} \|_{\vc{v}} =
\vc{\pi} \vc{v} \le b/c$.  \QED

\begin{rem}\label{rem-review-perturbation}
\citet{Kart86a,Kart86b,Kart86c} considered discrete-time
infinite-state Markov chains with uniform ergodicity (or equivalently,
strong stability; see \citet[Theorem~B]{Kart86a}), and then derived
perturbation bounds of a type similar to the bound
(\ref{comp-perturbation-bound-01}):
\begin{equation}
\| \vc{\varpi}^{\ast} - \vc{\varpi} \|
\le C_1 \cdot
{ C_2 \| \vc{P}^{\ast} - \vc{P} \|
\over 
1 - C_2 \| \vc{P}^{\ast} - \vc{P} \|
},
\label{Kart-perturbated-bound}
\end{equation}
where $\|\cdot\|$ denotes an appropriate norm, and where $\vc{\varpi}$
and $\vc{\varpi}^{\ast}$ are the stationary distributions of the
original transition kernel $\vc{P}$ and a perturbated transition
kernel $\vc{P}^{\ast}$, respectively. \citet{Mouh10} established a
bound of the type (\ref{Kart-perturbated-bound}) by using the norm of
a residual matrix of the original transition probability matrix (see
Theorem 5 therein). However, the perturbation bounds in these previous
studies are not easy to compute because the parameters $C_1$ and $C_2$
depend on $\|\vc{\varpi}\|$.  As for continuous-time infinite-state
Markov chains, \citet{LiuYuan12} presented a perturbation bound that
is similar to the bound (\ref{comp-perturbation-bound-01}) and
independent of $\|\vc{\pi}\|_{\vc{v}}$, under such an exponential
drift condition as corresponds to the condition (\ref{geo-drift-cond})
with $\vc{1}_{\bbF_K}$ being replaced by $\vc{1}_{\{(k,i)\}}$,
together with the condition that the infinitesimal generator is
bounded. The boundedness of the infinitesimal generator is removed by
\citet{LiuYuan15}.
\end{rem}

Next we derive a perturbation bound under the general
$\vc{f}$-modulated drift condition. To this end, we use the reduction
to exponential ergodicity, as in Theorem~\ref{thm-reduction-EXP}.
Recall here that if Condition~\ref{assumpt-f-ergodic} holds then
$\widehat{\vc{Q}}=\vc{\Delta}_{\vc{v}/\vc{f}}\vc{Q}$ satisfies the
exponential drift condition (\ref{ineqn-hat{Q}v}), which leads to
(\ref{ineqn-wh{pi}*v}).  Note also that, for all sufficiently
large $N \in \{K,K+1,\dots\}$,
\begin{equation}
\left[\widehat{\vc{\Phi}}_{\bbF_N}^{(\beta)} \right]_{\bbF_K} > \vc{O},
\label{add-eqn-160403-02}
\end{equation} 
which is confirmed as in the argument leading to (\ref{ineqn-K}).  We
now fix $N \in \{K,K+1,\dots\}$ such that (\ref{add-eqn-160403-02})
holds.  We then define
$\widehat{\vc{\Phi}}_{\bbF_N}^{(\beta)}:=(\widehat{\phi}_{\bbF_N}^{(\beta)}(k,i;\ell,j))_{(k,i;\ell,j)
  \in \bbF^2}$ as
\[
\widehat{\vc{\Phi}}_{\bbF_N}^{(\beta)}
= (\vc{I} - \widehat{\vc{Q}}_{\bbF_N}/\beta)^{-1},
\]
where $\widehat{\vc{Q}}_{\bbF_N} =
\vc{\Delta}_{\vc{v}/\vc{f}}\vc{Q}_{\bbF_N}$. We also define $\wh{C}_{K,N}^{(\beta)}$ as
\begin{equation}
\wh{C}_{K,N}^{(\beta)}
= ( \wh{b} + 1 )
\left(
1  + \wh{b} 
+ { 2\wh{b} \over \beta\overline{\widehat{\phi}\,}{}_{K,N}^{(\beta)} } 
\right),
\label{defn-wh{C}_K^{(beta)}}
\end{equation}
where
\[
\overline{\widehat{\phi}\,}{}_{K,N}^{(\beta)}
= \sup_{(\ell,j) \in \bbF_N}\min_{(k,i)\in\bbF_K} 
\widehat{\phi}_{\bbF_N}^{(\beta)}(k,i;\ell,j) > 0.
\]
Since $\overline{\widehat{\phi}\,}{}_{K,N}^{(\beta)}$ corresponds to
$\overline{\phi}_{K,N}^{(\beta)}$ in (\ref{defn-phi_{K,N}}), the
former can be computed in a similar way to the computation of the
latter (see Remark~\ref{rem-Le-Boud}).

The following theorem presents a computable perturbation bound under
the general $\vc{f}$-modulated drift condition.
\begin{thm}\label{thm-perturbation-general}
Suppose that Assumption~\ref{basic-assumpt} and
Condition~\ref{assumpt-f-ergodic-02} are satisfied.  Furthermore, fix
$N \in \{K,K+1,\dots\}$ arbitrarily such that
(\ref{add-eqn-160403-02}) holds. If $\vc{\pi}\vc{v} < \infty$ and
\begin{eqnarray}
\| \vc{\Delta}_{\vc{v}/\vc{f}} (\vc{Q}^{\ast} - \vc{Q}) \|_{\vc{v}} 
&<& { 1 \over \wh{C}_{K,N}^{(\beta)} },
\label{comp-bound-wh{Q}^*-wh{Q}}
\end{eqnarray}
then
\begin{eqnarray}
\| \vc{\pi}^{\ast} - \vc{\pi} \|_{\vc{f}}
&\le& \overline{C}{}_{\vc{f}/\vc{v}}\left( 1 + \wh{b} \,\overline{C}{}_{\vc{f}/\vc{v}} \right)
\cdot
{ \wh{b} \wh{C}_{K,N}^{(\beta)} \| \vc{\Delta}_{\vc{v}/\vc{f}} (\vc{Q}^{\ast} - \vc{Q}) \|_{\vc{v}} 
\over 
1 - \wh{C}_{K,N}^{(\beta)} \| \vc{\Delta}_{\vc{v}/\vc{f}} (\vc{Q}^{\ast} - \vc{Q}) \|_{\vc{v}} 
}.
\label{comp-bound-|pi^*-pi|f}
\end{eqnarray}
\end{thm}

\medskip
\noindent
{\it Proof.~}
Let $\widehat{\vc{\pi}}^{\ast}$ and $\widehat{\vc{Q}}^{\ast}$ denote
\[
\widehat{\vc{\pi}}^{\ast}
= {\vc{\pi}^{\ast}\vc{\Delta}_{\vc{f}/\vc{v}} \over \vc{\pi}^{\ast}(\vc{f}/\vc{v}) },
\qquad
\widehat{\vc{Q}}^{\ast}=\vc{\Delta}_{\vc{v}/\vc{f}}\vc{Q}^{\ast},
\]
respectively, where $\widehat{\vc{\pi}}^{\ast}$ is the probability
vector such that $\widehat{\vc{\pi}}^{\ast}\widehat{\vc{Q}}^{\ast} =
\vc{0}$.  Proceeding as in the derivation of
(\ref{eqn-pi-hat{pi}-02}), we have
\[
\vc{\pi}^{\ast} - \vc{\pi}
=
{1 \over \widehat{\vc{\pi}}^{\ast} \left(\vc{v}/\vc{f}\right)}
\left[
( \widehat{\vc{\pi}}^{\ast} - \widehat{\vc{\pi}} )
+
( \widehat{\vc{\pi}} - \widehat{\vc{\pi}}^{\ast} )
\left(\vc{v}/\vc{f}\right) 
{ 
\widehat{\vc{\pi}}
\over 
\widehat{\vc{\pi}} \left(\vc{v}/\vc{f}\right) 
}
\right]\vc{\Delta}_{\vc{v}/\vc{f}}.
\]
Using this equation and (\ref{ineqn-v-v/f}), we obtain
\begin{eqnarray}
\|\vc{\pi}^{\ast} - \vc{\pi}\|_{\vc{f}}
&\le&
{1 \over \widehat{\vc{\pi}}^{\ast} \left(\vc{v}/\vc{f}\right)}
\left[
| \widehat{\vc{\pi}}^{\ast} - \widehat{\vc{\pi}} |
+
| \widehat{\vc{\pi}} - \widehat{\vc{\pi}}^{\ast} |
\left(\vc{v}/\vc{f}\right) 
{ 
\widehat{\vc{\pi}}
\over 
\widehat{\vc{\pi}} \left(\vc{v}/\vc{f}\right) 
}
\right]\vc{v}
\nonumber
\\
&\le&
{1 \over \widehat{\vc{\pi}}^{\ast} \left(\vc{v}/\vc{f}\right)}
\left[
| \widehat{\vc{\pi}}^{\ast} - \widehat{\vc{\pi}} |\vc{v}
+
| \widehat{\vc{\pi}} - \widehat{\vc{\pi}}^{\ast} |
\vc{v} \cdot
{ 
\widehat{\vc{\pi}}\vc{v}
\over 
\widehat{\vc{\pi}} \left(\vc{v}/\vc{f}\right) 
}
\right]
\nonumber
\\
&=& 
{ 1  \over \widehat{\vc{\pi}}^{\ast} \left(\vc{v}/\vc{f}\right)}
\left(
1 +  { \widehat{\vc{\pi}} \vc{v} \over \widehat{\vc{\pi}} \left(\vc{v}/\vc{f}\right) } 
\right)
\| \widehat{\vc{\pi}}^{\ast} - \widehat{\vc{\pi}} \|_{\vc{v}}
\nonumber
\\
&\le& 
\overline{C}_{\vc{f}/\vc{v}}
\left(
1 +   \widehat{\vc{\pi}} \vc{v} \cdot \overline{C}_{\vc{f}/\vc{v}}
\right)
\| \widehat{\vc{\pi}}^{\ast} - \widehat{\vc{\pi}} \|_{\vc{v}}
\nonumber
\\
&\le&
\overline{C}_{\vc{f}/\vc{v}} 
\left(
1 + \widehat{b}\,\overline{C}_{\vc{f}/\vc{v}} 
\right) \| \widehat{\vc{\pi}}^{\ast} - \widehat{\vc{\pi}} \|_{\vc{v}},
\label{ineqn-pi-pi^*-02}
\end{eqnarray}
where the last inequality follows from (\ref{ineqn-wh{pi}*v}).

It remains to estimate $\| \widehat{\vc{\pi}}^{\ast} -
\widehat{\vc{\pi}} \|_{\vc{v}}$.  From (\ref{comp-bound-wh{Q}^*-wh{Q}}),
$\widehat{\vc{Q}}=\vc{\Delta}_{\vc{v}/\vc{f}}\vc{Q}$ and
$\widehat{\vc{Q}}^{\ast}=\vc{\Delta}_{\vc{v}/\vc{f}}\vc{Q}^{\ast}$, we
have
\[
\| \wh{\vc{Q}}^{\ast} - \wh{\vc{Q}} \|_{\vc{v}} 
= \| \vc{\Delta}_{\vc{v}/\vc{f}} (\vc{Q}^{\ast} - \vc{Q}) \|_{\vc{v}} 
< { 1 \over \wh{C}_{K,N}^{(\beta)} }.
\]
Thus, applying Theorem~\ref{thm-perturbation-exp} to $\wh{\vc{Q}}$
satisfying (\ref{ineqn-hat{Q}v}), we obtain
\begin{equation}
\left\| \wh{\vc{\pi}}^{\ast} - \wh{\vc{\pi}} \right\|_{\vc{v}}
\le \wh{b}
{ \wh{C}_{K,N}^{(\beta)} \| \wh{\vc{Q}}^{\ast} - \wh{\vc{Q}} \|_{\vc{v}} 
\over 
1 - \wh{C}_{K,N}^{(\beta)} \|\wh{\vc{Q}}^{\ast} - \wh{\vc{Q}} \|_{\vc{v}} 
}
= { \wh{b}\, \wh{C}_{K,N}^{(\beta)} 
\| \vc{\Delta}_{\vc{v}/\vc{f}} (\vc{Q}^{\ast} - \vc{Q}) \|_{\vc{v}} 
\over 
1 - \wh{C}_{K,N}^{(\beta)} \| \vc{\Delta}_{\vc{v}/\vc{f}} 
(\vc{Q}^{\ast} - \vc{Q}) \|_{\vc{v}} 
}.
\label{ineqn-wh{pi}^*-wh{pi}}
\end{equation}
Substituting (\ref{ineqn-wh{pi}^*-wh{pi}}) into (\ref{ineqn-pi-pi^*-02})
results in (\ref{comp-bound-|pi^*-pi|f}).  \QED

\medskip

\begin{rem}
A similar remark to Remark~\ref{rem-C_K^{(beta)}} applies to the bound
(\ref{comp-bound-|pi^*-pi|f}).  To save space, we omit the details.
\end{rem}

\appendix

\section{Proof of Proposition~\ref{prop-{n}Q-communication}}\label{appen-proof-prop-{n}Q-communication}

We first prove statement (i).  From (\ref{defn-(n)_Q}), we have
\[
\presub{n}q(k,i;\ell,j) = 0,\qquad (k,i) \in \bbF_n,\ (\ell,j) \in \overline{\bbF}_n,
\]
which shows that the Markov chain
$\{(\presub{n}X(t),\presub{n}J(t))\}$ cannot move from $\bbF_n$ to
$\overline{\bbF}_n$. Thus, $\bbF_n$ is closed and therefore includes
at least one closed communicating class.

We now denote by $\bbC$ a closed communicating class in $\bbF_n$. We
then assume that $\bbC \cap \bbL_n = \emptyset$, i.e., $\bbC \subseteq
\bbF_{n-1}$. In this setting, the submatrix
$\presub{n}\vc{Q}_{\bbC}:=(\presub{n}q(k,i;\ell,j))_{(k,i;\ell,j)\in\bbC^2}$
of $\presub{n}\vc{Q}$ is a conservative $q$-matrix. Furthermore, it
follows from (\ref{defn-(n)_Q}) and $\bbC \subseteq \bbF_{n-1}$ that
$\presub{n}\vc{Q}_{\bbC}$ is equal to the submatrix
$\vc{Q}_{\bbC}:=(q(k,i;\ell,j))_{(k,i;\ell,j)\in\bbC^2}$ of the
original generator $\vc{Q}$, i.e., $\presub{n}\vc{Q}_{\bbC}=
\vc{Q}_{\bbC}$. Therefore, $\vc{Q}_{\bbC}$ is a conservative
$q$-matrix, and $\bbC$ is a closed communicating class in the original
Markov chain $\{(X(t),J(t))\}$ with infinitesimal generator
$\vc{Q}$. This is, however, inconsistent with the irreducibility of
the Markov chain $\{(X(t),J(t))\}$. As a result, $\bbC \cap \bbL_n
\neq \emptyset$.

According to the above discussion, any closed communicating class in
$\bbF_n$ shares at least one element with $\bbL_n$. This implies that the
number of closed communicating classes in $\bbF_n$ is not greater than
the cardinality of $\bbL_n$, i.e., $S_1+1$. Consequently, statement (i) has  been proved.

Next we prove statement (ii). To this end, we assume that there exists
a closed communicating class $\bbC$ in $\overline{\bbF}_n$. Recall here
that the $|\overline{\bbF}_n| \times |\overline{\bbF}_n|$ southeast corner of $\presub{n}\vc{Q}$ is block-diagonal due to (\ref{defn-(n)_Q}). Thus, the closed communicating class $\bbC$ is
within a single level, i.e., $\bbC \subseteq \bbL_k$ for some $k \ge n+1$,
which implies that the $|\bbC| \times |\bbC|$ submatrix of
$\presub{n}\vc{Q}(k;k)=\vc{Q}(k;k)$ is a conservative
$q$-matrix. Therefore, the original Markov chain $\{(X(t),J(t))\}$
with infinitesimal generator $\vc{Q}$ cannot move out of $\bbC \subseteq
\bbL_k$. This contradicts the irreducibility of the Markov chain
$\{(X(t),J(t))\}$. Therefore, there are no closed communicating
classes in $\overline{\bbF}_n$.

\section{Applications of Dynkin's Formula}

In this appendix, we present two applications of Dynkin's formula
(see, e.g., \citet{Meyn93-III}).  For convenience, we redefine some of the symbols used in the body of the paper, in a different way.

We define $\{Y(t);t\ge 0\}$ as an irreducible regular-jump Markov chain
with state space $\bbZ_+$ and infinitesimal generator
$\vc{Q}:=(q(i,j))_{i,j\in\bbZ_+}$. For any $m \in \bbN$, we also define
$\{Y_m(t);t\ge 0\}$ as a stochastic process such that
\begin{equation}
Y_m(t)
=
\left\{
\begin{array}{ll}
Y(t), &  t < \tau_m,
\\
Y(\tau_m), &  t \ge \tau_m,
\end{array}
\right.
\label{defn-Phi_m(t)}
\end{equation}
where $\tau_m = \inf\{t \ge 0: Y(t) \ge m\}$. Since $\tau_m$ is a
stopping time for the Markov chain $\{Y(t)\}$, the stochastic process
$\{Y_m(t)\}$ is also a Markov chain (see, e.g., \citet[Chapter~8,
  Theorem 4.1]{Brem99}).

For any $m \in \bbN$, let $\vc{Q}_m:=(q_m(i,j))_{i,j\in\bbZ_+}$ denote
the infinitesimal generator of $\{Y_m(t)\}$. It then follows from
(\ref {defn-Phi_m(t)}) that
\begin{equation}
q_m(i,j) = 
\left\{
\begin{array}{lll}
q(i,j), & i = 0,1,\dots,m-1, & j \in \bbZ_+,
\\ 
0,      & i = m,m+1,\dots, & j \in \bbZ_+.
\end{array}
\right.
\label{eqn-q_m(i,j)}
\end{equation}
Furthermore, since $\{Y(t)\}$ is non-explosive, so is $\{Y_m(t)\}$ and
thus
\begin{equation}
\PP_i \! \left( \lim_{m\to\infty}\tau_m = \infty \right) = 1
\quad \mbox{for all~} i \in \bbZ_+,
\label{lim-tau_m}
\end{equation}
where $\PP_i(\,\, \cdot \,\,)$ represents $\PP(~ \cdot \mid Y(0) =
i)$ or $\PP(~ \cdot \mid Y_m(0) = i)$. For later use, let
$\EE_i[\,\, \cdot \,\,]$ denote $\EE[~ \cdot \mid Y(0) = i]$ or
$\EE[~ \cdot \mid Y_m(0) = i]$.

Let $\widehat{\tau}_m = \min(m,\tau_m,\tau)$ for $m \in \bbN$, where
$\tau$ denotes an arbitrary stopping time for the Markov chain
$\{Y(t)\}$. It then follows from (\ref{defn-Phi_m(t)}) and Dynkin's
formula (see, e.g., \citet[Equation~(8)]{Meyn93-III}) that, for any
real-valued column vector $\vc{x}:=(x(i))_{i\in\bbZ_+}$,
\begin{eqnarray}
\EE_i \! \left[ x(Y(\widehat{\tau}_m)) \right]
&=& \EE_i \! \left[ x(Y_m(\widehat{\tau}_m)) \right]
\nonumber
\\
&=& x(i) + 
\EE_i \!\! \left[ \int_0^{\widehat{\tau}_m} (\vc{Q}_m\vc{x})(Y(u)) \rmd u\right],
\qquad i=0,1,\dots,m-1,
\label{Dynkin-formula}
\end{eqnarray}
where $(\vc{Q}_m\vc{x})(i)$ is the $i$th element of the vector
$\vc{Q}_m\vc{x}$. Using (\ref{Dynkin-formula}), we obtain
Lemma~\ref{lem-compa} below, which is a continuous analogue of the
comparison Theorem for discrete-time Markov chains
(see \citet[Theorem~2.1]{Glyn96}).
\begin{lem}\label{lem-compa}
Suppose that $\{Y(t);t\ge0\}$ is an irreducible regular-jump Markov
chain. If there exist nonnegative column vectors
$\vc{v}:=(v(i))_{i\in\bbZ_+}$, $\vc{f}:=(f(i))_{i\in\bbZ_+}$ and
$\vc{w}:=(w(i))_{i\in\bbZ_+}$ such that
\begin{equation}
\vc{Q}\vc{v} \le  - \vc{f} + \vc{w},
\label{ineqn-Qv-02}
\end{equation}
then, for any $t \ge 0$ and stopping time $\tau$,
\begin{align}
&&&&
\EE_i \!\! \left[ \int_0^t f(Y(u)) \rmd u\right]
&\le v(i) + \EE_i \!\! \left[ \int_0^t w(Y(u)) \rmd u\right],
& i \in \bbZ_+, &&&&
\label{ineqn-compa-01}
\\
&&&&
\EE_i \!\! \left[ \int_0^{\tau} f(Y(u)) \rmd u\right]
&\le v(i) + \EE_i \!\! \left[ \int_0^{\tau} w(Y(u)) \rmd u\right],
& i \in \bbZ_+. &&&&
\label{ineqn-compa-02}
\end{align}
\end{lem}

\medskip
\noindent
{\it Proof.~}
It follows from (\ref{eqn-q_m(i,j)}) and (\ref{ineqn-Qv-02}) that, for
$m \in \bbN$,
\begin{align}
&&&&
(\vc{Q}_m\vc{v})(i) 
&\le  - f(i) + w(i), & i &= 0,1,\dots,m-1, &&&&
\label{ineqn-Q_mv}
\\
&&&&
(\vc{Q}_m\vc{v})(i) 
&= 0, & i &= m,m+1,\dots. &&&&
\label{eqn-Q_mv}
\end{align}
Substituting (\ref{ineqn-Q_mv}) and (\ref{eqn-Q_mv}) into
(\ref{Dynkin-formula}) with $\vc{x}=\vc{v}$ yields
\begin{eqnarray}
0 &\le&
\EE_i \! \left[ v(Y(\widehat{\tau}_m)) \right]
\nonumber
\\
&\le& v(i) 
+ \EE_i \!\! \left[ \int_0^{\widehat{\tau}_m} w(Y(u)) \rmd u\right]
- \EE_i \!\! \left[ \int_0^{\widehat{\tau}_m} f_m(Y(u)) \rmd u\right],
\qquad i \in \bbZ_+,
\label{ineqn-04}
\end{eqnarray}
where
\[
f_m(i)
=\left\{
\begin{array}{ll}
f(i), & i = 0,1,\dots,m-1,
\\
f(i) \vmin w(i), & i = m,m+1,\dots.
\end{array}
\right.
\]
Adding $\EE_i[ \int_0^{\widehat{\tau}_m} f_m(Y(u)) \rmd u]$
to both sides of (\ref{ineqn-04}), we obtain
\begin{eqnarray}
\EE_i \!\! \left[ \int_0^{\widehat{\tau}_m} f_m(Y(u)) \rmd u\right]
&\le& v(i) 
+ \EE_i \!\! \left[ \int_0^{\widehat{\tau}_m} w(Y(u)) \rmd u\right]
\nonumber
\\
&\le& v(i) 
+ \EE_i \!\! \left[ \int_0^{\tau} w(Y(u)) \rmd u\right],
\qquad i \in \bbZ_+,
\label{ineqn-150728-01}
\end{eqnarray}
where the second inequality follows from $\widehat{\tau}_m =
\min(m,\tau_m,\tau) \le \tau$. Note here that (\ref{lim-tau_m}) yields
$\PP_i(\lim_{m\to\infty}m \vmin\tau_m = \infty) = 1$ and thus
$\PP_i(\lim_{m\to\infty}\widehat{\tau}_m = \tau) = 1$. Therefore,
letting $m \to \infty$ in (\ref{ineqn-150728-01}) and using the
monotone convergence theorem, we have
(\ref{ineqn-compa-02}). Furthermore, replacing $\tau$ by $t$ and
proceeding as in the derivation of (\ref{ineqn-150728-01}), we obtain
\[
\EE_i \!\! \left[ \int_0^{t \vmin (m \vmin\tau_m)} f_m(Y(u)) \rmd u\right]
\le v(i) 
+ \EE_i \!\! \left[ \int_0^t w(Y(u)) \rmd u\right],
\qquad i \in \bbZ_+.
\]
Letting $m \to \infty$ in the above
inequality, we have (\ref{ineqn-compa-01}). 
\QED

\medskip

Next we discuss a Poisson equation associated with $\vc{Q}$. To this
end, we assume that the Markov chain $\{Y(t)\}$ is ergodic and has
the unique stationary distribution vector
$\vc{\pi}:=(\pi(i))_{i\in\bbZ_+}$. We then define
$\vc{g}^{\ddag}:=(g^{\ddag}(i))_{i\in\bbZ_+}$ as $\vc{g}^{\ddag} =
\vc{g} - (\vc{\pi}\vc{g})\vc{e}$, i.e.,
\[
g^{\ddag}(i) = g(i) - \vc{\pi}\vc{g},\qquad i \in \bbZ_+,
\]
where $\vc{g}:=(g(i))_{i\in\bbZ_+}$ is a given real-valued column
vector.  In this setting, we consider a Poisson equation:
\begin{equation}
-\vc{Q}\vc{h} = \vc{g}^{\ddag}.
\label{Poisson-eq}
\end{equation}
Using Lemma~\ref{lem-compa}, we prove the following result on a
solution of (\ref{Poisson-eq}).
\begin{lem}\label{lem-Poisson-eq}
Suppose that $\{Y(t);t\ge0\}$ is an irreducible regular-jump Markov
chain, and there exist some $b > 0$, $K \in \bbZ_+$, column vectors
$\vc{v} \ge \vc{0}$ and $\vc{f} \ge \vc{e}$ such that
\begin{equation*}
\vc{Q}\vc{v} \le  - \vc{f} + b \vc{1}_{\bbF_K}.
\label{ineqn-Qv-03}
\end{equation*}
For any fixed $j_{\ast} \in \bbZ_+$ and $|\vc{g}| \le
\vc{f}$, let
$\vc{h}_{j_{\ast}}:=(h_{j_{\ast}}(i))_{i\in\bbZ_+}$
denote
\begin{equation}
h_{j_{\ast}}(i)
= \EE_i \!\! 
\left[ \int_0^{\tau(j_{\ast})} g^{\ddag}(Y(t))\rmd t \right],
\qquad i \in \bbZ_+,
\label{defn-widetilde{h}-02}
\end{equation}
where $\tau(j_{\ast}) = \inf\{t \ge 0:Y(t)=j_{\ast}\}$. Under these
conditions, the vector $\vc{h}_{j_{\ast}}$ is a solution of the
Poisson equation (\ref{Poisson-eq}). In addition,
$h_{j_{\ast}}(j_{\ast}) = 0$.
\end{lem}

\medskip
\noindent
{\it Proof.~}
According to Theorem 7 of \citet{Meyn93-Proc}, 
the Markov chain
$\{Y(t)\}$ is ergodic under the conditions of this lemma.  It
follows from Lemma~\ref{lem-compa} with $\tau=\tau(j_{\ast})$ and
$\vc{w}=\vc{1}_{\bbF_K}$ that
\begin{eqnarray*}
\EE_i \!\! \left[ \int_0^{\tau(j_{\ast})} |g(Y(u))| \rmd u\right]
&\le& \EE_i \!\! \left[ \int_0^{\tau(j_{\ast})} f(Y(u)) \rmd u\right]
\nonumber
\\
&\le& v(i) + 
\EE_i \!\! \left[ \int_0^{\tau(j_{\ast})} 1_{\bbF_K}(Y(u)) \rmd u\right]
\nonumber
\\
&\le& v(i) + \EE_i[\tau(j_{\ast})] < \infty, \qquad i,j\in\bbZ_+,
\end{eqnarray*}
where the last inequality is due to the ergodicity of the Markov chain
$\{Y(t)\}$. Therefore, $\vc{h}_{j_{\ast}}$ is
well-defined. Furthermore, given $Y(0) = j_{\ast}$, we have
$\tau(j_{\ast}) = 0$ and thus $h_{j_{\ast}}(j_{\ast}) = 0$.

In what follows, we confirm that $\vc{h}_{j_{\ast}}$ is a solution of
(\ref{Poisson-eq}). For this purpose, we consider the embedded Markov
chain $\{\widetilde{Y}_n:=Y(t_n);n\in\bbZ_+\}$ of the Markov
chain $\{Y(t);t\ge0\}$ (see, e.g., \citet[Chapter~8,
  Section~4.2]{Brem99}), where $\{t_n;n\in\bbZ_+\}$ denotes a sequence
of time points such that $t_0 = 0$ and
\[
t_n = \inf\{t > t_{n-1}: Y(t) \neq Y(t_{n-1}) \},\qquad n \in \bbN.
\]
The transition probability matrix of $\{\widetilde{Y}_n\}$, denoted
by $\widetilde{\vc{P}}:=(\widetilde{p}(i,j))_{i,j\in\bbZ_+}$, is given
by
\begin{equation}
\widetilde{p}(i,j) =
\left\{
\begin{array}{cc}
0, &  j=i,
\\
\dm{q(i,j) \over |q(i,i)|}, & j \neq i.
\end{array}
\right.
\label{defn-tilde{P}}
\end{equation}
We also define $\widetilde{\tau}(j) = \inf\{n\in\bbZ_+:
\widetilde{Y}_n = j\}$ for $j \in \bbZ_+$ and $\Delta t_n = t_n -
t_{n-1}$ for $n \in \bbN$.  It then follows from
(\ref{defn-widetilde{h}-02}) that
\begin{eqnarray}
h_{j_{\ast}}(i)
&=& \EE_i \!\! 
\left[ \sum_{n=0}^{\widetilde{\tau}(j_{\ast})-1} 
\Delta t_{n+1} g^{\ddag}(\widetilde{Y}_n)\right]
\nonumber
\\
&=&  \sum_{n=0}^{\infty} 
\EE_i
[ 
\Delta t_{n+1} g^{\ddag}(\widetilde{Y}_n) I(n < \widetilde{\tau}(j_{\ast}))
]
\nonumber
\\
&=&  \sum_{n=0}^{\infty} \sum_{\nu \in \bbZ_+} g^{\ddag}(\nu) 
\EE_i
[  \Delta t_{n+1} 
I(n < \widetilde{\tau}(j_{\ast}))
I(\widetilde{Y}_n=\nu) ]
\nonumber
\\
&=&  \sum_{n=0}^{\infty} \sum_{\nu \in \bbZ_+} g^{\ddag}(\nu) 
\EE_i
[  \Delta t_{n+1} \mid n < \widetilde{\tau}(j_{\ast}), 
\widetilde{Y}_n=\nu ]
\cdot \EE_i
[ 
I(n < \widetilde{\tau}(j_{\ast}))
I(\widetilde{Y}_n=\nu) ], \qquad
\label{eqn-widetilde{h}-03}
\end{eqnarray}
where $I(\,\cdot\,)$ denotes the indicator function of the event in
the brackets.  Since $\widetilde{\tau}(j_{\ast})$ is a stopping time
for $\{\widetilde{Y}_n\}$, the event $\{n <
\widetilde{\tau}(j_{\ast})\}$ is determined by the set
$\{\widetilde{Y}_m;m=0,1,\dots,m\} = \{Y(t_m);m=0,1,\dots,m\}$. Thus,
given that $\widetilde{Y}_n=Y(t_n)=\nu$, the random variable $\Delta
t_{n+1}=t_{n+1} - t_n$ is independent of the event $\{n <
\widetilde{\tau}(j_{\ast})\}$, which leads to
\begin{equation}
\EE_i
[ \Delta t_{n+1} \mid n < \widetilde{\tau}(j_{\ast}), \widetilde{Y}_n=\nu
]
=
\EE[\Delta t_{n+1} \mid \widetilde{Y}_n = \nu] 
= {1 \over |q(\nu,\nu)|},
\qquad \nu \in \bbZ_+.
\label{eqn-E[s_n]}
\end{equation}
Substituting (\ref{eqn-E[s_n]}) into (\ref{eqn-widetilde{h}-03})
yields
\begin{eqnarray}
h_{j_{\ast}}(i)
&=&  \sum_{n=0}^{\infty} \sum_{\nu \in \bbZ_+} 
{g^{\ddag}(\nu) \over |q(\nu,\nu)|}
\EE_i
[ I(n < \widetilde{\tau}(j_{\ast}))
I(\widetilde{Y}_n=\nu) ]
\nonumber
\\
&=&  \EE_i \!\! \left[ 
\sum_{n=0}^{\widetilde{\tau}(j_{\ast})-1} \sum_{\nu \in \bbZ_+} 
{g^{\ddag}(\nu) \over |q(\nu,\nu)|}
I(\widetilde{Y}_n=\nu)
\right]
=  \EE_i \!\! \left[ 
\sum_{n=0}^{\widetilde{\tau}(j_{\ast})-1} 
\widetilde{g}(\widetilde{Y}_n) 
\right],
\label{eqn-widetilde{h}-04}
\end{eqnarray}
where $\widetilde{g}(\nu) = g^{\ddag}(\nu) / |q(\nu,\nu)|$ for
$\nu\in\bbZ_+$. From (\ref{eqn-widetilde{h}-04}), $\widetilde{p}(i,i) = 0$
and the Markov property of
$\{\widetilde{Y}_n\}$, we have
\begin{eqnarray}
h_{j_{\ast}}(i)
&=& \widetilde{g}(i) 
+ \EE_i \!\! \left[ 
\sum_{n=1}^{\widetilde{\tau}(j_{\ast})-1} 
\widetilde{g}(\widetilde{Y}_n) \cdot I(\widetilde{\tau}(j_{\ast}) \ge 2)
\right]
\nonumber
\\
&=& \widetilde{g}(i) 
+ \sum_{\nu\in\bbZ_+ \setminus\{i,j_{\ast}\}} \widetilde{p}(i,\nu)
\EE \!\! \left[ 
\sum_{n=1}^{\widetilde{\tau}(j_{\ast})-1} 
\widetilde{g}(\widetilde{Y}_n) \cdot I(\widetilde{\tau}(j_{\ast}) \ge 2)
\mid \widetilde{Y}_1 = \nu
\right]
\nonumber
\\
&=& \widetilde{g}(i) 
+ \sum_{\nu\in\bbZ_+ \setminus\{i,j_{\ast}\}} \widetilde{p}(i,\nu)
h_{j_{\ast}}(\nu), \qquad i \in \bbZ_+.
\label{add-160921-01}
\end{eqnarray}
Combining (\ref{add-160921-01}) with $\widetilde{g}(i) = g^{\ddag}(i)
/ |q(i,i)|$, $h_{j_{\ast}}(j_{\ast}) = 0$ and (\ref{defn-tilde{P}})
leads to
\begin{eqnarray*}
h_{j_{\ast}}(i)
&=& {g^{\ddag}(i) \over |q(i,i)|}
+ \sum_{\nu\in\bbZ_+ \setminus\{i\}} {q(i,\nu) \over |q(i,i)|}
h_{j_{\ast}}(\nu),
\qquad i \in \bbZ_+.
\end{eqnarray*}
Multiplying
both sides of the above equation by $|q(i,i)|$ results in
\[
- \sum_{\nu\in\bbZ_+} q(i,\nu) h_{j_{\ast}}(\nu)
= g^{\ddag}(i), \qquad i \in \bbZ_+,
\]
which shows that (\ref{Poisson-eq}) holds. 
\QED


\section*{Acknowledgments}
The author thanks Mr.\ Yosuke Katsumata for performing the numerical
calculations in Section~4.2.3 and for pointing out some typos in an
earlier version of this paper. The author also thanks Dr.\ Tetsuya
Takine for sharing his paper \cite{Taki16} prior to its publication. In addition, the author deeply appreciates the anonymous Reviewer B's comments and suggestions that helped the author to correct some errors in the previous versions of the proof of Lemmas~\ref{lem-unique-h} and \ref{lem-bound-h}. This research was supported in part by JSPS KAKENHI Grant Number JP15K00034.


%
%
%
%
\bibliographystyle{plain} 


\newpage
\setcounter{page}{1}
\pagestyle{myheadings} 
\markboth{\small H. Masuyama}
{Corrigendum to JORSJ, {\bf 60} (2017), 271--320}

\makeatother

\hfill

{\large{\bf
\begin{center}
Corrigendum:\\
``ERROR BOUNDS FOR LAST-COLUMN-BLOCK-AUGMENTED
TRUNCATIONS OF BLOCK-STRUCTURED MARKOV CHAINS"\\
Vol.~60, No.~3, 2017, pp.~271--320%
%
%
%
%
\end{center}
}
}

\smallskip

\begin{center}
{
Hiroyuki Masuyama%
\footnote[2]{E-mail: masuyama@tmu.ac.jp\\
Graduate School of Management, Tokyo Metropolitan University, 
Tokyo 192--0397, Japan}
}

\smallskip

{\small
{\it Tokyo Metropolitan University}
}
\end{center}

\bigskip
\makeatletter
\renewcommand{\theequation}{%
\arabic{equation}}
\makeatother
\setcounter{equation}{0}

Section~2.2 of Masuyama~\cite{Masu17} presents a computable and
nontrivial lower bound $\ol{\phi}_{K,N}^{(\beta)}$ for the factor
$\ol{\phi}_{K}^{(\beta)}$ of the error bounds given in Theorems
2.1, 2.2 and 2.4. The author stated that the lower bound
$\ol{\phi}_{K,N}^{(\beta)}$ exists because (see
\cite[Eq.~(2.66)]{Masu17})
\begin{equation}
\lim_{N\to\infty} \uparrow 
\ol{\phi}_{K,N}^{(\beta)} = \ol{\phi}_K^{(\beta)},
\label{eqn-(2.66)}
\end{equation}
where the symbol $\,\uparrow\,$ represents {\it ``convergence from
  below"}.  However, the proof of (\ref{eqn-(2.66)}), presented in
\cite{Masu17}, is not complete.  Thus, this corrigendum presents a
complete proof of (\ref{eqn-(2.66)}).

It follows from \cite[Section~2.2, Proposition~2.14]{Ande91} that, for
all $t \ge 0$ and $(k,i;\ell,j) \in \bbF^2$,
\[
\lim_{N\to\infty} \uparrow 
\left[ \exp\{ \vc{Q}_{\bbF_N} t\} \right]_{(k,i;\ell,j)}
=
p^{(t)}(k,i;\ell,j),
\]
where $\left[ \exp\{ \vc{Q}_{\bbF_N} t\} \right]_{(k,i;\ell,j)}$
denotes the $(k,i;\ell,j)$th element of $\exp\{ \vc{Q}_{\bbF_N}
t\}$. Therefore, by the monotone convergence theorem, we have, for all
$(k,i;\ell,j) \in \bbF^2$,
\begin{eqnarray}
\lim_{N\to\infty} \uparrow \int_0^{\infty} \beta \rme^{-\beta t}
\left[ \exp\{ \vc{Q}_{\bbF_N} t\} \right]_{(k,i;\ell,j)} \rmd t
=
\int_0^{\infty} \beta \rme^{-\beta t} p^{(t)}(k,i;\ell,j) \rmd t>0.
\label{add-eqn-17101201}
\end{eqnarray}
Using \cite[Eqs.~(2.3) and (2.59)]{Masu17}, we rewrite
(\ref{add-eqn-17101201}) as
\begin{eqnarray}
\lim_{N\to\infty} \uparrow \phi_{\bbF_N}^{(\beta)}(k,i;\ell,j)
=
\phi^{(\beta)}(k,i;\ell,j)
> 0, \qquad \forall (k,i;\ell,j) \in \bbF^2.
\label{eqn-00}
\end{eqnarray}
Although $\phi_{\bbF_N}^{(\beta)}(k,i;\ell,j)$ is defined for
$(k,i;\ell,j) \in (\bbF_N)^2$ (see \cite[Eq.~(2.59)]{Masu17}), we set
\begin{eqnarray}
\phi_{\bbF_N}^{(\beta)}(k,i;\ell,j) = 0, \qquad 
\mbox{$(k,i) \in \bbF\setminus\bbF_N$ or $(\ell,j) \in \bbF\setminus\bbF_N$}.
\label{eqn-01}
\end{eqnarray}
It then follows from (\ref{eqn-00}) and \cite[Eq.~(2.65)]{Masu17} that
$\{\overline{\phi}_{K,N}^{(\beta)};N=K,K+1,\dots\}$ is nondecreasing
and thus
\begin{eqnarray}
\lim_{N\to\infty} \overline{\phi}_{K,N}^{(\beta)}
&=& \sup_{N\ge K} \overline{\phi}_{K,N}^{(\beta)}
\nonumber
\\
&=& \sup_{N\ge K}
\sup_{(\ell,j)\in\bbF_N} \min_{(k,i)\in\bbF_K}
\phi_{\bbF_N}^{(\beta)}(k,i;\ell,j)
\nonumber
\\
&=& \sup_{N\ge K}
\sup_{(\ell,j)\in\bbF} \min_{(k,i)\in\bbF_K}
\phi_{\bbF_N}^{(\beta)}(k,i;\ell,j),
\label{eqn-02}
\end{eqnarray}
where the last equality holds due to (\ref{eqn-01}).  Note here that
the order of double supremum is interchangeable (see the lemma below),
i.e.,
\begin{eqnarray}
\sup_{N\ge K}
\sup_{(\ell,j)\in\bbF} \min_{(k,i)\in\bbF_K}
\phi_{\bbF_N}^{(\beta)}(k,i;\ell,j)
= \sup_{(\ell,j)\in\bbF} 
\sup_{N\ge K}
\min_{(k,i)\in\bbF_K}
\phi_{\bbF_N}^{(\beta)}(k,i;\ell,j).
\label{eqn-03}
\end{eqnarray}
Substituting (\ref{eqn-03}) into (\ref{eqn-02}), and using
(\ref{eqn-00}), we obtain
\begin{eqnarray*}
\lim_{N\to\infty} \overline{\phi}_{K,N}^{(\beta)}
&=& \sup_{(\ell,j)\in\bbF} \sup_{N\ge K}
\min_{(k,i)\in\bbF_K}
\phi_{\bbF_N}^{(\beta)}(k,i;\ell,j)
\nonumber
\\
&=& \sup_{(\ell,j)\in\bbF} \lim_{N\to\infty}
\min_{(k,i)\in\bbF_K}
\phi_{\bbF_N}^{(\beta)}(k,i;\ell,j)
\nonumber
\\
&=& 
\sup_{(\ell,j)\in\bbF} \min_{(k,i)\in\bbF_K} \lim_{N\to\infty} 
\phi_{\bbF_N}^{(\beta)}(k,i;\ell,j)
\nonumber
\\
&=& 
\sup_{(\ell,j)\in\bbF} \min_{(k,i)\in\bbF_K} \phi^{(\beta)}(k,i;\ell,j)
\nonumber
\\
&=& \ol{\phi}_{K}^{(\beta)},
\end{eqnarray*}
where the last equality follows from \cite[Eq.~(2.10)]{Masu17}. As a
result, we have proved that (\ref{eqn-(2.66)}) holds.

We close this corrigendum by providing the lemma, which enables us to
interchange the order of double supremum.
\begin{lemnonum}[Interchanging the Order of Double Supremum]
Let $\{a_{n,m};n,m\in\bbN\}$ denote a sequence of real numbers, where $\bbN=\{1,2,3,\dots\}$. We
then have
\begin{eqnarray*}
\sup_{(n,m) \in \bbN^2}a_{n,m}
= \sup_{n\in\bbN}\sup_{m\in\bbN}a_{n,m}
= \sup_{m\in\bbN}\sup_{n\in\bbN}a_{n,m}.
\end{eqnarray*}
\end{lemnonum}
\begin{proof}
By symmetry, it suffices to prove that 
\begin{equation}
\sup_{(n,m) \in \bbN^2}a_{n,m}
= \sup_{n\in\bbN}\sup_{m\in\bbN}a_{n,m}.
\label{eqn-04}
\end{equation}
If
\[
\sup_{(n,m) \in \bbN^2}a_{n,m}
> \sup_{n\in\bbN}\sup_{m\in\bbN}a_{n,m},
\]
then, for some $(n',m') \in \bbN^2$, we have $a_{n',m'} >
\sup_{n\in\bbN}\sup_{m\in\bbN}a_{n,m}$ whereas, by definition,
$a_{n',m'} \le \sup_{m\in\bbN}a_{n',m} \le
\sup_{n\in\bbN}\sup_{m\in\bbN}a_{n,m}$, which yields a contradiction.
On the other hand, if
\[
\sup_{(n,m) \in \bbN^2}a_{n,m}
< \sup_{n\in\bbN}\sup_{m\in\bbN}a_{n,m},
\]
then
\begin{eqnarray*}
\sup_{i\in\bbN}\sup_{j\in\bbN}a_{i,j}
&\le& \sup_{i\in\bbN}\sup_{j\in\bbN} \sup_{(n,m) \in \bbN^2}a_{n,m} 
\nonumber
\\
&=& \sup_{(n,m) \in \bbN^2}a_{n,m} 
<  \sup_{n\in\bbN}\sup_{m\in\bbN}a_{n,m},
\end{eqnarray*}
which also yields a contradiction. Consequently, (\ref{eqn-04}) holds.
\end{proof}



%
%
%
%
\bibliographystyle{plain} 

\end{document}